\documentclass{ws-ijgmmp}

\usepackage{xcolor}
\usepackage[verbose,hypertexnames=false]{hyperref}
\hypersetup{colorlinks=false,allbordercolors=blue,pdfborderstyle={/S/U/W 1}}
\usepackage{cancel}

\begin{document}

\markboth{Rattanasak Hama, Tiberiu Harko and Sorin V. Sabau}
{Instructions for Typing Manuscripts (From  dual connections to gravitational applications $\ldots$)}

%
\catchline{}{}{}{}{}
%

\title{From  dual connections to gravitational field equations - the curvature and Einstein tensors of the $\alpha$ - connection of a quasi-statistical manifold
}

\author{Rattanasak Hama}

\address{Faculty of Science and Industrial Technology,\\
Prince of Songkla University, Surat Thani Campus,\\
Surat Thani, 84000, Thailand\\
\email{rattanasak.h@psu.ac.th}
}

\author{Tiberiu Harko}

\address{Department of Physics, Babe\c s -Bolyai University, Kog\u alniceanu Street 1,\\
400084 Cluj-Napoca, Romania\\
Astronomical Observatory, 19 Cire\c silor Street, \\
400487 Cluj-Napoca, Romania\\
\email{tiberiu.harko@aira.astro.ro}
}

\author{Sorin V. Sabau}

\address{School of Biological Sciences, Department of Biology,\\
Tokai University, Sapporo 005-8600, Japan\\
Graduate School of Science and Technology, Physical and Mathematical Sciences,\\
Tokai University, Sapporo 005-8600, Japan\\
\email{sorin@tokai.ac.jp}
}
\maketitle

\begin{history}
\received{(Day Month Year)}
\revised{(Day Month Year)}
\accepted{(Day Month Year)}
\end{history}

\begin{abstract}
We present a detailed review of the mathematical foundations of the theory of the statistical and quasi-statistical manifolds, which recently have found many applications in general relativity, quantum mechanics, and mathematical statistics. In particular, we fully develop, in a rigorous and coherent way, the formulas and concepts necessary for the understanding of the mathematical basis of the statistical and quasi-statistical manifolds, including the properties of the dual connections and of the equiaffine connections. For each geometrical structure the explicit expressions of the curvatures and the Einstein tensors are explicitly obtained.  As possible applications to the field of gravitational theories  we explicitly compute the curvatures of a family of $\alpha$-connections {$\nabla^{(\alpha)}:=\frac{1+\alpha}{2}\nabla +\frac{1-\alpha}{2}\nabla ^{*}$, where $\nabla :=\nabla ^{(1)}$ and $\nabla ^{*}:=\nabla ^{(-1)}$ } are the dual connections of a quasi-statistical manifold $M$.  The Einstein vacuum field equations are also written down, and the physical relevance of the obtained results for gravity and cosmology is briefly discussed.
\end{abstract}


\ccode{Mathematics Subject Classification 2020: -}

\tableofcontents

\section{Introduction}

The information geometry is an important field of mathematics and applied sciences. In its initial formulation information geometry represents   a straightforward  application of the basic principles and results of differential geometry to statistics. The fundamental concept of information geometry is the notion of Statistical Manifold \cite{Amari}, which  was proposed initially by Lauritzen \cite{Laur}, and later systematically formalized by Kurose \cite{Kurose} and Noguchi \cite{Noguchi} as a general geometric theory developed in the framework of affine differential geometry. This geometric method is also called information geometry, and its applications in various fields of mathematical sciences, engineering and natural sciences has led to important advances in the corresponding fields  \cite{Caticha, Nielsen}.

One of the novel aspects introduced in the study of Statistical Manifolds is  the consideration of a geometry which contains, together with the metric, two affine connections which  are dual to each other \cite{Amari}. If the two connections are  torsionless, then a totally symmetric tensor, the cubic tensor does also exist \cite{Amari1}. 

The cubic tensor fully characterizes the two connections, and determines
their deviations from the Levi-Civita connection. 

In a more formal approach one can define a Statistical Manifold $(M,\nabla,h)$ as a (semi-)Riemannian manifold $(M,h)$, on which a torsionless affine connection $\nabla$ is defined,  and with $\nabla h$ totally symmetric \cite{Matsuoze}. For a Statistical Manifold a pair of mutually dual affine connections can be naturally introduced.  

If  $\nabla ^{*}$ is the dual connection of $\nabla $ with respect to $h$,then  it turns out that the triplet $(M,\nabla ^{*},h)$ is also a statistical manifold. In this case $(M,\nabla ^{*},h)$ is called the dual statistical manifold of $(M,\nabla,h)$. 

Essentially,  in their initial formulation, Statistical Manifolds represented an application of differential geometry for the study of stochastic processes. We will now briefly review, following the approach presented in \cite{Obata1} and \cite{Obata}, respectively, the main steps involved in the geometrization of statistical processes. 

Let's assume that on a differential manifold $S$ a function $f$ and a tangent vector $A$ are defined. Then, according to the standard definitions in differential geometry,  the directional derivative $\nabla _Af$ of $f$ toward $A$ is defined as $A (f)$.  The linearity and the Leibnitz condition for functions, as well as for vector fields, is satisfied by the derivative. 

In a statistical manifold $S$, the natural derivative of the vector field $8$ toward the tangent vector $A$, $\nabla _{A}B$ is introduced through the inner product $g\left(\nabla _{A}B,C\right)$, where
$ C$ is an arbitrary tangent vector. The function $g(\nabla _A B,C)$ is uniquely defined according to \cite{Obata}
$
g(\nabla _A B,C)=E\left[\left(ABl+\frac{1-\alpha}{2}AlBl\right)Cl\right]$,
where $x$ is a random variable, $l(x,\theta)=\ln p(x,\theta)$, and $p (x, \theta)$ is a probability density function, parameterized by the $n$-dimensional parameter $\theta=(\theta ^1, . . . , \theta ^n)$. Moreover,  $\alpha$ is a free parameter indicating  that the function $g$ is independent on the parametrization $\theta$. It is important to note that $g$  is invariant under the set of transformations of the random variable $x$. 

 One can now introduce the $\alpha$-connection coefficients $\Gamma _{ijk}$, defined according to \cite{Obata1, Obata}
\begin{equation}
    \Gamma _{ijk}=g\left(\nabla _{\partial _k}\partial _j,\partial _i\right)=E\left[\left(\partial _j\partial _kl+\frac{1-\alpha}{2}\partial _jl\partial _kl\right)\partial _il\right],
\end{equation}
where $E[x]$ denotes the expectation value of the random variable $x$.   We denote in the following the quantities corresponding to a given value of $\alpha$ by a superscript, as $\nabla ^{(\alpha)}$, $\Gamma _{ijk}^{(\alpha)}$ etc. For $\alpha =0$, the $\alpha$-connection coefficients take the limiting form of the Levi-Civita connection coefficients
\begin{equation}
\Gamma _{ijk}^{(0)}=\frac{1}{2}\left(\partial _kg_{ij}+\partial _jg_{ik}-\partial _ig_{jk}\right).
\end{equation}

Generally, the $\alpha$-connection can be represented as \cite{Obata1, Obata}
\begin{equation}
\Gamma _{ijk}=\Gamma _{ijk}^{(1)}+\frac{1-\alpha}{2}T_{ijk},
\end{equation}
where
\begin{equation}
T_{ijk}=E\left(\partial _il\partial _jl\partial _kl\right)=-E\left(\partial _i\partial_j\partial _k l\right)
-E\Bigg[\left(\partial _il\right)\left(\partial _j\partial _kl\right)+\left(\partial _jl\right)\left(\partial _k\partial _il\right)
+\left(\partial _kl\right)\left(\partial _i\partial _jl\right)\Bigg].
\end{equation}

It is a standard result in differential geometry that if three vector fields $A$, $B$, and $C$ are given,  $\nabla _A\nabla _B C$ does not commute with  $\nabla _B\nabla _AC$ \cite{Nakahara}. The noncommutativity degree of 
 is described by the curvature tensor $R$, a vector valued function $R(A,B,C)$, defined in the usual way according to
\begin{equation}
    R(A,B,C)=R(A,B)C\equiv \left[\nabla _A,\nabla _B\right]C-\nabla _{[A,B]}C.
\end{equation}
From its definition it immediately follows that the curvature tensor $R(A,B,C)$ is a trilinear function of $A$, $B$, and $C$.
If $R(A,B,C)\equiv 0$, the manifold $S$ is  flat. 

The components of the curvature tensor are obtained in coordinates according to the definition \cite{Nakahara}
\begin{equation}
    R\left(\partial _k, \partial _l, \partial _j \right)=R\left(\partial _k,\partial _l\right)\partial _j =\left[\nabla _{\partial _k},\nabla _{\partial _l}\right]\partial _j=R_{jkl}^i\partial _i,
\end{equation}
where
\begin{equation}
R_{jkl}^{i}=\partial _{k}\Gamma _{jl}^{i}-\partial _{l}\Gamma
_{jk}^{i}+\Gamma _{mk}^{i}\Gamma _{jl}^{m}-\Gamma _{ml}^{i}\Gamma _{jk}^{m}.
\end{equation}

Hence, within the framework of the theory of Statistical Manifolds, a full geometrization of the stochastic processes can be obtained \cite{Obata1, Obata}. Moreover, the geodesics, and the totally geodesic submanifolds of a Statistical Manifold, in the presence of an affine connection, are the natural extensions of the affine subspaces and straight lines in the Euclidean space \cite{Balan}. 

The concepts and methods inspired by the Statistical Manifold concept have found many applications, ranging from machine learning and data analysis \cite{Learning, Learning1} to thermodynamics \cite{Therm,Therm1,Therm2} and solitons \cite{Sol}. The concept of quantum statistical manifold was introduced in \cite{quantum}, by considering instead of a manifold of strictly positive density matrices  a manifold of faithful quantum states on the C$^{}$-algebra of bounded linear operators, under the assumption that the underlying Hilbert space is finite-dimensional. 

The definition of the Statistical Manifold as described above assumes that the affine connection is torsionless. However, to describe geometric structures on quantum state spaces,  Kurose \cite{Kurose1} considered statistical manifolds admitting torsion. This type of geometrical structures are called quasi-statistical manifolds. From a physical point of view,  for the description of quantum effects within the formalism of statistical manifolds, the inclusion of torsion may be necessary due to the non-commutativity of quantum mechanics.

The concept of quasi-semi-Weyl structure was introduced in \cite{Blaga}, and further investigated in \cite{Blaga1}. A symmetry similar to the mirror symmetry of string theory was considered in \cite{Zhang1}, by considering the parametrization of the statistical manifolds  as affine coordinates with respect to a flat connection $\nabla$, and considering its $g$-conjugate connection $\nabla ^{*}$, which is curvature-free but having torsion. A pseudo-Weitzenb\"{o}ck connection for the manifold of parametric statistical models was constructed in \cite{Zhang2}, which can be considered as a statistical manifold with torsion. For the investigation of other interesting properties of the $\alpha$-connections see \cite{Blaga2}. 

On a statistical manifold $\mathcal{M}$ with a pair of conjugate connections $\nabla  \equiv \nabla ^{(1)}$ and $\nabla ^{*} \equiv \nabla^{(-1)}$, the family of $\alpha$ connections $\nabla ^{(\alpha)}$  is given by $\nabla ^{(\alpha)}:=\frac{1+\alpha}{2}\nabla +\frac{1-\alpha}{2}\nabla ^{*}$. The expression of the curvature $R^{(\alpha)}$ for $\nabla ^{(\alpha)}$ in terms of the curvatures of  $\nabla$ and $\nabla ^{*}$ was obtained in \cite{Zhang3}. 

The concept of a torsion dual connection was developed in \cite{Iosifidis1}, where it was shown that for the torsion dual manifolds, flatness of one connection does not necessary implies the flatness of the other connection.  Moreover, in this case the curvature tensor of the latter is given by a specific divergence. A self-consistent definition of the mutual curvature tensor of the two connections was also provided, and the notion of a curvature dual connection was defined. 

The mutual curvature scalar as defined in \cite{Iosifidis1} was used to construct a biconnection gravitational theory in \cite{Iosifidis2}. The geometric framework of the theory consists of one metric and two affine connections, defined in a metric-affine gravity geometry. By coupling the two connections  with matter, it turns out that the geometry of the resulting theory is that of a statistical manifold. 

A length-preserving biconnection gravitational theory,  which extended general relativity by using the mutual curvature tensor as the fundamental object describing gravity was introduced in \cite{Csillag}. The two connections used to build up the theory are the Schr\"{o}dinger connection, and its dual. The dual of a non-metric Schr\"{o}dinger connection has torsion, even if the Schr\"{o}dinger connection  does not have. An important implications of this result is that the pair $(M,g,\nabla ^{*})$ is a quasi-statistical manifold. 

The gravitational field equations have been postulated as having the form of the standard Einstein equations, with the Ricci tensor and scalar replaced with the mutual curvature tensor and scalar, resulting in additional torsion-dependent terms. These new terms can be interpreted as describing an effective, geometric type dark energy. Two distinct cosmological models based on this theory were investigated,  one with conserved matter, and another one in which the effective dark energy and pressure terms  are related by a linear equation of state. The considered length-preserving biconnection gravity models fit well the observational data, and also reproduce well the predictions of the standard $\Lambda$CDM model at  redshifts $z<3$. 

Gravitational theories in the presence of biconnection geometric structures have also been considered in \cite{B1,B2,B3,B4,B5}. 

The results of \cite{Iosifidis1} and \cite{Csillag} did show that the generalization of the concept of the torsionless statistical manifold into the concept of quasi-statistical manifold with torsion may have important physical and cosmological implications. Hence, the detailed investigations of the mathematical properties of the quasi-statistical manifolds, involving the presence of torsion and nonmetricity, may have not only a mathematical relevance, but could also open some new perspectives in physics.  

It is the goal of the present paper to investigate and to obtain the full curvature expressions of the quasi-statistical manifolds, which involves the presence of the torsion. From the point of view of the possible physical applications the Einstein tensor plays a fundamental role. We obtain the expression of the Einstein tensor on quasi-statistical manifolds, which also includes the corrections coming from the torsion. Moreover, in order to increase the clarity and consistency of the present work, we also introduce and develop the basic differential geometric concepts (dual connection, equiaffine manifolds, statistical manifolds) that are necessary for the in depth understanding of the properties of Statistical Manifolds, and related concepts.  

The present paper is organized as follows. The basic definitions and properties of the dual connections are presented in Section~\ref{sect1}.
 The curvature tensor, its symmetry properties, the Bianchi identities, as well as the basic definition of the Einstein tensor and its fundamental properties are discussed in detail. The definitions and the basic properties of the equiaffine connections are introduced in Section~\ref{sect2}. The differential geometric properties of the Statistical Manifolds are reviewed in a systematic way in Section~\ref{sect3}. A detailed description of the geometric properties of the quasi-statistical manifolds is provided in Section~\ref{sect_Quasi_stat}. In particular a detailed discussion of the curvature properties and of the Einstein tensor is given.  The differential geometric properties of the $\alpha$-connections and   
the explicit calculation of the components of the curvature tensor for the $\alpha$-connection of the quasi-statistical manifold is presented in Section~\ref{sect_alpha}, and their Einstein tensors are explicitly obtained.  The divergence of the Einstein tensor is also considered, and its expression is explicitly calculated in the coordinate representation. Finally, we discuss and conclude our results in Section~\ref{sect6}.  The explicit details of the calculations of the geometric quantities, and the proof of some results are presented in the Appendix~\ref{App}. 

\section{Dual connections}\label{sect1}

We begin our presentation of the basic mathematical properties of the statistical and quasi-statistical manifolds with a detailed discussion of the concept of dual connection, and of its properties. Essentially, the dual connections represent the central novel mathematical element that allows the development of the mathematical concepts of statistical and quasi-statistical manifolds.    

\subsection{Definitions and basic properties}

Let $(M,g)$ be a (pseudo-)Riemannian manifold. Then it is known that there exists a unique linear connection on $M$ subject to the following conditions
\begin{equation}\label{eq_L1}
\widehat{\nabla} \text{ is } g\text{-metrical, i.e. }(\widehat{\nabla}_Xg)(Y,Z)=0,
\end{equation}
and
\begin{equation}\label{eq_L2}
\widehat{\nabla}\text{ is torsion free, i.e. }\widehat{T}(X,Y)=0,
\end{equation}
where $\widehat{T}(X,Y):=\widehat{\nabla}_XY-\widehat{\nabla}_YX-[X,Y]$. This is the well-known {\it Levi-Civita connection} of $(M,g)$.

Locally,  using $(x^i)$ as local coordinates on $M$ and $\{\frac{\partial}{\partial x^i}\}$ as canonical coordinates on $TM$, conditions above read
\begin{equation}\label{eq_L1'}
\widehat{C}_{kij}:=g_{ij\widehat{|}k}=\dfrac{\partial g_{ij}}{\partial x^k}-g_{mi}\widehat{\Gamma}^m_{\ jk}-g_{mj}\widehat{\Gamma}^m_{\ ik}=0,
\end{equation}
\begin{equation}\label{eq_L2'}
\widehat{T}^i_{\ kj}:=\widehat{\Gamma}^i_{\ jk}-\widehat{\Gamma}^i_{\ kj}=0,
\end{equation}
where $\widehat{\Gamma}^i_{\ jk}$ are the local coefficients of $\widehat{\nabla}$ given by
$$
\widehat{\nabla}_{\dfrac{\partial}{\partial x^j}}\frac{\partial}{\partial x^i}=\widehat{\Gamma}^k_{\ ij}\frac{\partial}{\partial x^k}.
$$

An equivalent formulation uses coframes. Indeed, the dual of the canonical frame $\{\frac{\partial}{\partial x^i}\}$ is the coframe $\{dx^i\}$.

Then, the Levi-Civita connection is given by the connection forms $\{\widehat{\omega}_i^{\ j}\}$ subject to the conditions
\begin{equation}\label{eq_L1''}
Q_{ij}:=dg_{ij}-g_{kj}\widehat{\omega}_i^{\ k}-g_{ik}\widehat{\omega}_j^{\ k}=0
\end{equation}
\begin{equation}\label{eq_L2''}
\Theta^i:=d(dx^i)-dx^j\wedge \widehat{\omega}_j^{\ i}=-dx^j\wedge \widehat{\omega}_j^{\ i}=0.
\end{equation}

Locally, we have
$$
\widehat{\omega}_j^{\ i}=\widehat{\Gamma}^i_{\ jk}dx^k.
$$

Moreover, we can consider a $g$-orthonormal frame $\{e_a\}$ in $TM$, with the dual coframe $\{\omega^a\}$. Then the Levi-Civita connection of $(M,g)$ has the connection forms $\{\widehat{\omega}_a^{\ b}\}$ subject to the conditions
\begin{equation}\label{eq_L1'''}
\widehat{\omega}_{ab}+\widehat{\omega}_{ba}=0.
\end{equation}
\begin{equation}\label{eq_L2'''}
d\omega^a-\omega^b\wedge\widehat{\omega}_b^{\ a}=0.
\end{equation}

Observe that there exists a pair of matrices $\{u_a^{\ i}\}$ and $\{v^a_{\ i}\}$ on $M$, inverse each other, such that
\begin{equation*}
\begin{split}
e_a=u_a^{\ i}\frac{\delta}{\delta x^i},\;
\omega^a=v^a_{\ i}dx^i.
\end{split}
\end{equation*}

Then
$$
\widehat{\omega}_b^{\ a}=(du_b^{\ i})v^a_{\ i}+u_b^{\ j}\widehat{\omega}_j^{\ i}v^a_{\ i},
$$
or, equivalently
$$
\widehat{\omega}_j^{\ i}=(dv^a_{\ j})u_a^{\ i}+v^b_{\ j}\widehat{\omega}_b^{\ a}u_a^{\ i}
$$
shows the relation between the connection forms in different frames. Here, indices $i,j,k,\dots$ indicate the canonical framing, while $a,b,c,\dots$ indicate the $g$-orthonormal frames.

Let us consider arbitrary affine connection $\nabla$ on the (pseudo-)Riemannian manifold  $(M,g)$ which is not metrical nor torsion free, i.e.
\begin{equation}\label{eq_A1}
C(X,Y,Z):=(\nabla_Xg)(Y,Z)=X(g(Y,Z))-g(\nabla_XY,Z)-g(Y,\nabla_XZ)\neq 0.
\end{equation}

Observe that 
\begin{equation}\label{C commut in last 2}
C(X,Y,Z)=C(X,Z,Y).
\end{equation}

Locally, in the canonical basis, these tensors have the components
\begin{equation}\label{eq_A1'}
C\left(\frac{\partial}{\partial x^k},\frac{\partial}{\partial x^i},\frac{\partial}{\partial x^j}\right)=
C_{kij}:=g_{ij|k}=\dfrac{\partial g_{ij}}{\partial x^k}-g_{mi}\Gamma^m_{\ jk}-g_{mj}\Gamma^m_{\ ik}\neq 0.
\end{equation}

Clearly, \eqref{C commut in last 2} reads $C_{kij}=C_{kji}$.

We recall that the torsion tensor of an affine connection $\nabla$ is defined as
\begin{equation}\label{eq_A2}
T(X,Y)=\nabla_XY-\nabla_YX-[X,Y].
\end{equation}

In canonical frames, if we denote
$$
T\left(\frac{\partial}{\partial x^k},\frac{\partial}{\partial x^l}\right)=T^i_{\ kl}\frac{\partial}{\partial x^i},
$$
then
\begin{equation}\label{eq_T2}
T^i_{\ kl}=\Gamma^i_{\ lk}-\Gamma^i_{\ kl}
\end{equation}
are the local coefficients of the torsion tensor, a geometric quantity showing the non-symmetry of the connection coefficients $\Gamma$.

Using canonical coframes we can define the torsion 2-form by
\begin{equation}\label{eq_T3}
\Theta^i:=-dx^k\wedge\omega_k^{\ i}
\end{equation}
and again, a simple computation shows that
\begin{equation}\label{eq_T4}
\Theta^i=\frac{1}{2}T^i_{\ kl}dx^k\wedge dx^l.
\end{equation}

In canonical coframes, if $\{\omega_j^{\ i}\}$ is the connection form of $\nabla$, then the conditions \eqref{eq_A1} and \eqref{eq_A2} read

\begin{equation}\label{eq_A1''}
Q_{ij}:=dg_{ij}-g_{kj}\omega_i^{\ k}-g_{ik}\omega_j^{\ k}\neq 0,
\end{equation}
and
\begin{equation}\label{eq_A2''}
\Theta^i:=-dx^j\wedge \omega_j^{\ i}\neq 0,
\end{equation}
respectively. Here $\{Q_{ij}\}$ and $\{\Theta^i\}$ are the non-metricity and torsion forms of $\nabla$, respectively. Obviously we have $Q_{ij}=C_{kij}dx^k$, and $Q_{ij}=Q_{ji}$.

Likewise, in terms of $g$-orthonomal coframes, we get

\begin{equation}\label{eq_A1'''}
Q_{ab}:={-(\omega_{ab}+\omega_{ba})}\neq 0,
\end{equation}
\begin{equation}\label{eq_A2'''}
\Theta^a:=d\omega^a-\omega^b\wedge\omega_b^{\ a}\neq 0.
\end{equation}

We recall the following well-known result.
\begin{proposition}\label{diff of 2 conn}
If $\nabla$ and $\nabla'$ are two affine connections on a manifold $M$, then there exists a (0,2)-tensor $h$ such that
$$
\nabla_XY-\nabla'_XY=h(X,Y),
$$
for any $X,Y\in\mathcal{X}(M)$.
\end{proposition}

By comparing formulas \eqref{eq_L1} and \eqref{eq_A1} above it is natural to make a transformation of $\nabla$ by means of $C$, i.e. we define a new connection $\overset{\ast}{\nabla}$ which will absorb the non-metricity of $\nabla$. Indeed, relation \eqref{eq_A1} can be written as
\begin{equation}\label{eq_D1}
X(g(Y,Z))-g(\nabla_XY,Z)-g(Y,\overset{\ast}{\nabla}_XZ)=0,
\end{equation}
where a new connection $\overset{\ast}{\nabla}$ can be defined by
\begin{equation}\label{eq_Rel1}
g(Y,\overset{\ast}{\nabla}_XZ)=g(Y,\nabla_XZ)+C(X,Y,Z).
\end{equation}

Observe that indeed, $\overset{\ast}{\nabla}$ is an affine connection due to Proposition \ref{diff of 2 conn}.

Locally, using canonical frames $\{\frac{\partial}{\partial x^i}\}$ we get
$$
g\left(\frac{\partial}{\partial x^j},\overset{\ast}{\nabla}_{\dfrac{\partial}{\partial x^i}}\frac{\partial}{\partial x^k}\right)=g\left(\frac{\partial}{\partial x^j},\nabla_{\dfrac{\partial}{\partial x^i}}\frac{\partial}{\partial x^k}\right)+C\left(\frac{\partial}{\partial x^i},\frac{\partial}{\partial x^j},\frac{\partial}{\partial x^k}\right),
$$
i.e.
$$
g_{mj}\overset{\ast}{\Gamma}\, ^m_{\ ki}=g_{mj}\Gamma^m_{\ ki}+C_{ijk},
$$
where we use
\begin{equation}\label{nabla* coeff}
\overset{\ast}{\nabla}_{\dfrac{\partial}{\partial x^i}}\frac{\partial}{\partial x^k}=\overset{\ast}{\Gamma}\, ^m_{\ ki}\frac{\partial}{\partial x^m}.
\end{equation}

Hence, we get
\begin{equation}\label{eq_Rel2}
\overset{\ast}{\Gamma}\, ^m_{\ ki}=\Gamma^m_{\ ki}+g^{mj}C_{ijk}.
\end{equation}

This formula can be regarded as the definition of the local coefficients of the connection $\overset{\ast}{\nabla}$ called the {\it dual connection} of $\nabla$.

By multiplying this formula with $dx^i$ it follows
$$
\overset{\ast}{\Gamma}\, ^m_{\ ki}dx^i=\Gamma^m_{\ ki}dx^i+g^{mj}C_{ijk}dx^i,
$$
i.e.
\begin{equation}\label{eq_Rel3}
\overset{\ast}{\omega}\,_k^{\ m}=\omega_k^{\ m}+g^{mj}C_{ijk}dx^i\\
=\omega_k^{\ m}+g^{mj}Q_{jk}\\
=\omega_k^{\ m}+Q^m_{\ k},
\end{equation}
where
$$
Q_{jk}=C_{ijk}dx^i,\ g^{mj}Q_{jk}=Q^m_{\ k},
$$
is the relation between the connection forms $\overset{\ast}{\omega}$ and $\omega$ of the dual connection $\overset{\ast}{\nabla}$ and $\nabla$, respectively. Using $g^{mj}Q_{jk}=Q^m_{\ k}$, it results $\overset{\ast}{\omega}\,_k^{\ m}=\omega_k^{\ m}+Q^m_{\ k}$. In this case, formulas \eqref{eq_A1'} and \eqref{eq_A1''} read
\begin{equation}\label{eq_D1'}
\frac{\partial g_{ij}}{\partial x^k}-g_{mi}\Gamma^m_{\ jk}-g_{mj}\overset{\ast}{\Gamma}\,^m_{\ ik}=0
\end{equation}
and
\begin{equation}\label{eq_D1''}
dg_{ij}-g_{kj}\omega_i^{\ k}-g_{ik}\overset{\ast}{\omega}\,_j^{\ k}=0,
\end{equation}
respectively.

Likewise, in terms of $g$-orthonormal coframes, the relation between the connection forms $\{\overset{\ast}{\omega}\,_b^{\ a}\}$ and $\{\omega_b^{\ a}\}$ read
\begin{equation}\label{eq_Rel4}
\overset{\ast}{\omega}\,_b^{\ a}=\omega_b^{\ a}+Q^a_{\ b},
\end{equation}
or
$$
\overset{\ast}{\omega}\,_{ba}=\omega_{ba}+Q_{ab},
$$
where
$$
Q_b^{\ a}=u_b^{\ j}Q^i_{\ j}v^a_{\ i}=u_b^{\ j}\left(g^{ik}Q_{kj}\right)v^a_{\ i}
$$

In this case, relation \eqref{eq_A1'''} reads
\begin{equation}\label{eq_D1'''}
\omega_{ab}+\overset{\ast}{\omega}\,_{ba}=0.
\end{equation}

\bigskip
\fbox{\parbox{10cm}
{\it 
We conclude that the dual connection $\overset{\ast}{\nabla}$ is obtained by absorbing the non-metricity 
tensor $C$ (or $Q$) into $\nabla$. This is the geometrical meaning of dual connections.
}
}
\subsection{The difference of the dual connections}

Next, we can define the difference of the dual connections $\nabla$ and $\overset{\ast}{\nabla}$ and observe that this difference is actually given by the non-metricity tensor $C$.

Indeed, if we put
\begin{equation}\label{eq_Diff1}
K_XY=K(X,Y):=\overset{\ast}{\nabla}_XY-\nabla_XY,
\end{equation}
then \eqref{eq_Rel1} implies
\begin{equation}\label{eq_Rel5},
g(Y,K(X,Z))=C(X,Y,Z)
\end{equation}
and
$$
g(Z,K(X,Y))=C(X,Z,Y)=C(X,Y,Z),
$$
i.e.
$$
g(Y,K(X,Z))=g(Z,K(X,Y))=C(X,Y,Z).
$$

Locally, in canonical frames
$$
K^k_{\ ij}=\overset{\ast}{\Gamma}\,^k_{\ ij}-\Gamma^k_{\ ij},
$$
where
$$
K_{\dfrac{\partial}{\partial x^j}}\frac{\partial}{\partial x^i}=K\left(\frac{\partial}{\partial x^j},\frac{\partial}{\partial x^i}\right)=K^k_{\ ij}\frac{\partial}{\partial x^k}.
$$

From \eqref{eq_Rel2} we have
$$
K^m_{\ ki}=g^{mj}C_{ijk}.
$$

If we denote the connection forms difference by
$$
k_j^{\ m}:=\overset{\ast}{\omega}\,_j^{\ m}-\omega_j^{\ m},
$$
then the 1-forms $\{k_j^{\ m}\}$ read
$$
k_j^{\ m}=g^{mk}C_{ikj}dx^i=g^{mk}Q_{kj}=Q^m_{\ j},
$$
and using $g$-orthonormal coframes
$$
k_b^{\ a}=Q_b^{\ a}.
$$

\begin{lemma}\label{lem_D1}
If $\nabla$ and $\overset{\ast}{\nabla}$ and the dual connections on a (pseudo-)Riemannian manifold $(M,g)$ then
\begin{enumerate}[(i)]
\item equation \eqref{eq_D1} is equivalent to
$$
X(g(Y,Z))-g(\overset{\ast}{\nabla}_XY,Z)-g(Y,\nabla_XZ)=0,
$$
\item equation \eqref{eq_D1'} is equivalent to
$$
\frac{\partial g_{ij}}{\partial x^k}-g_{mi}\overset{\ast}{\Gamma}\,^m_{\ jk}-g_{mj}\Gamma^m_{\ ik}=0,
$$
\item equation \eqref{eq_D1''} is equivalent to
$$
dg_{ij}-g_{kj}\overset{\ast}\omega_i^{\ k}-g_{ik}{\omega}\,_j^{\ k}=0.
$$
\end{enumerate}
\end{lemma}

\begin{proof}

The duality condition \eqref{eq_D1} gives
\begin{equation*}
\begin{split}
0&=Z(g(X,Y))-g(\nabla_ZX,Y)-g(X,\overset{\ast}{\nabla}_ZY)\\
&=Z(g(X,Y))-g(\overset{\ast}{\nabla}_ZX-K(Z,X),Y)-g(X,\nabla_ZY+K(Z,Y))\\
&=Z(g(X,Y))-g(\overset{\ast}{\nabla}_ZX,Y)-g(X,\nabla_ZY)
+g(Y,K(Z,X))-g(X,K(Z,Y))\\
&=Z(g(X,Y))-g(\overset{\ast}{\nabla}_ZX,Y)-g(X,\nabla_ZY)
+C(Z,Y,X)-C(Z,X,Y)\\
&=Z(g(X,Y))-g(\overset{\ast}{\nabla}_ZX,Y)-g(X,\nabla_ZY),
\end{split}
\end{equation*}
where we have used \eqref{eq_Rel5} to eliminate $K$ and the symmetry of $C$ in the last two terms (see \eqref{eq_A1}). Hence (i) is proved. 

The other two equivalences can be shown in the same way.
\end{proof}

We will show the relation between the difference tensor $K$ and the torsions of the dual connections.

\begin{proposition}\label{prop:K diff vs T diff}
If $\nabla$ and $\overset{\ast}{\nabla}$ are dual connections on a (pseudo-)Riemannian manifold $(M,g)$, with torsions $T$ and $\overset{\ast}{T}$, then
\begin{equation}
K(X,Y)-K(Y,X)=\overset{\ast}{T}(X,Y)-T(X,Y),
\end{equation}
for any $X,Y\in \mathcal{X}(M)$.
\end{proposition}
\begin{proof}
From \eqref{eq_Diff1}, we get 
\begin{equation*}
\begin{split}
K(X,Y)-K(Y,X)&=(\overset{\ast}{\nabla}_XY-\nabla_XY)-(\overset{\ast}{\nabla}_YX-\nabla_YX)\\
&=(\overset{\ast}{\nabla}_XY-\overset{\ast}{\nabla}_YX-[X,Y])-(\nabla_XY-\nabla_YX-[X,Y]),
\end{split}
\end{equation*}
hence the desired formula follows.
\end{proof}

Taking into account \eqref{eq_Rel5} and its variants, we get
\begin{corollary}
If $\nabla$ and $\overset{\ast}{\nabla}$ are dual connections on a (pseudo-)Riemannian manifold $(M,g)$, with torsions $T$ and $\overset{\ast}{T}$, then
\begin{enumerate}[(i)]
\item $g(Y,\overset{\ast}{T}(X,Z)-T(X,Z))=C(X,Z,Y)-C(Z,X,Y)$,
\item $C$ is totally symmetric if and only if $\overset{\ast}{T}(X,Y)=T(X,Y)$,
\end{enumerate}
for any $X,Y,Z\in \mathcal{X}(M)$.
\end{corollary}

\begin{proposition}
If $\nabla$ and $\overset{\ast}{\nabla}$ are dual connections on a (pseudo-)Riemannian manifold $(M,g)$ then
$$
C(X,Y,Z)+\overset{\ast}{C}(X,Y,Z)=0,
$$
where
$C$ and $\overset{\ast}{C}$ are the non-metricity tensors of $\nabla$ and $\overset{\ast}{\nabla}$, respectively.
\end{proposition}

\begin{proof}
We start with definition of $\overset{\ast}{C}$, i.e.
\begin{equation*}
\begin{split}
\overset{\ast}{C}(X,Y,Z)&:=(\overset{\ast}{\nabla}_Xg)(Y,Z)=X(g(Y,Z))-g(\overset{\ast}{\nabla}_XY,Z)-g(Y,\overset{\ast}{\nabla}_XZ)\\
&=X(g(Y,Z))-X(g(Y,Z))+g(Y,\nabla_XZ)-X(g(Y,Z))+g(\nabla_XY,Z)\\
&=-X(g(Y,Z))+g(\nabla_XY,Z)+g(Y,\nabla_XZ)=-C(X,Y,Z),
\end{split}
\end{equation*}
where we have used the duality condition in the form \eqref{eq_D1} and (i) from Lemma \eqref{lem_D1}.

\end{proof}

We recall that difference of two affine connections is a tensor, and that the sum of two affine connections is an affine connection. It is therefore natural to consider the {\it average connection} of the connections $\nabla$ and $\overset{\ast}{\nabla}$, namely
\begin{equation}\label{eq_Av1}
\overset{(0)}{\nabla}_XY:=\frac{1}{2}\left(\nabla_XY+\overset{\ast}{\nabla}_XY\right).
\end{equation}

Then, we have
\begin{proposition}\label{prop_A8}
The average connection of two dual connections on a (pseudo-)Riemannian manifold $(M,g)$ is $g$-metrical, i.e.
$$
(\overset{(0)}{\nabla}_Xg)(Y,Z)=0.
$$
\end{proposition}

\begin{proof}
By definition, we have

\begin{equation*}
\begin{split}
(\overset{(0)}{\nabla}_Xg)(Y,Z)&=X(g(Y,Z))-g(\overset{(0)}{\nabla}_XY,Z)-g(Y,\overset{(0)}{\nabla}_XZ)\\
&=X(g(Y,Z))-\frac{1}{2}\left[
g(\nabla_XY,Z)+g(\overset{\ast}{\nabla}_XY,Z)
\right]\\
&-\frac{1}{2}\left[
g(Y,\nabla_XZ)+g(Y,\overset{\ast}{\nabla}_XZ)
\right]\\
&=\frac{1}{2}\left\{
X(g(Y,Z))-g(\nabla_XY,Z)-g(Y,\overset{\ast}{\nabla}_XZ)
\right\}\\
&+\frac{1}{2}\left\{
X(g(Y,Z))-g(\overset{\ast}{\nabla}_XY,Z)-g(Y,\nabla_XZ)
\right\}=0
\end{split}
\end{equation*}
by taking into account of \eqref{eq_D1} and (i) of Lemma \ref{lem_D1}.

\end{proof}

\begin{remark}
Observe that since we haven't asked any condition on the torsion free of $\nabla$ and $\overset{\ast}{\nabla}$, there is no reason for $\overset{(0)}{\nabla}$ to be torsion free, i.e. at this moment $\overset{(0)}{\nabla}$ it is not the Levi-Civita connection of $(M,g)$.
\end{remark}

From relation \eqref{eq_Diff1} and \eqref{eq_Av1} we get
$$
\nabla_XY=\overset{(0)}{\nabla}_XY-\frac{1}{2}K(X,Y),
$$
and therefore 
$$
g(\nabla_XY,Z)=g(\overset{(0)}{\nabla}_XY,Z)-\frac{1}{2}g(K(X,Y),Z)
=g(\overset{(0)}{\nabla}_XY,Z)-\frac{1}{2}C(X,Y,Z),
$$
where we use \eqref{eq_Rel5}.

Likewise, using that 
$$
\overset{\ast}{\nabla}_XY=\overset{(0)}{\nabla}_XY+\frac{1}{2}K(X,Y),
$$
it results 
$$
g(\overset{\ast}{\nabla}_XY,Z)=g(\overset{(0)}{\nabla}_XY,Z)+\frac{1}{2}g(K(X,Y),Z)
=g(\overset{(0)}{\nabla}_XY,Z)+\frac{1}{2}C(X,Y,Z).
$$

Therefore, we obtain 
\begin{proposition}\label{prop: recover nablas}
Let $(M,g)$ be a (pseudo-) Riemannian manifold endowed with an affine connection $\overset{(0)}{\nabla}$ and a (0,3)-tensor field $C$. Then the affine connections $\nabla$ and $\overset{\ast}{\nabla}$ on $M$ defined by the following relations
\begin{equation}\label{recover nabla}
g(\nabla_XY,Z)=g(\overset{(0)}{\nabla}_XY,Z)-\frac{1}{2}C(X,Y,Z),
\end{equation}
\begin{equation}\label{recover nabla*}
g(\overset{\ast}{\nabla}_XY,Z)
=g(\overset{(0)}{\nabla}_XY,Z)+\frac{1}{2}C(X,Y,Z),
\end{equation}
have the following properties.
\begin{enumerate}
\item If $\overset{(0)}{\nabla}$ is $g$-metrical and $C(X,Y,Z)=C(X,Z,Y)$, then 
$\nabla$ and $\overset{\ast}{\nabla}$ are $g$-dual connections.
\item If $\overset{(0)}{\nabla}$ is $g$-metrical and $C(X,Y,Z)=C(X,Z,Y)$, then $(\nabla_Xg)(Y,Z)=C(X,Y,Z)$ and $(\overset{\ast}{\nabla}_Xg)(Y,Z)=-C(X,Y,Z)$. 
\item The following relations between the torsions $T$,  $\overset{\ast}{T}$ and $\overset{(0)}{T}$
 of $\nabla$, $\overset{\ast}{\nabla}$ and $\overset{(0)}{\nabla}$ hold good
 \begin{equation}
 \begin{split}
 g(T(X,Y),Z)&=g(\overset{(0)}{T}(X,Y),Z)-\frac{1}{2}\left\{C(X,Y,Z)-C(Y,X,Z)
 \right\},\\
 g(\overset{\ast}{T}(X,Y),Z)&=g(\overset{(0)}{T}(X,Y),Z)+\frac{1}{2}\left\{C(X,Y,Z)-C(Y,X,Z)
 \right\}.
 \end{split}
 \end{equation}
 \end{enumerate}
\end{proposition}

\begin{proof}
Before proving any properties, let us observe that $\nabla$ and $\overset{\ast}{\nabla}$ defined above are indeed affine connections. Indeed, since the difference of any two affine connection is a tensor field, then it is clear that they are well defined.

1. Using \eqref{recover nabla} and \eqref{recover nabla*}, we get
\begin{equation*}
\begin{split}
&X(g(Y,Z))-g(\nabla_XY,Z)-g(Y,\overset{\ast}{\nabla}_XZ) \\
&=
X(g(Y,Z))-g(\overset{(0)}{\nabla}_XY,Z)+\frac{1}{2}C(X,Y,Z)
-g(Y,\overset{(0)}{\nabla}_XZ)-\frac{1}{2}C(X,Z,Y)=0.
\end{split}
\end{equation*}
due to the $g$-metricity of $\overset{(0)}{\nabla}$ and symmetry of $C$ in the 2-nd and 3-rd components. Therefore $\nabla$ and $\overset{\ast}{\nabla}$ are $g$-dual connections.

2. We compute 
\begin{equation*}
\begin{split}
(\nabla_Xg)(Y,Z)&=X(g(Y,Z))-g(\nabla_XY,Z)-g(Y,\nabla_XZ)\\
&=X(g(Y,Z))-g(\overset{(0)}{\nabla}_XY,Z)+\frac{1}{2}C(X,Y,Z)
-g(Y,\overset{(0)}{\nabla}_XZ)+\frac{1}{2}C(X,Z,Y)\\
&=C(X,Y,Z),
\end{split}
\end{equation*}
where we have used the $g$-metricity of $\overset{(0)}{\nabla}$ and the symmetry of $C$ in the 2-nd and 3-rd components. 

The relation $(\overset{\ast}{\nabla}_Xg)(Y,Z)=-C(X,Y,Z)$ can be proved in an identical manner. 

3. Indeed, we have
\begin{equation*}
\begin{split}
g(T(X,Y),Z)&=g(\nabla_XY-\nabla_YX-[X,Y],Z)\\
&=g(\overset{(0)}{\nabla}_XY-\overset{(0)}{\nabla}_YX-[X,Y],Z)-\frac{1}{2}C(X,Y,Z)+\frac{1}{2}C(Y,X,Z)
\end{split}
\end{equation*}
and the relation follows. The second relation can be proved in an identical manner. 

\end{proof}

\begin{remark}
In the case $\overset{(0)}{\nabla}$ is the Levi-Civita connection of $g$ and $C$ is totally symmetric, the connections  $\nabla$ and $\overset{\ast}{\nabla}$ are both torsion free affine connections.
\end{remark}


\subsection{Curvature tensors}

Recall that the curvature $R$ of an affine connection $\nabla$ is defined by
\begin{equation}\label{eq_curv1}
R(X,Y)Z=\nabla_X\nabla_YZ-\nabla_Y\nabla_XZ-\nabla_{[X,Y]}Z.
\end{equation}

In local coordinates if we denote the coefficients
\begin{equation}\label{R tensors def}
R\left(\frac{\partial}{\partial x^k},\frac{\partial}{\partial x^l}\right)\frac{\partial}{\partial x^i}=R_{i\ kl}^{\ j}\frac{\partial}{\partial x^j},
\end{equation}
then a straightforward computation shows that
\begin{equation}\label{eq_Rel10}
R_{i\ kl}^{\ j}=\frac{\partial\Gamma^j_{\ il}}{\partial x^k}-\frac{\partial\Gamma^j_{\ ik}}{\partial x^l}+\Gamma^h_{\ il}\Gamma^j_{\ hk}-\Gamma^h_{\ ik}\Gamma^j_{\ hl},
\end{equation}
i.e.
$$
R=R_{i\ kl}^{\ j}\frac{\partial}{\partial x^j}\otimes dx^i \otimes dx^k \otimes dx^l.
$$

On the other hand, in terms of canonical coframes the curvature 2-form of $\nabla$ is defined as
$$
\Omega_i^{\ j}=d\omega_i^{\ j}-\omega_i^{\ h}\wedge\omega_h^{\ j},
$$
where $\{\omega_i^{\ j}\}$ is the connection form of $\nabla$. By using that $\omega_i^{\ j}=\Gamma^j_{\ ik}dx^k$ it follows
$$
\Omega_i^{\ j}=\frac{1}{2}R^{\ j}_{i\ kl}dx^k\wedge dx^l,
$$
where $R^{\ j}_{i\ kl}$ is given in \eqref{eq_Rel10}. If $R^{\ j}_{i\ kl}\equiv 0$ the affine connection $\nabla$ is called {\it flat}.

Using $g$-orthonormal coframes, we have
$$
\Omega_b^{\ a}=d\omega_b^{\ a}-\omega_b^{\ c}\wedge \omega_c^{\ a}.
$$

\begin{theorem}\label{thm_A10}
Let $\nabla$ be an affine connection on the (pseudo-)Riemannian manifold $(M,g)$. Then the following properties of the curvature tensor $R(X,Y)Z$ hold good.
\begin{enumerate}[(i)]
\item $R(X,Y)Z+R(Y,X)Z=0$;
\item if $\nabla$ is $g$-metrical, then $g(R(Z,V)Y,X)+g(R(Z,V)X,Y)=0$.
\end{enumerate}
\end{theorem}

\begin{proof}
\begin{enumerate}[(i)]
\item Follows immediately from definition \eqref{eq_curv1}.
\item Without assuming that $g$ is metrical, we write
\begin{equation}\label{eq_P0}
\begin{split}
g(R(X,Y)Z,W)&=g(\nabla_X\nabla_YZ-\nabla_Y\nabla_XZ-\nabla_{[X,Y]}Z,W)\\
&=g(\nabla_X\nabla_YZ,W)-g(\nabla_Y\nabla_XZ,W)-g(\nabla_{[X,Y]}Z,W).
\end{split}
\end{equation}

By using \eqref{eq_A1}, observe that
$$
g(\nabla_X\nabla_YZ,W)=X(g(\nabla_YZ,W))-g(\nabla_YZ,\nabla_XW)
-C(X,\nabla_YZ,W),
$$
and
$$
g(\nabla_YZ,\nabla_XW)=Y(g(Z,\nabla_XW))-g(Z,\nabla_Y\nabla_XW)
-C(Y,Z,\nabla_XW).
$$

Hence, we get
\begin{equation}\label{eq_P1}
\begin{split}
g(\nabla_X\nabla_YZ,W)&=X(g(\nabla_YZ,W))-Y(g(Z,\nabla_XW))
+g(Z,\nabla_Y\nabla_XW)\\
&+C(Y,Z,\nabla_XW)-C(X,\nabla_YZ,W).
\end{split}
\end{equation}

By interchanging $X$ and $Y$, we obtain
\begin{equation}\label{eq_P2}
\begin{split}
g(\nabla_Y\nabla_XZ,W)&=Y(g(\nabla_XZ,W))-X(g(Z,\nabla_YW))
+g(Z,\nabla_X\nabla_YW)\\
&+C(X,Z,\nabla_YW)-C(Y,\nabla_XZ,W).
\end{split}
\end{equation}

Likewise, from \eqref{eq_A1} we get
\begin{equation}\label{eq_P3}
\begin{split}
g(\nabla_{[X,Y]}Z,W)&=[X,Y](g(Z,W))-g(Z,\nabla_{[X,Y]}W)-C([X,Y],Z,W)\\
&=X(Y(g(Z,W)))-Y(X(g(Z,W)))-g(Z,\nabla_{[X,Y]}W)\\
&-C([X,Y],Z,W).
\end{split}
\end{equation}

By substituting \eqref{eq_P1}, \eqref{eq_P2}, \eqref{eq_P3} in \eqref{eq_P0}, it follows
\begin{equation*}
\begin{split}
g(R(X,Y)Z,W)&=\underline{X(g(\nabla_YZ,W))}-\underline{Y(g(Z,\nabla_XW))}\\
&+g(Z,\nabla_Y\nabla_XW)+C(Y,Z,\nabla_XW)-C(X,\nabla_YZ,W)\\
&-\underline{Y(g(\nabla_XZ,W))}+\underline{X(g(Z,\nabla_YW))}-g(Z,\nabla_X\nabla_YW)\\
&-C(X,Z,\nabla_YW)+C(Y,\nabla_XZ,W)\\
&-X(Y(g(Z,W)))+Y(X(g(Z,W)))+g(Z,\nabla_{[X,Y]}W)+C([X,Y],Z,W).
\end{split}
\end{equation*}

If we use again \eqref{eq_A1} for the underlined terms, we get
\begin{equation*}
\begin{split}
g(R(X,Y)Z,W)&=X\{
\underset{\text{\textcircled{\tiny 1}}}{\underline{Y(g(Z,W))}}
-g(Z,\nabla_YW)-C(Y,Z,W)
\}\\
&+Y\{
-\underset{\text{\textcircled{\tiny 2}}}{\underline{X(g(Z,W))}}
+g(W,\nabla_XZ)+C(X,Z,W)
\}\\
&+\underset{\text{\textcircled{\tiny 3}}}{\underline{g(\nabla_Y\nabla_XW,Z)}}
+C(Y,Z,\nabla_XW)-C(X,W,\nabla_YZ)\\
&+Y\{-X(g(Z,W))+g(Z,\nabla_XW)+C(X,Z,W)\}\\
&+X\{Y(g(Z,W))-g(W,\nabla_YZ)-C(Y,Z,W)\}\\
&-\underset{\text{\textcircled{\tiny 3}}}{\underline{g(\nabla_X\nabla_YW,Z)}}
-C(X,Z,\nabla_YW)+C(Y,W,\nabla_XZ)\\
&-\underset{\text{\textcircled{\tiny 1}}}{\underline{X(Y(g(Z,W)))}}
+\underset{\text{\textcircled{\tiny 2}}}{\underline{Y(X(g(Z,W)))}}
+\underset{\text{\textcircled{\tiny 3}}}{\underline{g(Z,\nabla_{[X,Y]W})}}\\
&+C([X,Y],Z,W).
\end{split}
\end{equation*}

By observing that the underlined terms \textcircled{\footnotesize 1}, \textcircled{\footnotesize 2} vanish and that \textcircled{\footnotesize 3} equals $-g(R(X,Y)W,Z)$ we obtain
\begin{equation*}
\begin{split}
&g(R(X,Y)Z,W)\\
&=-g(R(X,Y)W,Z)\\
&+X\{
\underset{\text{\textcircled{\tiny 1}}}{\underline{
-g(\nabla_YW,Z)-C(Y,Z,W)+Y(g(Z,W))-g(W,\nabla_YZ)
}}
-C(Y,Z,W)
\}\\
&+Y\{
\underset{\text{\textcircled{\tiny 2}}}{\underline{
g(\nabla_XZ,W)-C(X,Z,W)-X(g(Z,W))+g(Z,\nabla_XW)
}}
-C(Y,Z,W)
\}\\
&+C(Y,Z,\nabla_XW)-C(X,W,\nabla_YZ)-C(X,Z,\nabla_YW)\\
&+C(Y,W,\nabla_XZ)+C([X,Y],Z,W).
\end{split}
\end{equation*}

Observe that the underlined terms \textcircled{\footnotesize 1} and \textcircled{\footnotesize 2} vanish by \eqref{eq_A1}, hence we get
\begin{equation}\label{eq_P4}
\begin{split}
g(R(X,Y)Z,W)&=-g(R(X,Y)W,Z)
-X(C(Y,Z,W))+Y(C(X,Z,W))\\
&+C(Y,Z,\nabla_XW)-C(X,W,\nabla_YZ)-C(X,Z,\nabla_YW)\\
&+C(Y,W,\nabla_XZ)+C([X,Y],Z,W).
\end{split}
\end{equation}

In the case $g$-metrical, i.e. $C\equiv 0$ the conclusion follows immediately.
\end{enumerate}
\end{proof}

The tensor 
$$
R(X,Y,Z,V):=g(R(Z,V)Y,X),
$$
is called the {\it Riemann-Christoffel curvature tensor} of $\nabla$, hence (i), (ii) in Theorem \ref{thm_A10} can be also written as
\begin{enumerate}
\item[(i)$'$] $R(X,Y,Z,V)+R(X,Y,V,Z)=0$
\item[(ii)$'$] if $\nabla$ is $g$-metrical, then $R(X,Y,Z,V)+R(Y,X,Z,V)=0$.
\end{enumerate}

Let us recall here the following Bianchi identities for $\nabla$. 
\begin{enumerate}[(i)]
\item {\bf The First Bianchi identity}:
\begin{equation}\label{1-st Bianchi}
\begin{split}
&(\nabla_XR)(Y,Z)+(\nabla_YR)(Z,X)+(\nabla_ZR)(X,Y)\\
&=-R(T(X,Y),Z)-R(T(Y,Z),X)-R(T(Z,X),Y).
\end{split}
\end{equation}
\item {\bf The Second Bianchi identity}:
\begin{equation}\label{2-nd Bianchi}
\begin{split}
&(\nabla_XT)(Y,Z)+(\nabla_YT)(Z,X)+(\nabla_ZT)(X,Y)\\
&=R(X,Y)Z+R(Y,Z)X+R(Z,X)Y\\
&-T(T(X,Y),Z))-T(T(Y,Z),X))-T(T(Z,X),Y)).
\end{split}
\end{equation}
\end{enumerate}

For the proof see any comprehensive textbook on Riemannian geometry. 

Moreover, if $\nabla$ is torsion free then the First and Second Bianchi identities read
\begin{enumerate}
\item[(iii)$'$] $(\nabla_XR)(Y,Z)+(\nabla_YR)(Z,X)+(\nabla_ZR)(X,Y)=0$,
\item[(iv)$'$] $R(X,Y)Z+R(Y,Z)X+R(Z,X)Y=0$,
\end{enumerate}
respectively.

Since $\nabla$ and $\overset{\ast}{\nabla}$ are not $g$-metrical connections, the skew symetry condition (ii') above in the first two positions of $R(X,Y,Z,W)$ and $\overset{\ast}{R}(X,Y,Z,W)$ do not hold. However, the following combined relation can be proved.

\begin{theorem}\label{thm: R and R* 1-st 2 ind}
If $\nabla$ and $\overset{\ast}{\nabla}$ are dual connections on a (pseudo-)Riemannian manifold $(M,g)$, then
\begin{equation}\label{eq_12.1}
g(R(X,Y)Z,W)+g(\overset{\ast}{R}(X,Y)W,Z)=0,
\end{equation}
where $R$ and $\overset{\ast}{R}$ are the curvature tensor of $\nabla$ and $\overset{\ast}{\nabla}$, respectively.
\end{theorem}

\begin{proof}
Starting with
\begin{equation}\label{eq_thm_A12_0}
g(R(X,Y)Z,W)=g(\nabla_X\nabla_YZ,W)-g(\nabla_Y\nabla_XZ,W)-g(\nabla_{[X,Y]}Z,W),
\end{equation}
we compute the terms in the RHS separately using the duality condition \eqref{eq_D1}.

The first term reads:
\begin{equation}\label{eq_thm_A12_1}
\begin{split}
&g(\nabla_X\nabla_YZ,W)\\
&=X(g(\nabla_YZ,W))-g(\nabla_YZ,\overset{\ast}{\nabla}_XW)\\
&=X(Y(g(Z,W))-g(Z,\overset{\ast}{\nabla}_YW))-Y(g(Z,\overset{\ast}{\nabla}_XW))+g(Z,\overset{\ast}{\nabla}_Y\overset{\ast}{\nabla}_XW)\\
&=X(Y(g(Z,W)))-X(g(Z,\overset{\ast}{\nabla}_YW))-Y(g(Z,\overset{\ast}{\nabla}_XW))+g(Z,\overset{\ast}{\nabla}_Y\overset{\ast}{\nabla}_XW).
\end{split}
\end{equation}

The second term reads
\begin{equation}\label{eq_thm_A12_2}
\begin{split}
&g(\nabla_Y\nabla_XZ,W)\\
&=Y(X(g(Z,W)))-Y(g(Z,\overset{\ast}{\nabla}_XW))-X(g(Z,\overset{\ast}{\nabla}_YW))+g(Z,\overset{\ast}{\nabla}_X\overset{\ast}{\nabla}_YW).
\end{split}
\end{equation}

Finally, the third term reads
\begin{equation}\label{eq_thm_A12_3}
\begin{split}
g(\nabla_{[X,Y]}Z,W)&=[X,Y]g(Z,W)-g(Z,\overset{\ast}{\nabla}_{[X,Y]}W)\\
&=X(Y(g(Z,W)))-Y(X(g(Z,W)))-g(Z,\overset{\ast}{\nabla}_{[X,Y]}W).
\end{split}
\end{equation}

Substituting now \eqref{eq_thm_A12_1}, \eqref{eq_thm_A12_2}, \eqref{eq_thm_A12_3} in \eqref{eq_thm_A12_0} we get
\begin{equation*}
\begin{split}
g(R(X,Y)Z,W)&=
\cancel{\underset{\text{\textcircled{\tiny 1}}}{\underline{X(Y(g(Z,W)))}}}
-\cancel{\underset{\text{\textcircled{\tiny 4}}}{\underline{X(g(Z,\overset{\ast}{\nabla_YW}))}}}\\
&-\cancel{\underset{\text{\textcircled{\tiny 3}}}{\underline{Y(g(Z,\overset{\ast}{\nabla_XW}))}}}
+g(Z,\overset{\ast}{\nabla}_Y\overset{\ast}{\nabla}_XW)\\
&-\cancel{\underset{\text{\textcircled{\tiny 2}}}{\underline{Y(X(g(Z,W)))}}}
+\cancel{\underset{\text{\textcircled{\tiny 3}}}{\underline{Y(g(Z,\overset{\ast}{\nabla_XW}))}}}
+\cancel{\underset{\text{\textcircled{\tiny 4}}}{\underline{X(g(Z,\overset{\ast}{\nabla_YW}))}}}
-g(Z,\overset{\ast}{\nabla}_X\overset{\ast}{\nabla}_YW)\\
&-\cancel{\underset{\text{\textcircled{\tiny 1}}}{\underline{X(Y(g(Z,W)))}}}
+\cancel{\underset{\text{\textcircled{\tiny 2}}}{\underline{Y(X(g(Z,W)))}}}
+g(Z,\overset{\ast}{\nabla}_{[X,Y]}W)\\
&=-g(\overset{\ast}{\nabla}_X\overset{\ast}{\nabla}_YW-\overset{\ast}{\nabla}_Y\overset{\ast}{\nabla}_XW-\overset{\ast}{\nabla}_{[X,Y]}W,Z)\\
&=-g(\overset{\ast}{R}(X,Y)W,Z).
\end{split}
\end{equation*}

\end{proof}

We observe that \eqref{eq_12.1} can be written as 
$$
R(W,Z,X,Y)+\overset{\ast}{R}(Z,W,X,Y)=0.
$$

The relation between the curvature tensors $R$, $\overset{\ast}{R}$ and $\overset{(0)}{R}$ of the dual connections $\nabla$, $\overset{\ast}{\nabla}$ and average connection $\overset{(0)}{\nabla}$, respectively can be summarized in the following theorem.

\begin{theorem}\label{thm_C1}
Let $\nabla$, $\overset{\ast}{\nabla}$ be $g$-dual connections on the (pseudo-)Riemannian manifold $(M,g)$ and let $\overset{(0)}{\nabla}$ be the average connection. Then
\begin{enumerate}[(i)]
\item 
\begin{equation*}
\begin{split}
& R(X,Y)Z\\
&=\overset{(0)}{R}(X,Y)Z-\frac{1}{2}
\left\{
(\overset{(0)}{\nabla}_XK)(Y,Z)-(\overset{(0)}{\nabla}_YK)(X,Z)-\frac{1}{2}[K_X,K_Y]Z
\right\}
-\frac{1}{2}K(\overset{(0)}{T}(X,Y),Z),\\
&=\overset{(0)}{R}(X,Y)Z-\frac{1}{2}\left\{
(\nabla_XK)(Y,Z)-(\nabla_YK)(X,Z)+\frac{1}{2}[K_X,K_Y]Z
\right\}-\frac{1}{2}K(T(X,Y),Z);
\end{split}
\end{equation*}
\item
\begin{equation*}
\begin{split}
&\overset{\ast}{R}(X,Y)Z\\
&=\overset{(0)}{R}(X,Y)Z
+\frac{1}{2}
\left\{
(\overset{(0)}{\nabla}_XK)(Y,Z)-(\overset{(0)}{\nabla}_YK)(X,Z)+\frac{1}{2}[K_X,K_Y]Z
\right\}
+\frac{1}{2}K(\overset{(0)}{T}(X,Y),Z)\\
&=\overset{(0)}{R}(X,Y)Z
+\frac{1}{2}\left\{
(\nabla_XK)(Y,Z)-(\nabla_YK)(X,Z)+\frac{3}{2}[K_X,K_Y]Z\right\}
+\frac{1}{2}K(T(X,Y),Z).
\end{split}
\end{equation*}

Moreover, we have
\item 
\begin{equation*}
\begin{split}
&R(X,Y)Z-\overset{\ast}{R}(X,Y)Z\\
&=-\left\{
(\overset{(0)}{\nabla}_XK)(Y,Z)-(\overset{(0)}{\nabla}_YK)(X,Z)
\right\}-K(\overset{(0)}{T}(X,Y),Z)\\
&=-\left\{
(\nabla_XK)(Y,Z)-(\nabla_YK)(X,Z)+[K_X,K_Y]Z
\right\}-K(T(X,Y),Z)
\end{split}
\end{equation*}
and
\item
$$
R(X,Y)Z+\overset{\ast}{R}(X,Y)Z=2\overset{(0)}{R}(X,Y)Z+\frac{1}{2}[K_X,K_Y]Z.
$$
\end{enumerate}
\end{theorem}
\begin{proof}
See \ref{App1}
\end{proof}
\begin{corollary}\label{rem_C2}
In local coordinates, the formulas in Theorem \ref{thm_C1} read
\begin{enumerate}[(i)]
\item
\begin{equation*}
\begin{split}
R^{\ k}_{m\ ji}&=\overset{(0)}{R}\,^{\ k}_{m\ ji}-\frac{1}{2}(K^{\ k}_{m\ i\overset{(0)}{|}j}-K^{\ k}_{m\ j\overset{(0)}{|}i})+\frac{1}{4}(K^l_{\ mi}K^k_{\ lj}-K^l_{\ mj}K^k_{\ li})
-\frac{1}{2}\overset{(0)}{T}\,^l_{\ ji}K^k_{\ ml}\\
&=\overset{(0)}{R}\,^{\ k}_{m\ ji}-\frac{1}{2}(K^{\ k}_{m\ i|j}-K^{\ k}_{m\ j|i})-\frac{1}{4}(K^l_{\ mi}K^k_{\ lj}-K^l_{\ mj}K^k_{\ li})
-\frac{1}{2}T^l_{\ ji}K^k_{\ ml}
\end{split}
\end{equation*}
\item 
\begin{equation*}
\begin{split}
\overset{\ast}{R}\,^{\ k}_{m\ ji}&=\overset{(0)}{R}\,^{\ k}_{m\ ji}+\frac{1}{2}(K^{\ k}_{m\ i\overset{(0)}{|}j}-K^{\ k}_{m\ j\overset{(0)}{|}i})+\frac{1}{4}(K^l_{\ mi}K^k_{\ lj}-K^l_{\ mj}K^k_{\ li})
+\frac{1}{2}\overset{(0)}{T}\,^l_{\ ji}K^k_{\ ml}\\
&=\overset{(0)}{R}\,^{\ k}_{m\ ji}+\frac{1}{2}(K^{\ k}_{m\ i|j}-K^{\ k}_{m\ j|i})+\frac{3}{4}(K^l_{\ mi}K^k_{\ lj}-K^l_{\ mj}K^k_{\ li})
+\frac{1}{2}T^l_{\ ji}K^k_{\ ml}
\end{split}
\end{equation*}
\item 
\begin{equation*}
\begin{split}
R^{\ k}_{m\ ji}-\overset{\ast}{R}\,^{\ k}_{m\ ji}&=-(K^{\ k}_{m\ i\overset{(0)}{|}j}-K^{\ k}_{m\ j\overset{(0)}{|}i})-\overset{(0)}{T}\,^l_{\ ji}K^k_{\ ml}\\
&=-(K^{\ k}_{m\ i|j}-K^{\ k}_{m\ j|i})-(K^{l}_{\ mi}K^{k}_{\ lj}-K^{l}_{\ mj}K^k_{\ li})-T^l_{\ ji}K^k_{\ ml}
\end{split}
\end{equation*}
\item 
\begin{equation*}
\begin{split}
R^{\ k}_{m\ ji}+\overset{\ast}{R}\,^{\ k}_{m\ ji}=2\overset{(0)}{R}\,^{\ k}_{m\ ji}
+\frac{1}{2}(K^l_{\ mi}K^k_{\ lj}-K^l_{\ mj}K^k_{\ li}).
\end{split}
\end{equation*}
\end{enumerate}
\end{corollary}
\begin{proof}
See \ref{App1}
\end{proof}

If we define the {\it Ricci curvature tensors} 
as $R_{mi}:=R^{\ j}_{m\ ji}$,
then we can use Corollary \ref{rem_C2} to obtain the following result.

\begin{proposition}\label{prop_C3}
Let $\nabla$ and $\overset{\ast}{\nabla}$ be $g$-dual connections on the (pseudo-)Riemannian manifold $(M,g)$ and let $\overset{(0)}{\nabla}$ be the average connection. Then
\begin{enumerate}[(i)]
\item 
\begin{equation*}
\begin{split}
R_{mi}&=\overset{(0)}{R}_{mi}-\frac{1}{2}(K^{\ j}_{m\ i\overset{(0)}{|}j}-K^{\ j}_{m\ j\overset{(0)}{|}i})+\frac{1}{4}(K^l_{\ mi}K^j_{\ lj}-K^l_{\ mj}K^j_{\ li})
-\frac{1}{2}\overset{(0)}{T}\,^l_{\ ji}K^j_{\ ml}\\
&=\overset{(0)}{R}_{mi}-\frac{1}{2}(K^{\ j}_{m\ i|j}-K^{\ j}_{m\ j|i})-\frac{1}{4}(K^l_{\ mi}K^j_{\ lj}-K^l_{\ mj}K^j_{\ li})
-\frac{1}{2}T^l_{\ ji}K^j_{\ ml},
\end{split}
\end{equation*}
\item 
\begin{equation*}
\begin{split}
\overset{\ast}{R}_{mi}&=\overset{(0)}{R}_{mi}+\frac{1}{2}(K^{\ j}_{m\ i\overset{(0)}{|}j}-K^{\ j}_{m\ j\overset{(0)}{|}i})+\frac{1}{4}(K^l_{\ mi}K^j_{\ lj}-K^l_{\ mj}K^j_{\ li})
+\frac{1}{2}\overset{(0)}{T}\,^l_{\ ji}K^j_{\ ml}\\
&=\overset{(0)}{R}_{mi}+\frac{1}{2}(K^{\ j}_{m\ i|j}-K^{\ j}_{m\ j|i})+\frac{3}{4}(K^l_{\ mi}K^j_{\ lj}-K^l_{\ mj}K^j_{\ li})
+\frac{1}{2}T^l_{\ ji}K^j_{\ ml},
\end{split}
\end{equation*}
\item 
\begin{equation*}
\begin{split}
R_{mi}-\overset{\ast}{R}_{mi}&=-(K^{\ j}_{m\ i\overset{(0)}{|}j}-K^{\ j}_{m\ j\overset{(0)}{|}i})-\overset{(0)}{T}\,^l_{\ ji}K^j_{\ ml}\\
&=-(K^{\ j}_{m\ i|j}-K^{\ j}_{m\ j|i})-(K^{l}_{\ mi}K^{j}_{\ lj}-K^{l}_{\ mj}K^j_{\ li})-T^l_{\ ji}K^j_{\ ml},
\end{split}
\end{equation*}
\item 
\begin{equation*}
\begin{split}
R_{mi}+\overset{\ast}{R}_{mi}=2\overset{(0)}{R}_{mi}+\frac{1}{2}(K^l_{\ mi}K^j_{\ lj}-K^l_{\ mj}K^j_{\ li}).
\end{split}
\end{equation*}
\end{enumerate}
\end{proposition}

\begin{remark}
In the formulas above, the terms $K^{\ j}_{m\ i|j}$ can be regarded as the {\it divergence} of $K$, while $K^j_{\ mj}$ will be called the {\it trace on the right} of $K$, since $K^j_{\ mj}$ is not symmetric in the lower indices in the general case.

More formally, 
\begin{itemize}
\item The {\it right trace} will be denoted as 
\begin{equation*}
(K^k_{\ ik})(X)=Tr_1(K)(X)=Tr(Y\mapsto K(X,Y)),
\end{equation*}
in other words, we fix the first input vector field to be $X$, this gives a linear map from $TM\to TM$ defined by $Y\mapsto K(X,Y)$.
Then take the trace of this resulting linear map.

Locally, we have 
\begin{equation*}
K^k_{\ ik}X^i=\sum_kg\left(  \frac{\partial}{\partial x^k},K(X,\frac{\partial}{\partial x^k}) \right).
\end{equation*}
\item The {\it left trace} will be denoted as 
\begin{equation*}
(K^k_{\ ki})(X)=Tr_2(K)(X)=Tr(Y\mapsto K(Y,X)),
\end{equation*}
in other words, we fix the first input vector field to be $X$, this gives a linear map from $TM\to TM$ defined by $Y\mapsto K(Y,X)$.
Then take the trace of this resulting linear map.

Locally, we have 
\begin{equation*}
K^k_{\ ik}X^i=\sum_kg\left(  \frac{\partial}{\partial x^k},K(\frac{\partial}{\partial x^k},X) \right).
\end{equation*}
\item The {\it $\nabla$-divergence} of $K$ will be denoted as
\begin{equation*}
Div_\nabla(K)(X,Y)=\sum_j
g\left(\frac{\partial}{\partial x^j},(\nabla_\frac{\partial}{\partial x^j} K)(X,Y)\right).
\end{equation*}

Locally, we have 
\begin{equation*}
Div_\nabla(K)_{ik}=K^{\ j}_{i\ k|j}.
\end{equation*}
\end{itemize}
\end{remark}

\begin{corollary}
Let $\nabla$ and $\overset{\ast}{\nabla}$ be $g$-dual connections on the (pseudo-)Riemannian manifold $(M,g)$ and let $\overset{(0)}{\nabla}$ be the average connection. If the right trace of $K$ vanish, i.e. $Tr_1(K)=0$, then
\begin{enumerate}[(i)]
\item 
\begin{equation*}
\begin{split}
R_{mi}&=\overset{(0)}{R}_{mi}-\frac{1}{2}K^{\ j}_{m\ i\overset{(0)}{|}j}-\frac{1}{4}K^l_{\ mj}K^j_{\ li}-\frac{1}{2}\overset{(0)}{T}\,^l_{\ ji}K^j_{\ ml}\\
&=\overset{(0)}{R}_{mi}-\frac{1}{2}K^{\ j}_{m\ i|j}+\frac{1}{4}K^l_{\ mj}K^j_{\ li}
-\frac{1}{2}T^l_{\ ji}K^j_{\ ml}
\end{split}
\end{equation*}
\item 
\begin{equation*}
\begin{split}
\overset{\ast}{R}_{mi}&=\overset{(0)}{R}_{mi}+\frac{1}{2}K^{\ j}_{m\ i\overset{(0)}{|}j}-\frac{1}{4}K^l_{\ mj}K^j_{\ li}
+\frac{1}{2}\overset{(0)}{T}\,^l_{\ ji}K^j_{\ ml}\\
&=\overset{(0)}{R}_{mi}+\frac{1}{2}K^{\ j}_{m\ i|j}-\frac{3}{4}K^l_{\ mj}K^j_{\ li}
+\frac{1}{2}T^l_{\ ji}K^j_{\ ml},
\end{split}
\end{equation*}
\item 
\begin{equation*}
\begin{split}
R_{mi}-\overset{\ast}{R}_{mi}&=-K^{\ j}_{m\ i\overset{(0)}{|}j}-\overset{(0)}{T}\,^l_{\ ji}K^j_{\ ml}\\
&=-K^{\ j}_{m\ i|j}+K^{l}_{\ mj}K^j_{\ li}-T^l_{\ ji}K^j_{\ ml},
\end{split}
\end{equation*}
\item 
\begin{equation*}
\begin{split}
R_{mi}+\overset{\ast}{R}_{mi}=2\overset{(0)}{R}_{mi}-\frac{1}{2}K^l_{\ mj}K^j_{\ li}.
\end{split}
\end{equation*}
\end{enumerate}
\end{corollary}

\begin{corollary}\label{cor_C4}
Let $\nabla$ and $\overset{\ast}{\nabla}$ be $g$-dual connections on the (pseudo-)Riemannian manifold $(M,g)$ and let $\overset{(0)}{\nabla}$ be the average connection. 
Then, we have
\begin{enumerate}
\item 
\begin{equation*}
\begin{split}
R_{mi}-R_{im}&=\overset{(0)}{R}_{mi}-\overset{(0)}{R}_{im}\\
&-\frac{1}{2}\left\{(K^j_{\ mi}-K^j_{\ im})_{\overset{(0)}{|}j}-K^{\ j}_{m\ j\overset{(0)}{|}i}+K^{\ j}_{i\ j\overset{(0)}{|}m}\right\}\\
&+\frac{1}{4}\left\{(K^l_{\ mi}-K^l_{\ im})K^j_{\ lj}-K^l_{\ mj}K^j_{\ li}+K^l_{\ ij}K^j_{\ lm}\right\}\\
&-\frac{1}{2}\left\{\overset{(0)}{T}\,^l_{\ ji}K^j_{\ ml}-\overset{(0)}{T}\,^l_{\ jm}K^j_{\ il}\right\},\\
\end{split}
\end{equation*}
\item 
\begin{equation*}
\begin{split}
\overset{\ast}{R}_{mi}-\overset{\ast}{R}_{im}&=\overset{(0)}{R}_{mi}-\overset{(0)}{R}_{im}\\
&+\frac{1}{2}\left\{(K^j_{\ mi}-K^j_{\ im})_{\overset{(0)}{|}j}-K^{\ j}_{m\ j\overset{(0)}{|}i}+K^{\ j}_{i\ j\overset{(0)}{|}m}\right\}\\
&+\frac{1}{4}\left\{(K^l_{\ mi}-K^l_{\ im})K^j_{\ lj}-K^l_{\ mj}K^j_{\ li}+K^l_{\ ij}K^j_{\ lm}\right\}\\
&+\frac{1}{2}\left\{\overset{(0)}{T}\,^l_{\ ji}K^j_{\ ml}-\overset{(0)}{T}\,^l_{\ jm}K^j_{\ il}\right\},\\
\end{split}
\end{equation*}
\item 
\begin{equation*}
\begin{split}
R_{mi}+R_{im}&=\overset{(0)}{R}_{mi}+\overset{(0)}{R}_{im}\\
&-\frac{1}{2}\left\{(K^j_{\ mi}+K^j_{\ im})_{\overset{(0)}{|}j}-(K^{\ j}_{m\ j\overset{(0)}{|}i}+K^{\ j}_{i\ j\overset{(0)}{|}m})\right\}\\
&+\frac{1}{4}\left\{(K^l_{\ mi}+K^l_{\ im})K^j_{\ lj}-(K^l_{\ mj}K^j_{\ li}+K^l_{\ ij}K^j_{\ lm})\right\}\\
&-\frac{1}{2}\left\{\overset{(0)}{T}\,^l_{\ ji}K^j_{\ ml}+\overset{(0)}{T}\,^l_{\ jm}K^j_{\ il}\right\},\\
\end{split}
\end{equation*}
\item 
\begin{equation*}
\begin{split}
\overset{\ast}{R}_{mi}+\overset{\ast}{R}_{im}&=\overset{(0)}{R}_{mi}+\overset{(0)}{R}_{im}\\
&+\frac{1}{2}\left\{(K^j_{\ mi}+K^j_{\ im})_{\overset{(0)}{|}j}-(K^{\ j}_{m\ j\overset{(0)}{|}i}+K^{\ j}_{i\ j\overset{(0)}{|}m})\right\}\\
&+\frac{1}{4}\left\{(K^l_{\ mi}+K^k_{\ im})K^j_{\ lj}-(K^l_{\ mj}K^j_{\ li}+K^l_{\ ij}K^j_{\ lm})\right\}\\
&+\frac{1}{2}\left\{\overset{(0)}{T}\,^l_{\ ji}K^j_{\ ml}+\overset{(0)}{T}\,^l_{\ jm}K^j_{\ il}\right\}.\\
\end{split}
\end{equation*}
\end{enumerate}
\end{corollary}

\begin{corollary}\label{cor_C4_1}
Let $\nabla$ and $\overset{\ast}{\nabla}$ be $g$-dual connections on the (pseudo-)Riemannian manifold $(M,g)$ and let $\overset{(0)}{\nabla}$ be the average connection. If the right trace of $K$ vanish, i.e. $Tr_1(K)=0$, then we have
\begin{enumerate}
\item 
\begin{equation*}
\begin{split}
R_{mi}-R_{im}&=\overset{(0)}{R}_{mi}-\overset{(0)}{R}_{im}\\
&-\frac{1}{2}\left\{(K^j_{\ mi}-K^j_{\ im})_{\overset{(0)}{|}j}
\right\}
+\frac{1}{4}\left\{
K^l_{\ ij}K^j_{\ lm}-K^l_{\ ji}K^j_{\ ml}
\right\}\\
&-\frac{1}{2}\left\{\overset{(0)}{T}\,^l_{\ ji}K^j_{\ ml}-\overset{(0)}{T}\,^l_{\ jm}K^j_{\ il}\right\},\\
\end{split}
\end{equation*}
\item 
\begin{equation*}
\begin{split}
\overset{\ast}{R}_{mi}-\overset{\ast}{R}_{im}&=\overset{(0)}{R}_{mi}-\overset{(0)}{R}_{im}\\
&+\frac{1}{2}\left\{(K^j_{\ mi}-K^j_{\ im})_{\overset{(0)}{|}j}
\right\}+\frac{1}{4}\left\{
K^l_{\ ij}K^j_{\ lm}-K^l_{\ ji}K^j_{\ ml}\right\}\\
&+\frac{1}{2}\left\{\overset{(0)}{T}\,^l_{\ ji}K^j_{\ ml}-\overset{(0)}{T}\,^l_{\ jm}K^j_{\ il}\right\},\\
\end{split}
\end{equation*}
\item 
\begin{equation*}
\begin{split}
R_{mi}+R_{im}&=\overset{(0)}{R}_{mi}+\overset{(0)}{R}_{im}\\
&-\frac{1}{2}\left\{(K^j_{\ mi}+K^j_{\ im})_{\overset{(0)}{|}j}
\right\}-\frac{1}{4}\left\{
K^l_{\ ij}K^j_{\ lm}+K^l_{\ ji}K^j_{\ ml}\right\}\\
&-\frac{1}{2}\left\{\overset{(0)}{T}\,^l_{\ ji}K^j_{\ ml}+\overset{(0)}{T}\,^l_{\ jm}K^j_{\ il}\right\},\\
\end{split}
\end{equation*}
\item 
\begin{equation*}
\begin{split}
\overset{\ast}{R}_{mi}+\overset{\ast}{R}_{im}&=\overset{(0)}{R}_{mi}+\overset{(0)}{R}_{im}\\
&+\frac{1}{2}(K^j_{\ mi}+K^j_{\ im})_{\overset{(0)}{|}j}
-\frac{1}{4}
(K^l_{\ ij}K^j_{\ lm}+K^l_{\ ji}K^j_{\ ml})\\
&+\frac{1}{2}\left\{\overset{(0)}{T}\,^l_{\ ji}K^j_{\ ml}+\overset{(0)}{T}\,^l_{\ jm}K^j_{\ il}\right\}.\\
\end{split}
\end{equation*}
\end{enumerate}
\end{corollary}

\begin{corollary}\label{cor_C4_2}
Let $\nabla$ and $\overset{\ast}{\nabla}$ be $g$-dual connections on the (pseudo-)Riemannian manifold $(M,g)$ and let $\overset{(0)}{\nabla}$ be the average connection. If  the right trace and the $\overset{(0)}{\nabla}$-divergence of $K$ and vanish, i.e. $Tr_1(K)=0$ and $Div_{\overset{(0)}{\nabla}}(K)=0$, then we have
\begin{enumerate}
\item 
\begin{equation*}
\begin{split}
R_{mi}-R_{im}&=\overset{(0)}{R}_{mi}-\overset{(0)}{R}_{im}\\
&
+
\frac{1}{4}\left\{
K^l_{\ ij}K^j_{\ lm}-K^l_{\ ji}K^j_{\ ml}
\right\}-\frac{1}{2}\left\{\overset{(0)}{T}\,^l_{\ ji}K^j_{\ ml}-\overset{(0)}{T}\,^l_{\ jm}K^j_{\ il}\right\}\\
\end{split}
\end{equation*}
\item 
\begin{equation*}
\begin{split}
\overset{\ast}{R}_{mi}-\overset{\ast}{R}_{im}&=\overset{(0)}{R}_{mi}-\overset{(0)}{R}_{im}\\
&
+\frac{1}{4}\left\{
K^l_{\ ij}K^j_{\ lm}-K^l_{\ ji}K^j_{\ ml}\right\}+\frac{1}{2}\left\{\overset{(0)}{T}\,^l_{\ ji}K^j_{\ ml}-\overset{(0)}{T}\,^l_{\ jm}K^j_{\ il}\right\}\\
\end{split}
\end{equation*}
\item 
\begin{equation*}
\begin{split}
R_{mi}+R_{im}&=\overset{(0)}{R}_{mi}+\overset{(0)}{R}_{im}\\
&
-\frac{1}{4}\left\{
K^l_{\ ij}K^j_{\ lm}+K^l_{\ ji}K^j_{\ ml}\right\}-\frac{1}{2}\left\{\overset{(0)}{T}\,^l_{\ ji}K^j_{\ ml}+\overset{(0)}{T}\,^l_{\ jm}K^j_{\ il}\right\},\\
\end{split}
\end{equation*}
\item 
\begin{equation*}
\begin{split}
\overset{\ast}{R}_{mi}+\overset{\ast}{R}_{im}&=\overset{(0)}{R}_{mi}+\overset{(0)}{R}_{im}\\
&
-\frac{1}{4}
(K^l_{\ ij}K^j_{\ lm}+K^l_{\ ji}K^j_{\ ml})+\frac{1}{2}\left\{\overset{(0)}{T}\,^l_{\ ji}K^j_{\ ml}+\overset{(0)}{T}\,^l_{\ jm}K^j_{\ il}\right\}.\\
\end{split}
\end{equation*}
\end{enumerate}
\end{corollary}

\bigskip

We will study now the symmetry of the Ricci curvature tensors for an affine connection $\nabla$ with torsion $T$ using the Second Bianchi identity. Indeed, if we take $X,Y,Z$ to be $\dfrac{\partial}{\partial x^i},\dfrac{\partial}{\partial x^j},\dfrac{\partial}{\partial x^k}$, respectively, then
\begin{equation}\label{eq_B1}
\begin{split}
&R(X,Y)Z+R(Y,Z)X+R(Z,X)Y\\
&=R\left(\frac{\partial}{\partial x^i},\frac{\partial}{\partial x^j}\right)\frac{\partial}{\partial x^k}
+R\left(\frac{\partial}{\partial x^j},\frac{\partial}{\partial x^k}\right)\frac{\partial}{\partial x^i}
+R\left(\frac{\partial}{\partial x^k},\frac{\partial}{\partial x^i}\right)\frac{\partial}{\partial x^j}\\
&=\left(R^{\ l}_{k\ ij}+R^{\ l}_{i\ jk}+R^{\ l}_{j\ ki}\right)\frac{\partial}{\partial x^l}.
\end{split}
\end{equation}

Likewise
\begin{equation}\label{eq_B2}
\begin{split}
&T(T(X,Y)Z)+T(T(Y,Z),X)+T(T(Z,X),Y)\\
&=T\left(\left(\frac{\partial}{\partial x^i},\frac{\partial}{\partial x^j}\right),\frac{\partial}{\partial x^k}\right)
+T\left(\left(\frac{\partial}{\partial x^j},\frac{\partial}{\partial x^k}\right),\frac{\partial}{\partial x^i}\right)
+T\left(\left(\frac{\partial}{\partial x^k},\frac{\partial}{\partial x^i}\right),\frac{\partial}{\partial x^j}\right)\\
&=\left(T^m_{\ ij}T^l_{\ mk}+T^m_{\ jk}T^l_{\ mi}+T^m_{\ ki}T^l_{\ mj}\right)\frac{\partial}{\partial x^l}.
\end{split}
\end{equation}

Finally, we compute
\begin{equation}\label{eq_B3}
\begin{split}
&(\nabla_XT)(Y,Z)+(\nabla_YT)(Z,X)+(\nabla_ZT)(X,Y)\\
&=\left(\nabla_{\dfrac{\partial}{\partial x^i}}T\right)\left(\frac{\partial}{\partial x^j},\frac{\partial}{\partial x^k}\right)
+\left(\nabla_{\dfrac{\partial}{\partial x^j}}T\right)\left(\frac{\partial}{\partial x^k},\frac{\partial}{\partial x^i}\right)
+\left(\nabla_{\dfrac{\partial}{\partial x^k}}T\right)\left(\frac{\partial}{\partial x^i},\frac{\partial}{\partial x^j}\right)\\
&=\left(\frac{\partial T^l_{\ jk}}{\partial x^i}+\Gamma^l_{\ mi}T^m_{\ jk}-\Gamma^m_{\ ji}T^l_{\ mk}-\Gamma^m_{\ ki}T^l_{\ jm}\right)\frac{\partial}{\partial x^l}\\
&+\left(\frac{\partial T^l_{\ ki}}{\partial x^j}+\Gamma^l_{\ mj}T^m_{\ ki}-\Gamma^m_{\ kj}T^l_{\ mi}-\Gamma^m_{\ ij}T^l_{\ km}\right)\frac{\partial}{\partial x^l}\\
&+\left(\frac{\partial T^l_{\ ij}}{\partial x^k}+\Gamma^l_{\ mk}T^m_{\ ij}-\Gamma^m_{\ ik}T^l_{\ mj}-\Gamma^m_{\ jk}T^l_{\ im}\right)\frac{\partial}{\partial x^l}.
\end{split}
\end{equation}

Let us observe that the properties of the curvature tensor $R^{\ i}_{j\ kl}$ give the following alternative form of the trace $l=i$ of \eqref{eq_B1}.

\begin{equation*}
\begin{split}
R^{\ i}_{k\ ij}+R^{\ i}_{i\ jk}+R^{\ i}_{j\ ki}&=
R_{kj}+g^{is}R_{isjk}-R^{\ i}_{j\ ik}
=R_{kj}-R_{jk}+g^{is}R_{isjk},
\end{split}
\end{equation*}
where we use property (i) in Theorem \ref{thm_A10}, namely $R^{\ i}_{j\ ki}+R^{\ i}_{j\ ik}=0$. Observe that since $\nabla$ is not metrical, the skew symmetry in the first two indices of $R_{iskj}$ does not hold, see formula \eqref{eq_P4}.

It follows
\begin{equation}\label{eq_Ricc1}
\begin{split}
R_{kj}-R_{jk}+g^{is}R_{isjk}=(T^m_{\ ij}T^i_{\ mk}+T^m_{\ jk}T^i_{\ mi}+T^m_{\ ki}T^i_{\ mj})
+T^{\ i}_{j\ k|i}+T^{\ i}_{k\ i|j}+T^{\ i}_{i\ j|k},
\end{split}
\end{equation}
where we take the trace in \eqref{eq_B2} and \eqref{eq_B3}. Obviously, the last 3 terms can be further computed using last equality in \eqref{eq_B3}, but we do not need that degree of detail in the present exposition. 

In the case $\nabla$ is torsion free, we get
\begin{equation}\label{eq_Ricc2}
R_{kj}-R_{jk}+g^{is}R_{isjk}=0.
\end{equation}

We recall that the $g$-dual connections $\nabla$ and $\overset{\ast}{\nabla}$ are called {\it curvature conjugate symmetric connection} if they have the same curvature tensor, i.e.
\begin{equation}\label{eq_13.10.1}
R(X,Y)Z=\overset{\ast}{R}(X,Y)Z
\end{equation}
for any $X,Y,Z\in\mathcal{X}(M)$. Locally, this means
\begin{equation*}
R^{\ i}_{j\ kl}=(\overset{\ast}{R}\,^{\ i}_{j\ kl})^{\ast}.
\end{equation*}

From Theorem \ref{thm: R and R* 1-st 2 ind}, we obtain

\begin{corollary}\label{cor_13.10}
Let $\nabla$ and $\overset{\ast}{\nabla}$ be $g$-dual connections on a (pseudo-)Riemannian manifold $(M,g)$. If $\nabla$ and $\overset{\ast}{\nabla}$ are curvature conjugate symmetric connections, then 
\begin{enumerate}[(i)]
\item
$g(R(X,Y)W,Z)+g(R(X,Y)Z,W)=0$,
\item $g(\overset{\ast}{R}(X,Y)W,Z)+g(\overset{\ast}{R}(X,Y)Z,W)=0$. 
\end{enumerate}
\end{corollary}

\begin{proof}
Indeed, if $\nabla$ and $\overset{\ast}{\nabla}$ are curvature conjugate symmetric, then formula \eqref{eq_13.10.1} implies $g(R(X,Y)Z,W)-g(\overset{\ast}{R}(X,Y)Z,W)=0$.

On the other hand, taking into account of \eqref{eq_12.1}, by subtracting theses formulas, we get
$$
g(\overset{\ast}{R}(X,Y)W,Z)+g(\overset{\ast}{R}(X,Y)Z,W)=0.
$$

Likewise, if we write \eqref{eq_13.10.1} as $R(X,Y)W=\overset{\ast}{R}(X,Y)W$, we get $g(R(X,Y)W,Z)-g(\overset{\ast}{R}(X,Y)W,Z)=0$ and by adding with \eqref{eq_12.1} it follows $g(R(X,Y)W,Z)+g(R(X,Y)Z,W)=0$.
\end{proof}

Locally, Corollary \ref{cor_13.10} states that if $\nabla$ and $\overset{\ast}{\nabla}$ are curvature conjugate symmetric, then
\begin{enumerate}
\item $
R_{iskj}+R_{sikj}=0,
$
\item $
\overset{\ast}{R}_{iskj}+\overset{\ast}{R}_{sikj}=0.
$
\end{enumerate}

Therefore, if we return to \eqref{eq_Ricc2} it follows.
\begin{proposition}
Let $\nabla$ and $\overset{\ast}{\nabla}$ be dual connection on the (pseudo-)Riemannian manifold $(M,g)$.

If $\nabla$ is torsion free and $\nabla$, $\overset{\ast}{\nabla}$ are curvature conjugate symmetric, then the Ricci curvature tensor $R(X,Y)$ is symmetric.
\end{proposition}

\section{Equiaffine connections}\label{sect2}

Let us consider the dual connections $\nabla$ and $\overset{\ast}{\nabla}$ on the (pseudo-)Riemannian manifold $(M,g)$.

\subsection{Definition of the equiaffine connection}

Recall that
\begin{equation}\label{eq_11.1}
\omega_g:=\sqrt{|g|}dx^1\wedge\ldots\wedge dx^n
\end{equation}
is the volume form of $(M,g)$, where $|g|$ is the absolute value of the determinant of $g$.

If $X=X^i\frac{\partial}{\partial x^i}$ is a vector field on $M$, then a straightforward computation gives
\begin{equation}\label{eq_11.2}
\begin{split}
\nabla_X\omega_g&=\nabla_X(\sqrt{|g|}dx^1\wedge\ldots\wedge dx^n)\\
&=\nabla_X(\sqrt{|g|})dx^1\wedge\ldots\wedge dx^n+\sqrt{|g|}\nabla_X(dx^1\wedge\ldots\wedge dx^n)\\
&=X^i\frac{\partial\sqrt{|g|}}{\partial x^i}dx^1\wedge\ldots\wedge dx^n+\sqrt{|g|}X^i\nabla_{\dfrac{\partial}{\partial x^i}}(dx^1\wedge\ldots\wedge dx^n).
\end{split}
\end{equation}

Let us recall that
\begin{equation}\label{eq_11.3}
\frac{\partial\sqrt{|g|}}{\partial x^i}=\frac{1}{2}\sqrt{|g|}g^{pq}\frac{\partial g_{pq}}{\partial x^i},
\end{equation}
(see \cite{BCS} for the Riemannian case and \cite{N} for the (pseudo-)Riemannian case).

Next, we  will show 
\begin{lemma}\label{lem: 2.4}
If $\nabla$ is an affine connection on $M$ with local coefficients $\Gamma$, then 
\begin{equation}\label{eq_11.4}
\nabla_{\dfrac{\partial}{\partial x^i}}(dx^1\wedge\ldots\wedge dx^n)=-(Tr_2\Gamma_i) dx^1\wedge\ldots\wedge dx^n,
\end{equation}
where $Tr_2\Gamma_i=\sum_{k=1}^n\Gamma^k_{\ ki}$ is the left trace of $\Gamma$.
\end{lemma}

Indeed, we compute
\begin{equation*}
\begin{split}
&\left(\nabla_{\dfrac{\partial}{\partial x^i}}dx^1\wedge\ldots\wedge dx^n\right)\left(\frac{\partial}{\partial x^1},\ldots,\frac{\partial}{\partial x^n}\right)\\
&=\nabla_{\dfrac{\partial}{\partial x^i}}\left(dx^1\wedge\ldots\wedge dx^n\left(\frac{\partial}{\partial x^1},\ldots\frac{\partial}{\partial x^n}\right)\right)-\sum_{j=1}^ndx^1\wedge\ldots\wedge dx^n\left(\dfrac{\partial}{\partial x^1},\ldots,\nabla_{\dfrac{\partial}{\partial x^i}}\dfrac{\partial}{\partial x^j},\ldots,\dfrac{\partial}{\partial x^n}\right).
\end{split}
\end{equation*}

Taking into account that
\begin{equation*}
\begin{split}
dx^1\wedge\ldots\wedge dx^n\left(\frac{\partial}{\partial x^1},\ldots,\frac{\partial}{\partial x^n}\right)
&=
\begin{vmatrix}
dx^1\left(\dfrac{\partial}{\partial x^1}\right) & dx^1\left(\dfrac{\partial}{\partial x^2}\right) & \ldots & dx^1\left(\dfrac{\partial}{\partial x^n}\right)\\
dx^2\left(\dfrac{\partial}{\partial x^1}\right) & dx^2\left(\dfrac{\partial}{\partial x^2}\right) & \ldots & dx^2\left(\dfrac{\partial}{\partial x^n}\right)\\
\vdots & \vdots & \ddots & \vdots \\
dx^n\left(\dfrac{\partial}{\partial x^1}\right) & dx^n\left(\dfrac{\partial}{\partial x^2}\right) & \ldots & dx^n\left(\dfrac{\partial}{\partial x^n}\right)\\
\end{vmatrix}\\
&=\det E_n=1,
\end{split}
\end{equation*}
it results that
$$
\nabla_{\dfrac{\partial}{\partial x^i}}\left(dx^1\wedge\ldots\wedge dx^n\left(\frac{\partial}{\partial x^1,\ldots,\frac{\partial}{\partial x^n}}\right)\right)=0.
$$

Next, we compute
\begin{equation*}
\begin{split}
&-\sum_{j=1}^ndx^1\wedge\ldots\wedge dx^n\left(\frac{\partial}{\partial x^1},\ldots,\nabla_{\dfrac{\partial}{\partial x^i}}\frac{\partial}{\partial x^j},\ldots,\frac{\partial}{\partial x^n}\right)\\
&-\Bigg\{
dx^1\wedge\ldots\wedge dx^n\left(\nabla_{\dfrac{\partial}{\partial x^i}}\frac{\partial}{\partial x^1},\ldots,\frac{\partial}{\partial x^n}\right)
+dx^1\wedge\ldots\wedge dx^n\left(
\frac{\partial}{\partial x^1},\nabla_{\dfrac{\partial}{\partial x^i}}\frac{\partial}{\partial x^2},\ldots,\frac{\partial}{\partial x^n}
\right)\\
&+\ldots+dx^1\wedge\ldots\wedge dx^n\left(
\frac{\partial}{\partial x^1},\ldots,\nabla_{\dfrac{\partial}{\partial x^i}}\frac{\partial}{\partial x^n}
\right)
\Bigg\}\\
&=-\Bigg\{
\Gamma^k_{\ 1i}dx^1\wedge\ldots\wedge dx^n\left(\frac{\partial}{\partial x^k},\ldots,\frac{\partial}{\partial x^n}\right)+\Gamma^k_{\ 2i}dx^1\wedge\ldots\wedge dx^n\left(\frac{\partial}{\partial x^1},\frac{\partial}{\partial x^k},\ldots,\frac{\partial}{\partial x^n}\right)\\
&+\ldots+\Gamma^k_{\ ni}dx^1\wedge\ldots\wedge dx^n\left(\frac{\partial}{\partial x^1},\ldots,\frac{\partial}{\partial x^k}\right)
\Bigg\}=-\sum_{k=1}^n\Gamma^k_{\ ki}=-Tr_2\Gamma_i,
\end{split}
\end{equation*}
and therefore \eqref{eq_11.4} is proved.

If we return to \eqref{eq_11.2}, we get
\begin{equation*}
\begin{split}
  \nabla_X\omega_g&=X^i\frac{\partial\sqrt{|g|}}{\partial x^i}dx^1\wedge\ldots\wedge dx^n+
                    \sqrt{|g|}
                    X^i\nabla_{\dfrac{\partial}{\partial x^i}}(dx^1\wedge\ldots\wedge dx^n)\\
&=X^i\left(\frac{\partial\sqrt{|g|} }{\partial x^i}-\sqrt{|g|}\sum_{k=1}^n\Gamma^k_{\ ki}\right)dx^1\wedge\ldots\wedge dx^n.
\end{split}
\end{equation*}

It results
\begin{equation}
  \begin{split}
  \nabla_X\omega_g&=\frac{1}{\sqrt{|g|}}\left[X(\sqrt{|g|})-\sqrt{|g|} \cdot Tr_2(\Gamma)(X)\right]\cdot \omega_g\\
    &=\left[X(\log \sqrt{|g|})-Tr_2(\Gamma)(X)\right]\cdot \omega_g.
      \end{split}
\end{equation}

Locally, we have
\begin{equation}\label{eq_11.5}
\begin{split}
\nabla_X\omega_g&=\frac{1}{2}\sqrt{|g|}g^{pq}X^i\left(\frac{\partial}{\partial x^i}g_{pq}\right)dx^1\wedge\ldots\wedge dx^n\\
&-\sqrt{|g|}X^i\left(Tr_2\Gamma_i\right)dx^1\wedge\ldots\wedge dx^n\\
&=\sqrt{|g|}X^i\left\{
\frac{1}{2}g^{pq}\frac{\partial g_{pq}}{\partial x^i}-Tr_2\Gamma_i
\right\}dx^1\wedge\ldots\wedge dx^n.
\end{split}
\end{equation}

Let us observe that the quantity in parenthesis can be written as 
\begin{equation}\label{eq_11.6}
\frac{1}{2}g^{pq}\frac{\partial g_{pq}}{\partial x^i}-\sum_{k=1}^n\Gamma^k_{\ ki}=\frac{1}{2}g^{pq}\left(\nabla_{\dfrac{\partial}{\partial x^i}}g\right)\left(\frac{\partial}{\partial x^p},\frac{\partial}{\partial x^q}\right).
\end{equation}

Indeed, if we start from the right hand side, we get
\begin{equation*}
\begin{split}
\frac{1}{2}g^{pq}\left(\nabla_{\dfrac{\partial}{\partial x^i}}g\right)\left(\frac{\partial}{\partial x^p},\frac{\partial}{\partial x^q}\right)
&=\frac{1}{2}g^{pq}\left\{
\frac{\partial g_{pq}}{\partial x^i}-g_{mp}\Gamma^m_{\ qi}-g_{mq}\Gamma^m_{\ pi}
\right\}\\
&=\frac{1}{2}g^{pq}\frac{\partial g_{pq}}{\partial x^i}-\frac{1}{2}g^{pq}\left\{
g_{mp}\Gamma^m_{\ qi}+g_{mq}\Gamma^m_{\ pi}
\right\}
\end{split}
\end{equation*}
and by paying attention to the summation indices, we obtain \eqref{eq_11.6} immediately. 

By substituting this relation, \eqref{eq_11.5} becomes
\begin{equation*}
\begin{split}
\nabla_X\omega_g&=\sqrt{|g|}X^i\frac{1}{2}g^{pq}\left(\nabla_{\dfrac{\partial}{\partial x^i}}g\right)\left(\frac{\partial}{\partial x^p},\frac{\partial}{\partial x^q}\right)dx^1\wedge\ldots\wedge dx^n\\
&=\frac{1}{2}g^{pq}(\nabla_Xg)\left(\frac{\partial}{\partial x^p},\frac{\partial}{\partial x^q}\right)\omega_g=\frac{1}{2}trace_g(\nabla_Xg)(\cdot,\cdot)\omega_g,
\end{split}
\end{equation*}
where we use the notation
$$
trace_g(\nabla_Xg)(\cdot,\cdot):=g^{pq}(\nabla_Xg)\left(\frac{\partial}{\partial x^p},\frac{\partial}{\partial x^q}\right)=X^ig^{pq}(g_{pq|i}).
$$

In conclusion, we obtain 
\begin{proposition}
The covariant derivative of the volume form $\omega_g$ with respect to the affine connection $\nabla$ is 
$$
\nabla_X\omega_g=\frac{1}{2}trace_g(\nabla_Xg)(\cdot,\cdot)\omega_g.
$$
\end{proposition}

Next, let us observe that using \eqref{eq_A1}, this relation can be written in terms of the cubic tensor $C$ as follows
$$
\nabla_X\omega_g=\frac{1}{2}g^{pq}C\left(X,\frac{\partial}{\partial x^p},\frac{\partial}{\partial x^q}\right)\omega_g
$$
and using \eqref{eq_Rel5} we get
\begin{equation*}
\begin{split}
\nabla_X\omega_g&=\frac{1}{2}g^{pq}g\left(K\left(X,\frac{\partial}{\partial x^p}\right),\frac{\partial}{\partial x^q}\right)\omega_g\\
&=\frac{1}{2}g^{pq}g\left(K_X\frac{\partial}{\partial x^p},\frac{\partial}{\partial x^q}\right)\omega_g=\frac{1}{2}Tr_2(K)(X)\omega_g,
\end{split}
\end{equation*}
where we use the notation $Tr_2$ for the left trace, i.e. 
$$
Tr_2(K)(X)=g^{pq}g\left(K_X\frac{\partial}{\partial x^p},\frac{\partial}{\partial x^q}\right)=
X^i g^{pq}g\left(K_{\frac{\partial}{\partial x^i}}\frac{\partial}{\partial x^p},\frac{\partial}{\partial x^q}\right)=X^iK^k_{\ ki}.
$$

It results that 
\begin{equation}\label{nabla Tr2 K}
\nabla_X\omega_g=\frac{1}{2}Tr_2(K)(X)\omega_g.
\end{equation}

\begin{remark}
Taking into account formula \eqref{eq_Rel5}, it is easy to see that 
\begin{equation*}
trace_g(\nabla_Xg)(\cdot,\cdot)=\frac{1}{2}Tr_2(K)(X)
\end{equation*}
which motives the relation above in a more direct way. 
\end{remark}
\begin{theorem}\label{thm_E1}
  If $\nabla$ is an affine connection on a (pseudo-)Riemannian manifold $(M,g)$, then
  the following conditions are equivalent
  \begin{enumerate}[(i)]
  \item $\nabla_X\omega_g=0$;
  \item $Tr_2(\Gamma)(X)=X(\log \sqrt{|g|})$;
  \item $Tr_2(K)(X)=0$,
    \end{enumerate}
     for any tangent vector $X\in \mathcal{X}(M)$. 
\end{theorem}

In particular, by taking $X=\dfrac{\partial}{\partial x^i}$ it follows that $\nabla_{\dfrac{\partial}{\partial x^i}}\omega_g=0$ is equivalent to 
\begin{equation*}
\begin{split}
g^{pq}g\left(K\left(\frac{\partial}{\partial x^i},\frac{\partial}{\partial x^p}\right),\frac{\partial}{\partial x^q}\right)
&=g^{pq}g\left(K^k_{\ pi}\frac{\partial}{\partial x^k},\frac{\partial}{\partial x^q}\right)\\
&=g^{pq}g_{kq}K^k_{\ pi}=K^k_{\ ki}=Tr_2K_i.
\end{split}
\end{equation*}

We obtain
\begin{corollary}\label{cor_E2}
  If $\nabla$ is an affine connection on a (pseudo-)Riemannian manifold $(M,g)$, then
   the following conditions are equivalent
  \begin{enumerate}[(i)]
  \item $\nabla_{\dfrac{\partial}{\partial x^i}}\omega_g=0$;
  \item $Tr_2(\Gamma)_i=\Gamma^k_{\ ki}=\dfrac{\partial \log \sqrt{|g|}}{\partial x^i}$;
  \item $Tr_2(K)_i=K^k_{\ ki}=0.$
    \end{enumerate}
\end{corollary}

Let us recall that an affine connection $\nabla$ is called {\it locally equiaffine} if there exists a volume form $\omega$ at any point of $M$ such that $\nabla_X\omega=0$, for any $X\in\mathcal{X}(M)$.

In the case $\omega$ is globally defined on $M$ such that $\nabla_X\omega=0$, then $\nabla$ is called {\it equiaffine}.

In particular, if $\nabla$ is equiaffine connection with the associated volume form $\omega_g$, where $\omega_g$ is the volume form of $g$, i.e. $\nabla_X\omega_g=0$, then clearly the left trace $Tr_2K_i=0$, in this sense $\nabla$ is called {\it trace-free}.

However, observe that in general $K$ is not symmetric. More generally, we can define
$$
\omega:=\lambda(x)dx^1\wedge\ldots\wedge dx^n,
$$
where $\lambda(x)$ is a no-vanishing positive smooth function on the manifold $M$. Then, in the same manner as above, we have
\begin{equation*}
\begin{split}
\nabla_X\omega&=X^i\frac{\partial\lambda}{\partial x^i}dx^1\wedge\ldots\wedge dx^n+\lambda(x)X^i\nabla_{\dfrac{\partial}{\partial x^i}}(dx^1\wedge\ldots\wedge dx^n)\\
&=X^i\left(\frac{\partial\lambda}{\partial x^i}-\lambda\sum_{k=1}^n\Gamma^k_{\ ki}\right)dx^1\wedge\ldots\wedge dx
\end{split}
\end{equation*}

Therefore, we can write 
\begin{equation*}
\nabla_X\omega=\frac{1}{\sqrt{|g|}}\left[X(\lambda)-\lambda \cdot Tr_2(\Gamma)(X)\right]\cdot \omega_g.
\end{equation*}

Moreover, taking into account that $\omega_g=\dfrac{\sqrt{|g|}}{\lambda}$, formula \eqref{nabla Tr2 K} implies
\begin{equation*}
  \nabla_X\left( \frac{\sqrt{|g|}}{\lambda} \omega \right)=\frac{1}{2}Tr_2(K)(X)\frac{\sqrt{|g|}}{\lambda}\omega,
\end{equation*}
and by a straighforward computation, we obtain
\begin{equation*}
  \nabla_X\omega=\left[ \frac{1}{2}Tr_2(K)(X)-X\left(\log \frac{\sqrt{|g|}}{\lambda}
      \right)
    \right]\omega.
\end{equation*}

Summarizing, we obtain 
\begin{theorem}
If $\nabla$ is an affine connection on the (pseudo)-Riemannian manifold $(M,g)$ and let $\omega:=\lambda(x)dx^1\wedge\ldots\wedge dx^n,
$ be an arbitrary volume form on $M$, 
where $\lambda(x)$ is a no-vanishing positive smooth function. 

Then the following conditions are equivalent
\begin{enumerate}[(i)]
\item $\nabla$ is an equiaffine connection, i.e.
  $\nabla_X\omega=0$;
\item $Tr_2(K)(X)=2X\left(\log \frac{\sqrt{|g|}}{\lambda}\right)$;
  \item $X(\log\lambda(x))= Tr_2(\Gamma)(X)$,
  \end{enumerate}
for any vector field $X\in \mathcal{X}(M)$.
\end{theorem}

By taking $X=\dfrac{\partial}{\partial x^i}$ we obtain
\begin{corollary}
If $\nabla$ is an affine connection on the (pseudo)-Riemannian manifold $(M,g)$ and let $\omega:=\lambda(x)dx^1\wedge\ldots\wedge dx^n,
$ be an arbitrary volume form on $M$, 
where $\lambda(x)$ is a no-vanishing positive smooth function. 

Then the following conditions are equivalent
\begin{enumerate}[(i)]
\item $\nabla_{\dfrac{\partial}{\partial x^i}}\omega=0$;
  \item  $Tr_2(K)_i=2\frac{\partial}{\partial x^i}\left(\log \frac{\sqrt{|g|}}{\lambda}\right)$;
\item $\frac{\partial}{\partial x^i}\log\lambda(x)=\sum_{k=1}^n\Gamma^k_{\ ki}=Tr_2\Gamma_i.$
  \end{enumerate}
  \end{corollary}

\begin{remark}
In the case $\nabla$ has non-vanishing torsion $T$, then the condition (iii) above can be written in the equivalent form
\begin{equation}\label{equiaffine Gamma right trace cond}
\frac{\partial}{\partial x^i}\log \lambda(x)=\sum_{k=1}^n\Gamma^k_{\ ik}+\sum_{k=1}^n T^k_{\ ik}.
\end{equation}
\end{remark}

\subsection{The Ricci curvature tensors}

Let us turn now our attention to Ricci curvature tensors. Using \eqref{eq_Rel10} we have
$$
R^{\ k}_{i\ kj}=\frac{\partial\Gamma^k_{\ ij}}{\partial x^k}-\frac{\partial\Gamma^k_{\ ik}}{\partial x^j}+\Gamma^h_{\ ij}\Gamma^k_{\ hk}-\Gamma^h_{\ ik}\Gamma^k_{\ hj},
$$
hence
\begin{equation*}
\begin{split}
R_{ij}-R_{ji}&=R^{\ k}_{i\ kj}-R^{\ k}_{j\ ki}\\
&=\frac{\partial\Gamma^k_{\ ij}}{\partial x^k}-\frac{\partial\Gamma^k_{\ ik}}{\partial x^j}+\Gamma^h_{\ ij}\Gamma^k_{\ hk}-\Gamma^h_{\ ik}\Gamma^k_{\ hj}\\
&-\frac{\partial\Gamma^k_{\ ji}}{\partial x^k}+\frac{\partial\Gamma^k_{\ jk}}{\partial x^i}+\Gamma^h_{\ ji}\Gamma^k_{\ hk}-\Gamma^k_{\ jh}\Gamma^h_{\ ki}\\
&=\frac{\partial}{\partial x^k}(\Gamma^k_{\ ij}-\Gamma^k_{\ ji})+(\Gamma^h_{\ ij}-\Gamma^h_{\ ji})\Gamma^k_{\ hk}+(\Gamma^h_{ki}\Gamma^k_{\ jh}-\Gamma^h_{\ ik}\Gamma^k_{\ hj})
+\left(\frac{\partial\Gamma^k_{\ jk}}{\partial x^i}-\frac{\partial\Gamma^k_{\ ik}}{\partial x^j}\right).
\end{split}
\end{equation*}

Let us observe that
\begin{equation*}
\begin{split}
\Gamma^h_{\ ki}\Gamma^k_{\ jh}-\Gamma^h_{\ ik}\Gamma^k_{\ hj}&=(T^h_{\ ik}+\Gamma^h_{\ ik})(T^k_{\ jh}+\Gamma^k_{\ hj})-\Gamma^h_{\ ik}\Gamma^k_{\ hj}\\
&=T^h_{\ ik}T^k_{\ hj}+\Gamma^k_{\ hj}T^h_{\ ik}+\Gamma^h_{\ ik}T^k_{\ hj},
\end{split}
\end{equation*}
hence
\begin{equation*}
\begin{split}
R_{ij}-R_{ji}&=\underset{\text{\textcircled{\tiny 1}}}{
\frac{\partial}{\partial x^k}(T^k_{\ ji})}
+\underset{\text{\textcircled{\tiny 2}}}{\Gamma^k_{\ hk}T^h_{\ ji}}+\Gamma^k_{\ hj}T^h_{\ ik}+\underset{\text{\textcircled{\tiny 3}}}{\Gamma^h_{\ ik}T^k_{\ hj}}+T^h_{\ ik}T^k_{\ hj}
+\left(\frac{\partial\Gamma^k_{\ jk}}{\partial x^i}-\frac{\partial\Gamma^k_{\ ik}}{\partial x^j}\right).
\end{split}
\end{equation*}

We recall that
\begin{equation*}
\begin{split}
\nabla_kT^l_{\ ji}&=\frac{\partial T^l_{\ ji}}{\partial x^k}+\Gamma^l_{\ hk}T^h_{\ ji}-\Gamma^h_{\ jk}T^l_{\ ki}-\Gamma^h_{\ ik}T^l_{\ jh}\\
&=\underset{\text{\textcircled{\tiny 1}}}{\frac{\partial T^l_{\ ji}}{\partial x^k}}+\underset{\text{\textcircled{\tiny 2}}}{\Gamma^l_{\ hk}T^h_{\ ji}}+\Gamma^h_{\ jk}T^l_{\ ik}+\underset{\text{\textcircled{\tiny 3}}}{\Gamma^h_{\ ik}T^l_{\ hj}},
\end{split}
\end{equation*}
where we use that $T$ is skew symmetric in the lower indices. By taking the trace $l\to k$, we obtain
\begin{equation*}
\begin{split}
R_{ij}-R_{ji}&=\nabla_kT^k_{\ ji}-\Gamma^k_{\ jh}T^h_{\ ik}
+\Gamma^k_{\ hj}T^h_{\ ik}+T^h_{\ ik}T^k_{\ hj}
+\left(\frac{\partial\Gamma^k_{\ jk}}{\partial x^i}-\frac{\partial\Gamma^k_{\ ik}}{\partial x^j}\right)\\
&=\nabla_kT^k_{\ ji}+(\Gamma^k_{\ hj}-\Gamma^k_{\ jh})T^h_{\ ik}
+T^h_{\ ik}T^k_{\ hj}
+\left(\frac{\partial\Gamma^k_{\ jk}}{\partial x^i}-\frac{\partial\Gamma^k_{\ ik}}{\partial x^j}\right)\\
&=\nabla_kT^k_{\ ji}+T^k_{\ jh}T^h_{\ ik}
+T^h_{\ ik}T^k_{\ hj}
+\left(\frac{\partial\Gamma^k_{\ jk}}{\partial x^i}-\frac{\partial\Gamma^k_{\ ik}}{\partial x^j}\right)\\
&=\nabla_kT^k_{\ ji}
+\left(\frac{\partial\Gamma^k_{\ jk}}{\partial x^i}-\frac{\partial\Gamma^k_{\ ik}}{\partial x^j}\right).
\end{split}
\end{equation*}

We obtain
\begin{theorem}\label{thm_A13.19}
If $\nabla$ is an affine connection with torsion $T$, then
$$
R_{ij}-R_{ji}=\nabla_kT^k_{\ ji}
+\left(\frac{\partial\Gamma^k_{\ jk}}{\partial x^i}-\frac{\partial\Gamma^k_{\ ik}}{\partial x^j}\right).
$$
\end{theorem}

\begin{corollary}\label{cor: Ricci diff by torsion}
If $\nabla$ is an equiaffine connection with volume form 
$\omega=\lambda dx^1\wedge \dots \wedge dx^n$, $\lambda>0$, then  
$$
R_{ij}-R_{ji}=\nabla_kT^k_{\ ji}-\left(\frac{\partial T^k_{\ jk}}{\partial x^i}-\frac{\partial T^k_{\ ik}}{\partial x^j}\right).
$$
\end{corollary}
Indeed, from \eqref{equiaffine Gamma right trace cond} we get 
\begin{equation*}
\left(\frac{\partial\Gamma^k_{\ jk}}{\partial x^i}-\frac{\partial\Gamma^k_{\ ik}}{\partial x^j}\right)+\left(\frac{\partial T^k_{\ jk}}{\partial x^i}-\frac{\partial T^k_{\ ik}}{\partial x^j}\right)=0,
\end{equation*}
hence the conclusion follows from Theorem \ref{thm_A13.19}.

\begin{corollary}
If $\nabla$ is torsion free and $Tr_1(\Gamma_i)=\Gamma^k_{\ ik}=0$, then the Ricci curvature tensors are symmetric, i.e.
$
R_{ij}=R_{ji}.
$
\end{corollary}

\begin{corollary}\label{cor_A13.21}
Let $\nabla$ be a torsion free affine connection. If $\nabla$ is equiaffine with volume form 
$\omega=\lambda dx^1\wedge \dots \wedge dx^n$, $\lambda>0$, then the Ricci curvature tensor are symmetric.
\end{corollary}
It follows directly from Corollary \ref{cor: Ricci diff by torsion}.

\begin{lemma}
Let $\nabla $ and $\bar{\nabla}$ be two equiaffine connections with
local coefficients $\Gamma $ and $\overset{\ast}{\Gamma}$ and corresponding
parallel volumes
\begin{equation*}
\omega :=\lambda (x)\cdot dx^{1}\wedge dx^{2}\wedge ...\wedge dx^{n},
\end{equation*}
\begin{equation*}
\bar{\omega}:=\bar{\lambda}(x)\cdot dx^{1}\wedge dx^{2}\wedge ...\wedge
dx^{n}.
\end{equation*}

Then the connection $a\nabla +b$ $\bar{\nabla}$ is equiaffine with parallel
volume $\lambda ^{a}(x)\cdot \bar{\lambda}^{b}(x)dx^{1}\wedge dx^{2}\wedge
...\wedge dx^{n}$.
\end{lemma}

\begin{proof}
If $\nabla $ and $\bar{\nabla}$ be two equiaffine connections this is
equivalent to

\begin{equation*}
\frac{\partial }{\partial x^{i}}\log \lambda (x)=\Gamma _{ki}^{k},
\end{equation*}

\begin{equation*}
\frac{\partial }{\partial x^{i}}\log \bar{\lambda}(x)=\bar{\Gamma}_{ki}^{k}.
\end{equation*}

Therefore by multiplying with $a$ and $b$ and adding we find

\begin{equation*}
a\frac{\partial }{\partial x^{i}}\log \lambda (x)+b\frac{\partial }{\partial
x^{i}}\log \bar{\lambda}(x)=a\Gamma _{ki}^{k}+b\bar{\Gamma}_{ki}^{k},
\end{equation*}%
i.e.,

\begin{equation*}
\frac{\partial }{\partial x^{i}}\log \lambda ^{a}(x)\bar{\lambda}%
^{b}(x)=a\Gamma _{ki}^{k}+b\bar{\Gamma}_{ki}^{k},
\end{equation*}

and the Lemma is proved.
\end{proof}

\begin{proposition}\label{prop: equiaffine dual conn}
If $\nabla $ and $\overset{\ast}{\nabla}$ are dual affine connections
on $\left( M,g\right) $ then $\nabla $ is equiaffine with parallel volume $%
\omega =\lambda (x)\cdot dx^{1}\wedge dx^{2}\wedge ...\wedge dx^{n}$ if and
only if $\overset{\ast}{\nabla}$ is equiaffine with parallel volume $\overset{\ast}{\omega}=\frac{|g|}{\lambda }dx^{1}\wedge dx^{2}\wedge ...\wedge dx^{n}$, where $%
g=\det \left( g_{ij}\right) $.
\end{proposition}

\begin{proof}
Locally, the dual connection coefficients are written as

\begin{equation*}
\frac{\partial g_{ij}}{\partial x^{k}}-g_{mj}\Gamma _{ik}^{m}-g_{mi}\overset{\ast}{\Gamma}\,
_{jk}^{m}=0.
\end{equation*}

By multiplication with $g^{il}$ we get

\begin{equation*}
g^{il}\frac{\partial g_{ij}}{\partial x^{k}}-g^{il}g_{mj}\Gamma
_{ik}^{m}-\overset{\ast}{\Gamma}\,_{jk}^{l}=0.
\end{equation*}

We take the trace of $j$ with $l$

\begin{equation*}
g^{ij}\frac{\partial g_{ij}}{\partial x^{k}}-g^{ij}g_{mj}\Gamma
_{ik}^{m}-\overset{\ast}{\Gamma}\,_{jk}^{j}=0.
\end{equation*}

Hence

\begin{equation*}
g^{ij}\frac{\partial g_{ij}}{\partial x^{k}}-\Gamma _{ik}^{i}-\overset{\ast}{\Gamma}
\,_{jk}^{j}=0,
\end{equation*}%
and therefore

\begin{equation}
\overset{\ast}{\Gamma}\,_{jk}^{j}=g^{ij}\frac{\partial g_{ij}}{\partial x^{k}}-\Gamma
_{ik}^{i}=\frac{\partial }{\partial x^{k}}\log |g|-\Gamma _{ik}^{i},
\label{23}
\end{equation}%
where we have used the well-known formula $g^{ij}\frac{\partial g_{ij}}{%
\partial x^{k}}=\frac{\partial }{\partial x^{k}}\log |g|$ from Riemannian
geometry.

Let us assume now that $\nabla $ is equiaffine with the corresponding
parallel volume $\omega =\lambda (x)\cdot dx^{1}\wedge dx^{2}\wedge
...\wedge dx^{n}$. Then we know that

\begin{equation*}
\frac{\partial }{\partial x^{k}}\log \lambda (x)=\Gamma _{jk}^{j}.
\end{equation*}

From (\ref{23}) it follows

\begin{equation*}
\overset{\ast}{\Gamma}\,_{jk}^{j}=\frac{\partial }{\partial x^{k}}\left\{ \log |g|-\log\lambda \right\} =\frac{\partial }{\partial x^{k}}\log \frac{|g|}{\lambda }.
\end{equation*}

Hence $\overset{\ast}{\Gamma}\,_{jk}^{j}$ is equivalent with volume $\overset{\ast}{\omega}$.
\end{proof}

\begin{remark}
Observe that in general, the left and right traces of $K$ and $\Gamma$, respectively, are different. 
\end{remark}

\section{Statistical manifolds}\label{sect3}

In the present Section we review in a systematic and comprehensive way the basic mathematical properties of the statistical manifolds. 

\subsection{Definitions and basic properties}
We will consider now some relation between the torsions $T$ and $\overset{\ast}{T}$ of two $g$-dual connections $\nabla$ and $\overset{\ast}{\nabla}$.

Firstly, we give the following result. 

\begin{theorem}\label{thm_2}
If $\nabla,\overset{\ast}{\nabla}$ are $g$-dual connections on a (pseudo-)Riemannian manifold $(M,g)$, then the following are equivalent
\begin{enumerate}[(i)]
\item $\nabla$ and $\overset{\ast}{\nabla}$ have the same torsion, i.e. $T(X,Y)=\overset{\ast}{T}(X,Y)$;
\item $C$ is totally symmetric, i.e. $C(X,Y,Z)=C(Y,X,Z)$;
\item the average connection $\overset{(0)}{\nabla}$ has the same torsion as $\nabla$, i.e. $$\overset{(0)}{T}(X,Y)=T(X,Y)(=\overset{\ast}{T}(X,Y));$$
\item the difference tensor is symmetric, i.e. $K(X,Y)=K(Y,X)$.
\end{enumerate}
\end{theorem}

\begin{proof}
(i) $\Leftrightarrow$ (ii) If we subtract the relations in the definitions of $T$ and $\overset{\ast}{T}$ we get
\begin{equation}\label{eq_*1n}
\begin{split}
\overset{\ast}{T}(X,Y)-T(X,Y)&=(\overset{\ast}{\nabla}_XY-\nabla_XY)-(\overset{\ast}{\nabla}_YX-\nabla_YX)\\
&=K(X,Y)-K(Y,X).
\end{split}
\end{equation}

Hence
\begin{equation}\label{eq_*2n}
\begin{split}
g(\overset{\ast}{T}(X,Y)-T(X,Y),Z)&=g(K(X,Y),Z)-g(K(Y,X),Z)\\
=C(X,Z,Y)-C(Y,Z,X)&=C(X,Y,Z)-C(Y,X,Z),
\end{split}
\end{equation}
where we have used the symmetry of $C$ in the last two components and \eqref{eq_Rel5}.

Taking into account that this formula holds good for any tangent vectors $X,Y,Z$, using that $g$ is (non-degenerate pseudo-)Riemannian metric, we obtain the desired equivalence.

(i) $\Leftrightarrow$ (iii) From the definition \eqref{eq_Av1} of $\overset{(0)}{\nabla}$ we obtain
$$
\overset{(0)}{T}(X,Y)=\frac{1}{2}(T(X,Y)+\overset{\ast}{T}(X,Y)).
$$

If $T=\overset{\ast}{T}$, then obviously $\overset{(0)}{T}=T$. Conversely, if $\overset{(0)}{T}(X,Y)=T(X,Y)$, then $\dfrac{1}{2}[T(X,Y)+\overset{\ast}{T}(X,Y)]=T(X,Y)$ implies $T=\overset{\ast}{T}$, hence the equivalence is shown.

(i) $\Leftrightarrow$ (iv) The proof is obviously if we take into account relation \eqref{eq_A1} in the proof of first equivalence.
\end{proof}

This allows us to give the following definition
\begin{definition}
Let $(M,g)$ be a (pseudo-)Riemannian manifold and let $\nabla$ and $\overset{\ast}{\nabla}$ be dual $g$-connections. If $\nabla$ and $\overset{\ast}{\nabla}$ have the same torsion tensor, then $(M,g,\nabla)$ will be called a {\it pre-statistical manifold}. 

In particular, if $\nabla$ is torsion-free, then the pre-statistical manifold $(M,g,\nabla)$ is called a {\it statistical manifold}. 
\end{definition}

\begin{remark}
\begin{enumerate}[(1)]
\item This definition implies that $(M,g,\nabla)$ is pre-statistical manifold if and only if $C$ is totally symmetric. This formulation corresponds to the definition given by Kurose in \cite{K}.
\item Observe that on pre-statistical manifolds $T(X,Y)=\overset{\ast}{T}(X,Y)\neq 0$ in general.
\item On a statistical manifold $T(X,Y)=\overset{\ast}{T}(X,Y)= 0$, i.e. both connections $\nabla$ and its dual $\overset{\ast}{\nabla}$ are torsion-free.
\end{enumerate}
\end{remark}

We observe that the torsion of the average connection $\overset{(0)}{\nabla}$ is the average of $T$ and $\overset{\ast}{T}$, i.e.
$$
\overset{(0)}{T}(X,Y)=\frac{1}{2}\left\{
T(X,Y)+\overset{\ast}{T}(X,Y)
\right\}.
$$

It follows that on a pre-statistical manifold, we have
$$
\overset{(0)}{T}(X,Y)=T(X,Y)=\overset{\ast}{T}(X,Y).
$$

We also have
\begin{proposition}\label{prop_S1}
If $\nabla$ and $\overset{\ast}{\nabla}$ are $g$-dual connections, then any two of the following conditions imply the rest of them
\begin{enumerate}[(i)]
\item $\nabla$ is torsion free;
\item $\overset{\ast}{\nabla}$ is torsion free;
\item $C$ is totally symmetric (or $\overset{\ast}{C}$ is totally symmetric);
\item the average connection $\overset{(0)}{\nabla}$ coincides with the Levi-Civita connection of $g$.
\end{enumerate}
\end{proposition}

\begin{proof}
\boxed{\textrm{Assume (i), (ii) hold good. }}

Then (i) in Theorem \ref{thm_2} holds good, hence $C$ is totally symmetric and $\overset{(0)}{T}(X,Y)=0$. Taking now into account Proposition \ref{prop_A8} it follows $\overset{(0)}{\nabla}$ is Levi-Civita connection of $g$.

\boxed{\textrm{Assume (i), (iii) hold good.}}

Then \eqref{eq_*2n} in the proof of Theorem \ref{thm_2} implies $\overset{\ast}{T}=0$. Then $T=\overset{\ast}{T}$ implies (iv) as above.

\boxed{\textrm{Assume (i),(iv) hold good.}} 

Then obviously $\overset{\ast}{T}=0$, hence $C$ is totally symmetric by \eqref{eq_*2n}.

\boxed{\textrm{Assume (ii), (iii) hold good.}}

 Then, from \eqref{eq_*2n} it follows $T=0$ and (iv) as above.

\boxed{\textrm{Assume (ii), (iv) hold good.}}

 The obviously $T=0$, hence $C$ totally symmetric by \eqref{eq_*2n}.

\boxed{\textrm{Assume (iii), (iv) hold good.}}

 Then \eqref{eq_*2n} implies $T-\overset{\ast}{T}=0$. On the other hand (iv) implies $2\overset{(0)}{T}=T+\overset{\ast}{T}=0$, hence $T=0$ and $\overset{\ast}{T}=0$.

\end{proof}

\begin{remark}
\begin{enumerate}[(1)]
\item The notion of statistical manifold can now be given in different equivalent forms. Every two of the conditions in Proposition \ref{prop_S1} is equivalent to the fact that $(M,g,\nabla)$ is statistical manifold.
\item Obviously, from Theorem \ref{thm_2} it is clear that on a statistical manifold the difference tensor $K$ is symmetric, i.e.
$$
K(X,Y)=K(Y,X).
$$
\end{enumerate}
\end{remark}

From Proposition \ref{prop: recover nablas} we obtain immediately the following result.
\begin{proposition}
Let $(M,g)$ be a (pseudo-) Riemannian manifold, let $\overset{(0)}{\nabla}$ be its Levi-Civita connection and let $C$ be a (0,3)-tensor field on $M$. Then the affine connections $\nabla$ and $\overset{\ast}{\nabla}$ on $M$ defined by the following relations
\begin{equation}\label{recover nabla-x}
g(\nabla_XY,Z)=g(\overset{(0)}{\nabla}_XY,Z)-\frac{1}{2}C(X,Y,Z)
\end{equation}
\begin{equation}\label{recover nabla*-x}
g(\overset{\ast}{\nabla}_XY,Z)
=g(\overset{(0)}{\nabla}_XY,Z)+\frac{1}{2}C(X,Y,Z)
\end{equation}
have the following properties.
\begin{enumerate}
\item The affine connections  
$\nabla$ and $\overset{\ast}{\nabla}$ are $g$-dual.
\item We have  $(\nabla_Xg)(Y,Z)=C(X,Y,Z)$ and $(\overset{\ast}{\nabla}_Xg)(Y,Z)=-C(X,Y,Z)$. 
\item The affine connections  
$\nabla$ and $\overset{\ast}{\nabla}$ are torsion free, hence $(M,g,\nabla)$ is a statistical manifold. 
 \end{enumerate}
\end{proposition}

\subsection{Curvature related properties}

From Theorem \ref{thm_C1}, we get
\begin{corollary}\label{cor_S2}
Let $(M,g,\nabla)$ be a statistical structure on $M$, with dual connection $\overset{\ast}{\nabla}$. Then

\begin{enumerate}[(i)]
\item 
\begin{equation*}
\begin{split}
R(X,Y)Z&=\overset{(0)}{R}(X,Y)Z-\frac{1}{2}
\left\{
(\overset{(0)}{\nabla}_XK)(Y,Z)-(\overset{(0)}{\nabla}_YK)(X,Z)-\frac{1}{2}[K_X,K_Y]Z
\right\}\\
&=\overset{(0)}{R}(X,Y)Z-\frac{1}{2}\left\{
(\nabla_XK)(Y,Z)-(\nabla_YK)(X,Z)+\frac{1}{2}[K_X,K_Y]Z
\right\},
\end{split}
\end{equation*}
\item
\begin{equation*}
\begin{split}
\overset{\ast}{R}(X,Y)Z&=\overset{(0)}{R}(X,Y)Z
+\frac{1}{2}
\left\{
(\overset{(0)}{\nabla}_XK)(Y,Z)-(\overset{(0)}{\nabla}_YK)(X,Z)+\frac{1}{2}[K_X,K_Y]Z
\right\}\\
&=\overset{(0)}{R}(X,Y)Z
+\frac{1}{2}\left\{
(\nabla_XK)(Y,Z)-(\nabla_YK)(X,Z)\right\}+\frac{3}{4}[K_X,K_Y]Z,
\end{split}
\end{equation*}
\item 
\begin{equation*}
\begin{split}
R(X,Y)Z-\overset{\ast}{R}(X,Y)Z&=-
(\overset{(0)}{\nabla}_XK)(Y,Z)+(\overset{(0)}{\nabla}_YK)(X,Z)\\
&=-
(\nabla_XK)(Y,Z)+(\nabla_YK)(X,Z)-[K_X,K_Y]Z,
\end{split}
\end{equation*}
and
\item
$$
R(X,Y)Z+\overset{\ast}{R}(X,Y)Z=2\overset{(0)}{R}(X,Y)Z+\frac{1}{2}[K_X,K_Y]Z.
$$
\end{enumerate}
\end{corollary}

Locally by taking $X=\dfrac{\partial}{\partial x^j},Y=\dfrac{\partial}{\partial x^i},Z=\dfrac{\partial}{\partial x^m}$, we get

\begin{corollary}\label{cor_S3}
On a statistical manifold $(M,g,\nabla)$ with dual connection $\overset{\ast}{\nabla}$ we have

\begin{enumerate}[(i)]
\item
\begin{equation*}
\begin{split}
R^{\ k}_{m\ ji}&=\overset{(0)}{R}\,^{\ k}_{m\ ji}-\frac{1}{2}(K^{\ k}_{m\ i\overset{(0)}{|}j}-K^{\ k}_{m\ j\overset{(0)}{|}i})+\frac{1}{4}(K^l_{\ mi}K^k_{\ lj}-K^l_{\ mj}K^k_{\ li})\\
&=\overset{(0)}{R}\,^{\ k}_{m\ ji}-\frac{1}{2}(K^{\ k}_{m\ i|j}-K^{\ k}_{m\ j|i})-\frac{1}{4}(K^l_{\ mi}K^k_{\ lj}-K^l_{\ mj}K^k_{\ li}),
\end{split}
\end{equation*}
\item 
\begin{equation*}
\begin{split}
\overset{\ast}{R}\,^{\ k}_{m\ ji}&=\overset{(0)}{R}\,^{\ k}_{m\ ji}+\frac{1}{2}(K^{\ k}_{m\ i\overset{(0)}{|}j}-K^{\ k}_{m\ j\overset{(0)}{|}i})+\frac{1}{4}(K^l_{\ mi}K^k_{\ lj}-K^l_{\ mj}K^k_{\ li})\\
&=\overset{(0)}{R}\,^{\ k}_{m\ ji}+\frac{1}{2}(K^{\ k}_{m\ i|j}-K^{\ k}_{m\ j|i})+\frac{3}{4}(K^l_{\ mi}K^k_{\ lj}-K^l_{\ mj}K^k_{\ li}),
\end{split}
\end{equation*}
\item 
\begin{equation*}
\begin{split}
R^{\ k}_{m\ ji}-\overset{\ast}{R}\,^{\ k}_{m\ ji}&=-(K^{\ k}_{m\ i\overset{(0)}{|}j}-K^{\ k}_{m\ j\overset{(0)}{|}i})\\
&=-(K^{\ k}_{m\ i|j}-K^{\ k}_{m\ j|i})-(K^{l}_{\ mi}K^{k}_{\ lj}-K^{l}_{\ mj}K^k_{\ li}),
\end{split}
\end{equation*}
\item 
\begin{equation*}
\begin{split}
R^{\ k}_{m\ ji}+\overset{\ast}{R}\,^{\ k}_{m\ ji}=2\overset{(0)}{R}\,^{\ k}_{m\ ji}
+\frac{1}{2}(K^l_{\ mi}K^k_{\ lj}-K^l_{\ mj}K^k_{\ li}).
\end{split}
\end{equation*}
\end{enumerate}
\end{corollary}

\begin{proof}
Indeed, from Corollary \ref{rem_C2}, if $(M,g,\nabla)$ is statistical manifold, the all torsions are zero, hence the result.

\end{proof}

Moreover, we get
\begin{corollary}\label{cor_S4}
On the statistical manifold $(M,g,\nabla)$ with dual connection $\overset{\ast}{\nabla}$, we have

\begin{enumerate}[(i)]
\item 
\begin{equation*}
\begin{split}
R_{mi}&=\overset{(0)}{R}_{mi}-\frac{1}{2}(K^{\ j}_{m\ i\overset{(0)}{|}j}-K^{\ j}_{m\ j\overset{(0)}{|}i})+\frac{1}{4}(K^l_{\ mi}K^j_{\ lj}-K^l_{\ mj}K^j_{\ li})
\\
&=\overset{(0)}{R}_{mi}-\frac{1}{2}(K^{\ j}_{m\ i|j}-K^{\ j}_{m\ j|i})-\frac{1}{4}(K^l_{\ mi}K^j_{\ lj}-K^l_{\ mj}K^j_{\ li}),
\end{split}
\end{equation*}
\item 
\begin{equation*}
\begin{split}
\overset{\ast}{R}_{mi}&=\overset{(0)}{R}_{mi}+\frac{1}{2}(K^{\ j}_{m\ i\overset{(0)}{|}j}-K^{\ j}_{m\ j\overset{(0)}{|}i})+\frac{1}{4}(K^l_{\ mi}K^j_{\ lj}-K^l_{\ mj}K^j_{\ li})\\
&=\overset{(0)}{R}_{mi}+\frac{1}{2}(K^{\ j}_{m\ i|j}-K^{\ j}_{m\ j|i})+\frac{3}{4}(K^l_{\ mi}K^j_{\ lj}-K^l_{\ mj}K^j_{\ li}),
\end{split}
\end{equation*}
\item 
\begin{equation*}
\begin{split}
R_{mi}-\overset{\ast}{R}_{mi}&=-(K^{\ j}_{m\ i\overset{(0)}{|}j}-K^{\ j}_{m\ j\overset{(0)}{|}i})\\
&=-(K^{\ j}_{m\ i|j}-K^{\ j}_{m\ j|i})-(K^{l}_{\ mi}K^{j}_{\ lj}-K^{l}_{\ mj}K^j_{\ li}),
\end{split}
\end{equation*}
\item 
\begin{equation*}
\begin{split}
R_{mi}+\overset{\ast}{R}_{mi}=2\overset{(0)}{R}_{mi}+\frac{1}{2}(K^l_{\ mi}K^j_{\ lj}-K^l_{\ mj}K^j_{\ li}).
\end{split}
\end{equation*}
\end{enumerate}
\end{corollary}

\begin{proof}
Indeed, formulas above results from Proposition \ref{prop_C3} by taking torsions equal to zero.

\end{proof}

If we denote the Ricci scalar of $\nabla$ as $R:=g^{ij}R_{ij}$ , then

\begin{corollary}\label{cor_S5}
On the statistical manifold $(M,g,\nabla)$ with dual connection $\overset{\ast}{\nabla}$ we have

\begin{enumerate}[(i)]
\item 
\begin{equation*}
\begin{split}
R_{mi}-R_{im}&=\frac{1}{2}(K^{\ j}_{m\ j\overset{(0)}{|}i}-K^{\ j}_{i\ j\overset{(0)}{|}m}),
\end{split}
\end{equation*}
\item 
\begin{equation*}
\begin{split}
\overset{\ast}{R}_{mi}-\overset{\ast}{R}_{im}&=-\frac{1}{2}(K^{\ j}_{m\ j\overset{(0)}{|}i}-K^{\ j}_{i\ j\overset{(0)}{|}m}).
\end{split}
\end{equation*}

\end{enumerate}

\end{corollary}

\begin{proof}
Indeed, taking torsions equal to zero in (i), (ii) of Corollary \ref{cor_C4}, taking into account the symmetry $K^j_{\ mi}=K^j_{\ im}$, the formulas follow.

\end{proof}

\begin{remark}
Observe that the difference of Ricci scalars can be written in the index free form
\begin{equation*}
\begin{split}
R-\overset{\ast}{R}&=trace_g\left( Div_{\overset{(0)}{\nabla}}K-\overset{(0)}{\nabla}Tr_2(K)\right)\\
& =
trace_g\left( Div_{{\nabla}}K-{\nabla}Tr_2(K)\right)-
g(trace_gK,Tr_1(K))-Tr_3(K),
\end{split}
\end{equation*}
where $Tr_3(K):=g^{mi}K^{l}_{\ mj}K^{j}_{\ li}$.
Indeed,  just multiply (iii) of Corollary \ref{cor_S4} by $g^{mi}$ and use the notations defined already.
\end{remark}

From Corollary \ref{cor_S5} (i) it results 
\begin{corollary}\label{cor_S6}
On a statistical manifold, if trace $K^j_{\ mj}=0$ then the Ricci curvature tensors are symmetric. 
\end{corollary}

Also, we have 
\begin{corollary}\label{cor_S7}
On a statistical manifold $(M,g,\nabla)$ which is Ricci curvature conjugate symmetric, the Ricci curvature tensors are symmetric.
\end{corollary}
Indeed, observe that if $\nabla$ is Ricci curvature conjugate symmetric, then Corollary \ref{cor_S4} (iii) implies 
\begin{equation*}
0=R_{mi}-\overset{\ast}{R}_{mi}=-(K^{\ j}_{m\ i\overset{(0)}{|}j}-K^{\ j}_{m\ j\overset{(0)}{|}i}),
\end{equation*}
and interchanging $m$ and $i$, we can also write 
\begin{equation*}
0=R_{im}-\overset{\ast}{R}_{im}=-(K^{\ j}_{i\ m\overset{(0)}{|}j}-K^{\ j}_{i\ j\overset{(0)}{|}m}).
\end{equation*}
Now, subtracting these two relations and taking into account that $K$ is symmetric in the lower indices, we get 
\begin{equation*}
K^{\ j}_{m\ j\overset{(0)}{|}i}-K^{\ j}_{i\ j\overset{(0)}{|}m}=0
\end{equation*}
hence the conclusion follows from Corollary \ref{cor_S5} (i).

\begin{corollary}\label{cor_S8}
Let $(M,\nabla,g)$ be a statistical manifold. If $\nabla$ is equiaffine with associated volume form $\omega_g$, then the Ricci curvature tensors are symmetric.
\end{corollary}

\begin{proof}
We know from Corollary \ref{cor_S6} that if $K^j_{\ mj}=0$, then the Ricci tensors are symmetric, so we need to show that $\nabla_X\omega_g=0$ implies this condition.

From Corollary \ref{cor_E2} we know that $\nabla_X\omega_g=0$ is equivalent to $K^k_{\ ki}=0$. On the other hand, since $(M,g,\nabla)$ is statistical manifold, $K$ is symmetric (Theorem \ref{thm_2}, iv), hence the conclusion.

\end{proof}

\subsection{The Einstein tensor}

Let us start with a classical statistical manifold $(M,g,\nabla,\overset{\ast}{\nabla})$ and consider the Einstein tensor for $\nabla$ defined as
\begin{equation}\label{eq_app1}
G_{ij}:=R_{(ij)}-\frac{1}{2}g_{ij}R,
\end{equation}
where $R_{ij}:=R^{\ k}_{i\ kj}$ is the Ricci curvature of $\nabla$,  $R_{(ij)}:=\frac{1}{2}(R_{ij}+R_{ji})$, and $R:=g^{ij}R_{ij}$ is the Ricci scalar of $\nabla$, respectively.

We start with the second Bianchi identities for $\nabla$ and use $R^{\ j}_{l\ hk}=-R^{\ j}_{l\ kh}$.

We will study in the following the divergence $\nabla\cdot G$ of the Einstein tensor. Locally, this relation means $\nabla^iG_{ij}$.

{\underline {\it Step 1. Bianchi Identities}}

Taking into account the vanishing of the torsion $T$, we consider the First Bianchi Identities \eqref{1-st Bianchi}, that is 
$$
\nabla_iR^{\ h}_{l\ jk}+\nabla_jR^{\ h}_{l\ ki}+\nabla_kR^{\ h}_{l\ ij}=0.
$$

{\underline {\it Step 2. Covariant derivative $\nabla_kR$}}

Summing $i$ with $h$ in Bianchi Identities, we get
$$
\nabla_hR^{\ h}_{l\ jk}+\nabla_jR^{\ h}_{l\ kh}+\nabla_kR^{\ h}_{l\ hj}=0.
$$

We multiply by $g^{lj}$ and by summation it results
\begin{equation}\label{eq_app2}
g^{lj}\nabla_hR^{\ h}_{l\ jk}-g^{lj}\nabla_jR_{lk}+g^{lj}\nabla_kR_{lj}=0,
\end{equation}
where we use $R^{\ h}_{i\ hj}=R_{ij}$.

\begin{remark}\label{rem: some nabla terms}
\begin{enumerate}[(i)]
\item Firstly, observe that
$$
\nabla_kR=\nabla_k(g^{lj}R_{lj})=(\nabla_kg^{lj})R_{lj}+g^{lj}\nabla_kR_{lj},
$$
gives
\begin{equation}\label{eq_app2.1}
g^{lj}\nabla_kR_{lj}=\nabla_kR-(\nabla_kg^{lj})R_{lj}.
\end{equation}
\item Next, observe that
\begin{equation}\label{eq_app2.2}
\begin{split}
g^{lj}\nabla_h(R^{\ h}_{l\ jk})&=\nabla_h(g^{lj}R^{\ h}_{l\ jk})-(\nabla_hg^{lj})R^{\ h}_{l\ jk}\\
&=\nabla_h(g^{lj}g^{hs}R_{lsjk})-(\nabla_hg^{lj})R^{\ h}_{l\  jk}\\
&=-\nabla_h(g^{lj}g^{hs}\overset{\ast}{R}_{sljk})-(\nabla_hg^{lj})R^{\ h}_{l\  jk}\\
&=-\nabla_h(g^{hs}\overset{\ast}{R}\,^{j}_{s\ jk})-(\nabla_hg^{lj})R^{\ h}_{l\  jk}\\
&=-\nabla_h(g^{hs}\overset{\ast}{R}_{sk})-(\nabla_hg^{lj})R^{\ h}_{l\  jk}\\
&=-(\nabla_hg^{hs})\overset{\ast}{R}_{sk}-g^{hs}\nabla_h\overset{\ast}{R}_{sk}-(\nabla_hg^{lj})R^{\ h}_{l\  jk},
\end{split}
\end{equation}
where we use $R_{lsjk}=-\overset{\ast}{R}_{sljk}$ (see Theorem \ref{thm: R and R* 1-st 2 ind}).
\end{enumerate}
\end{remark}

By substituting \eqref{eq_app2.1} and \eqref{eq_app2.2} in \eqref{eq_app2} we obtain
$$
-(\nabla_hg^{hs})\overset{\ast}{R}_{sk}-\nabla^h\overset{\ast}{R}_{hk}-(\nabla_hg^{lj})R^{\ h}_{l\  jk}
-g^{lj}\nabla_jR_{lk}+\nabla_kR-(\nabla_kg^{lj})R_{lj}=0.
$$
and from here
\begin{equation}\label{eq_app3}
\nabla_kR=\nabla^hR_{hk}+\nabla^h\overset{\ast}{R}_{hk}+(\nabla_hg^{lj})R^{\ h}_{l\  jk}+(\nabla_kg^{lj})R_{lj}+(\nabla_hg^{hs})\overset{\ast}{R}_{sk},
\end{equation}
where we denote $\nabla^h:=g^{hs}\nabla_s$.

\begin{remark}
Observe that in the case of Levi-Civita connection, due to its metricity, this formula is equivalent to
$$
\nabla_kR=2\nabla^hR_{hk}.
$$
\end{remark}

{\underline {\it Step 3. Covariant derivative of Einstein tensor $\nabla_kG_{ij}$}}

Next, we take the convariant derivation of \eqref{eq_app1} with respect to $\nabla$
$$
\nabla_kG_{ij}=\nabla_kR_{(ij)}-\frac{1}{2}(\nabla_kg_{ij})R-\frac{1}{2}g_{ij}\nabla_kR
$$
and substitute $\nabla_kR$ from \eqref{eq_app3}. It follows
\begin{equation}\label{eq_app4}
\begin{split}
\nabla_kG_{ij}&=\nabla_kR_{(ij)}-\frac{1}{2}(\nabla_kg_{ij})R\\
&\quad-\frac{1}{2}g_{ij}\left\{
\nabla^hR_{hk}+\nabla^h\overset{\ast}{R}_{hk}+(\nabla_hg^{ls})R^{\ h}_{l\  sk}+(\nabla_kg^{ls})R_{ls}+(\nabla_hg^{hs})\overset{\ast}{R}_{sk}
\right\}.
\end{split}
\end{equation}

{\underline {\it Step 4. Divergence of Einstein tensor $\nabla^iG_{ij}$}}

Now we multiply by $g^{pk}$ and summation:
\begin{equation*}
\begin{split}
g^{pk}\nabla_kG_{ij}&=g^{pk}\nabla_kR_{(ij)}-\frac{1}{2}g^{pk}(\nabla_kg_{ij})R\\
&\quad-\frac{1}{2}g^{pk}g_{ij}\left\{
\nabla^hR_{hk}+\nabla^h\overset{\ast}{R}_{hk}+(\nabla_hg^{ls})R^{\ h}_{l\  sk}+(\nabla_kg^{ls})R_{ls}+(\nabla_hg^{hs})\overset{\ast}{R}_{sk}
\right\}.
\end{split}
\end{equation*} 
and take the trace $p=i$
\begin{equation*}
\begin{split}
\nabla^iG_{ij}&=\nabla^iR_{(ij)}-\frac{1}{2}(\nabla^ig_{ij})R\\
&\quad-\frac{1}{2}\left\{
\nabla^iR_{ij}+\nabla^i\overset{\ast}{R}_{ij}+(\nabla_hg^{ls})R^{\ h}_{l\ sj}+(\nabla_jg^{ls})R_{ls}
+(\nabla_hg^{hs})\overset{\ast}{R}_{sj}
\right\}\\
&=-\frac{1}{2}\left\{
\nabla^i(\overset{\ast}{R}_{ij}-R_{ji})+(\nabla^ig_{ij})R+(\nabla_hg^{ls})R^{\ h}_{l\ sj}
+(\nabla_jg^{ls})R_{ls}+(\nabla_hg^{hs})\overset{\ast}{R}_{sj}
\right\}.
\end{split}
\end{equation*}

\begin{remark}
In the case $\nabla$ is Levi-Civita connection, the Ricci tensor $R_{ij}$ is symmetric, $R_{ij}=\overset{\ast}{R}_{ij}$ and $\nabla$ is metrical, hence we get the usual formula
$$
\nabla^iG_{ij}=0.
$$
\end{remark}

Let us observe now that we actually can compute the $\overset{\ast}{\nabla}$-divergence of $\overset{\ast}{G}_{ij}$ in a similar manner,
where the Einstein tensor for $\overset{\ast}{\nabla}$ defined as
\begin{equation}\label{eq_app1_*}
\overset{\ast}{G}_{ij}:=\overset{\ast}{R}_{(ij)}-\frac{1}{2}g_{ij}\overset{\ast}{R}.
\end{equation}

Indeed, the $\overset{\ast}{\nabla}$-divergence of $\overset{\ast}{G}_{ij}$
is given as
\begin{equation}\label{nabla star G star}
\overset{\ast}{\nabla}\,^{i}\overset{\ast}{G}_{ij}=
-\frac{1}{2}\left\{
\overset{\ast}{\nabla}\,^{i}(R_{ij}-\overset{\ast}{R}_{ji})+(\overset{\ast}{\nabla}\,^{i}g_{ij})\overset{\ast}{R}+
(\overset{\ast}{\nabla}_hg^{ls})\overset{\ast}{R}\,^{h}_{l\  sj}+(\overset{\ast}{\nabla}_jg^{ls})\overset{\ast}{R}_{ls}+(\overset{\ast}{\nabla}_hg^{hs})R_{sj}
\right\}.
\end{equation}

Next, in order to compute the mixed divergences of types $\nabla^i \overset{\ast}{G}_{ij}$ and $\overset{\ast}{\nabla}\,^{i}G_{ij}$, let us observe that by {subtracting} formulas \eqref{eq_app1} and \eqref{eq_app1_*}, we get 
\begin{equation}
G_{ij}-\overset{\ast}{G}_{ij}:=R_{(ij)}-\overset{\ast}{R}_{(ij)}-\frac{1}{2}g_{ij}(R-\overset{\ast}{R}),
\end{equation}
hence
\begin{equation}\label{nabla star G}
\overset{\ast}{\nabla}\,^{i} G_{ij}=\overset{\ast}{\nabla}\,^{i}  \overset{\ast}{G}_{ij}+\overset{\ast}{\nabla}\,^{i} (R_{(ij)}-\overset{\ast}{R}_{(ij)})-\frac{1}{2}\overset{\ast}{\nabla}\,^{i} \left[g_{ij}(R-\overset{\ast}{R})\right]
\end{equation}
and 
\begin{equation}\label{nabla G star}
\nabla^i \overset{\ast}{G}_{ij}=\nabla^i  G_{ij}+\nabla^i (\overset{\ast}{R}_{(ij)}-R_{(ij)})-\frac{1}{2}\nabla^{i} \left[g_{ij}(\overset{\ast}{R}-R)\right],
\end{equation}
where $R_{ij}-\overset{\ast}{R}_{ij}$ and $R-\overset{\ast}{R}$ are given in Corollaries \ref{cor_S4} and \ref{cor_S5}.

By substituting now \eqref{eq_app1_*} in the first relation above, we get 
\begin{equation*}
\begin{split}
\overset{\ast}{\nabla}\,^{i} G_{ij}&=\overset{\ast}{\nabla}\,^{i} (\overset{\ast}{R}_{(ij)}-\frac{1}{2}g_{ij}\overset{\ast}{R})+\overset{\ast}{\nabla}\,^{i} (R_{(ij)}-\overset{\ast}{R}_{(ij)})-\frac{1}{2}(\overset{\ast}{\nabla}\,^{i} g_{ij})(R-\overset{\ast}{R})-\frac{1}{2}\overset{\ast}{\nabla}_j (R-\overset{\ast}{R})\\
&=\overset{\ast}{\nabla}\,^{i} \overset{\ast}{R}_{(ij)}-\frac{1}{2}(\overset{\ast}{\nabla}\,^{i} g_{ij})\overset{\ast}{R}
-\frac{1}{2}\overset{\ast}{\nabla}_j\overset{\ast}{R}\\
&+\overset{\ast}{\nabla}\,^{i} (R_{(ij)}-\overset{\ast}{R}_{(ij)})-\frac{1}{2}(\overset{\ast}{\nabla}\,^{i} g_{ij})(R-\overset{\ast}{R})-\frac{1}{2}\overset{\ast}{\nabla}_j (R-\overset{\ast}{R})\\
&=\overset{\ast}{\nabla}\,^{i} R_{(ij)}-\frac{1}{2}(\overset{\ast}{\nabla}\,^{i} g_{ij})R
-\frac{1}{2}\overset{\ast}{\nabla}_j R.
\end{split}
\end{equation*}

An alternative form is obtained by substituting \eqref{nabla star G star}.

Likewise, by substituting now \eqref{eq_app1} in \eqref{nabla G star}, we get 
\begin{equation*}
\begin{split}
\nabla^i \overset{\ast}{G}_{ij}&=\nabla^i  (R_{(ij)}-\frac{1}{2}g_{ij}R)+\nabla^i (\overset{\ast}{R}_{(ij)}-R_{(ij)})-\frac{1}{2}(\nabla^{i} g_{ij})(\overset{\ast}{R}-R)-\frac{1}{2}\nabla_j(\overset{\ast}{R}-R)\\
&=\nabla^i  R_{(ij)}-\frac{1}{2}\nabla^i  (g_{ij})R-\frac{1}{2}\nabla_jR\\
&+\nabla^i (\overset{\ast}{R}_{(ij)}-R_{(ij)})-\frac{1}{2}(\nabla^{i} g_{ij})(\overset{\ast}{R}-R)-\frac{1}{2}\nabla_j(\overset{\ast}{R}-R)\\
&=\nabla^i \overset{\ast}{R}_{(ij)}-\frac{1}{2}(\nabla^{i} g_{ij})\overset{\ast}{R}-\frac{1}{2}\nabla_j\overset{\ast}{R}.
\end{split}
\end{equation*}

Summarizing, we obtain 

\begin{theorem}\label{thm_app1}
Let $(M,g,\nabla,\overset{\ast}{\nabla})$ be a statistical manifold and $G_{ij}:=R_{ij}-\frac{1}{2}g_{ij}R$ the Einstein tensor of $\nabla$. Then the divergences of Einstein tensors are  given by
\begin{equation*}
\begin{split}
\nabla^iG_{ij}&=-\frac{1}{2}\left\{
\nabla^i(\overset{\ast}{R}_{ij}-R_{ji})+(\nabla^ig_{ij})R+(\nabla_hg^{ls})R^{\ h}_{l\ sj}
+(\nabla_jg^{ls})R_{ls}+(\nabla_hg^{hs})\overset{\ast}{R}_{sj}
\right\}\\
\overset{\ast}{\nabla}\,^{i}\overset{\ast}{G}_{ij}&=
-\frac{1}{2}\left\{
\overset{\ast}{\nabla}\,^{i}(R_{ij}-\overset{\ast}{R}_{ji})+(\overset{\ast}{\nabla}\,^{i}g_{ij})\overset{\ast}{R}+
(\overset{\ast}{\nabla}_hg^{ls})\overset{\ast}{R}\,^{\ h}_{l\  sj}+(\overset{\ast}{\nabla}_jg^{ls})\overset{\ast}{R}_{ls}+(\overset{\ast}{\nabla}_hg^{hs})R_{sj}
\right\}\\
\overset{\ast}{\nabla}\,^{i} G_{ij}&=\overset{\ast}{\nabla}\,^{i} R_{(ij)}-\frac{1}{2}(\overset{\ast}{\nabla}\,^{i} g_{ij})R
-\frac{1}{2}\overset{\ast}{\nabla}_j R\\
\nabla^i \overset{\ast}{G}_{ij}&=\nabla^i \overset{\ast}{R}_{(ij)}-\frac{1}{2}(\nabla^{i} g_{ij})\overset{\ast}{R}-\frac{1}{2}\nabla_j\overset{\ast}{R}.
\end{split}
\end{equation*}

\end{theorem}

\begin{remark}
We point out that if necessary, by using (iii) of Corollary \ref{cor_S4} and \eqref{eq_app3}, one can actually compute $\nabla_j\overset{\ast}{R}$ and substitute in the formulas above, but the relations get longer and complicated. same for ${\overset{\ast}\nabla}_j{R}$.
\end{remark}
\begin{corollary}
  Let $(M,g,\nabla,\overset{\ast}{\nabla})$ be a statistical manifold which we assume it is Ricci conjugate symmetric. Then the following formulas hold good
  \begin{equation}
    \begin{split}
\nabla^iG_{ij}&=-\frac{1}{2}\left\{
(\nabla^ig_{ij})R+(\nabla_hg^{ls})R^{\ h}_{l\ sj}
+(\nabla_jg^{ls})R_{ls}+(\nabla_hg^{hs})R_{sj}
                \right\}\\
      \overset{\ast}{\nabla}\,^{i}\overset{\ast}{G}_{ij}&=
-\frac{1}{2}\left\{(\overset{\ast}{\nabla}\,^{i}g_{ij})\overset{\ast}{R}+
(\overset{\ast}{\nabla}_hg^{ls})\overset{\ast}{R}\,^{\ h}_{l\  sj}+(\overset{\ast}{\nabla}_jg^{ls})\overset{\ast}{R}_{ls}+(\overset{\ast}{\nabla}_hg^{hs})\overset{\ast}{R}_{sj}
                                    \right\}\\
      \overset{\ast}{\nabla}\,^{i} G_{ij}&=\overset{\ast}{\nabla}\,^{i}  \overset{\ast}{G}_{ij},\quad 
      \nabla^i \overset{\ast}{G}_{ij}=\nabla^i  G_{ij}.
      \end{split}
\end{equation}
\end{corollary}

Indeed, observe that in the case of Ricci conjugate symmetric statistical manifolds, the Ricci curvature tensors and the Ricci curvature tensors of $\nabla$ and $\overset{\ast}{\nabla}$ are equal. This implies that actually, the Einstein tensors are equal $G_{ij}=\overset{\ast}{G}_{ij}$, hence the formulas above follows immediately.


\section{Quasi-statistical manifolds}\label{sect_Quasi_stat}

\subsection{Definitions and basic properties}
We have seen that a special case of dual connections are those of equal torsion (we called them pre-statistical manifolds), and that in the case when these torsions vanish, we obtain the well-known notion of statistical manifolds.

Another type is the case when only one of the torsions $T$ or $\overset{\ast}{T}$ vanish.

\begin{definition}\label{def_Q1}
Let $(M,g)$ be a pseudo-Riemannian manifold and let $\nabla$ and $\overset{\ast}{\nabla}$ be dual $g$-connections. If $\nabla$ has vanishing torsion $T=0$, then $(M,g,\nabla)$ is called a {\it quasi-statistical manifold}.
\end{definition}

\begin{theorem}\label{thm_3}
If $\nabla$ and $\overset{\ast}{\nabla}$ are $g$-dual connections on a (pseudo-)Riemannian manifold $(M,g)$, then the following are equivalent
\begin{enumerate}[(i)]
\item $T=0$ ($\overset{\ast}{T}\neq 0$),
\item $C(X,Y,Z)-C(Y,X,Z)=g(\overset{\ast}{T}(X,Y),Z)$,
\item $\overset{\ast}{T}(X,Y)=2\overset{(0)}{T}(X,Y)$,
\item $K(X,Y)-K(Y,X)=\overset{\ast}{T}(X,Y).$
\end{enumerate}
\end{theorem}

\begin{proof}
(i) $\Leftrightarrow$ (ii) Follows from \eqref{eq_*2n}.

(ii) $\Leftrightarrow$ (iii) Follows from the definitions.

(i) $\Leftrightarrow$ (iv) Follows from \eqref{eq_*1n}.

\end{proof}

\begin{remark}
\begin{enumerate}[(1)]
\item It is clear that having non-vanishing torsion the average connection $\overset{(0)}{\nabla}$ of a quasi-statistical manifold $(M,g,\nabla)$ is NOT the Levi-Civita connection of $g$.

Moreover, in this case the tensor $C$ is not totally symmetric anymore. 
\item In this case, the difference tensor $K$ is not symmetric anymore.
\item The quasi-statistical manifold can also be defined by asking $\overset{\ast}{T}=0$ and $T\neq 0$, but we prefer the Definition \ref{def_Q1} because it simplifies the curvature relations in the form we give in Theorem \ref{thm_C1}.
\end{enumerate}
\end{remark}

\subsection{Curvature related properties}

\begin{theorem}\label{thm_Q1}
Let $(M,g,\nabla)$ be a quasi-statistical with dual connection $\overset{\ast}{\nabla}$. Then

\begin{enumerate}[(i)]
\item 
\begin{equation*}
\begin{split}
R(X,Y)Z&=\overset{(0)}{R}(X,Y)Z-\frac{1}{2}
\left\{
(\overset{(0)}{\nabla}_XK)(Y,Z)-(\overset{(0)}{\nabla}_YK)(X,Z)-\frac{1}{2}[K_X,K_Y]Z
\right\}\\
&-\frac{1}{2}K(\overset{(0)}{T}(X,Y),Z),\\
&=\overset{(0)}{R}(X,Y)Z-\frac{1}{2}\left\{
(\nabla_XK)(Y,Z)-(\nabla_YK)(X,Z)+\frac{1}{2}[K_X,K_Y]Z
\right\}
\end{split}
\end{equation*}
\item
\begin{equation*}
\begin{split}
\overset{\ast}{R}(X,Y)Z&=\overset{(0)}{R}(X,Y)Z
+\frac{1}{2}
\left\{
(\overset{(0)}{\nabla}_XK)(Y,Z)-(\overset{(0)}{\nabla}_YK)(X,Z)+\frac{1}{2}[K_X,K_Y]Z
\right\}\\
&+\frac{1}{2}K(\overset{(0)}{T}(X,Y),Z)\\
&=\overset{(0)}{R}(X,Y)Z
+\frac{1}{2}\left\{
(\nabla_XK)(Y,Z)-(\nabla_YK)(X,Z)+\frac{3}{2}[K_X,K_Y]Z\right\}
\end{split}
\end{equation*}

\item 
\begin{equation*}
\begin{split}
R(X,Y)Z-\overset{\ast}{R}(X,Y)Z&=-\left\{
(\overset{(0)}{\nabla}_XK)(Y,Z)-(\overset{(0)}{\nabla}_YK)(X,Z)
\right\}-K(\overset{(0)}{T}(X,Y),Z)\\
&=-\left\{
(\nabla_XK)(Y,Z)-(\nabla_YK)(X,Z)+[K_X,K_Y]Z
\right\}
\end{split}
\end{equation*}
and
\item
$$
R(X,Y)Z+\overset{\ast}{R}(X,Y)Z=2\overset{(0)}{R}(X,Y)Z+\frac{1}{2}[K_X,K_Y]Z.
$$
\end{enumerate}
\end{theorem}

\begin{proof}
Follows directly from Theorem \ref{thm_C1} taking $T=0$.

\end{proof}

\begin{corollary}\label{cor_Q2}
In local coordinates, the formulas in Theorem \ref{thm_Q1} read
\begin{enumerate}[(i)]
\item
\begin{equation*}
\begin{split}
R^{\ k}_{m\ ji}&=\overset{(0)}{R}\,^{\ k}_{m\ ji}-\frac{1}{2}(K^{\ k}_{m\ i\overset{(0)}{|}j}-K^{\ k}_{m\ j\overset{(0)}{|}i})+\frac{1}{4}(K^l_{\ mi}K^k_{\ lj}-K^l_{\ mj}K^k_{\ li})
-\frac{1}{2}\overset{(0)}{T}\,^l_{\ ji}K^k_{\ ml}\\
&=\overset{(0)}{R}\,^{\ k}_{m\ ji}-\frac{1}{2}(K^{\ k}_{m\ i|j}-K^{\ k}_{m\ j|i})-\frac{1}{4}(K^l_{\ mi}K^k_{\ lj}-K^l_{\ mj}K^k_{\ li})
\end{split}
\end{equation*}
\item 
\begin{equation*}
\begin{split}
\overset{\ast}{R}\,^{\ k}_{m\ ji}&=\overset{(0)}{R}\,^{\ k}_{m\ ji}+\frac{1}{2}(K^{\ k}_{m\ i\overset{(0)}{|}j}-K^{\ k}_{m\ j\overset{(0)}{|}i})+\frac{1}{4}(K^l_{\ mi}K^k_{\ lj}-K^l_{\ mj}K^k_{\ li})
+\frac{1}{2}\overset{(0)}{T}\,^l_{\ ji}K^k_{\ ml}\\
&=\overset{(0)}{R}\,^{\ k}_{m\ ji}+\frac{1}{2}(K^{\ k}_{m\ i|j}-K^{\ k}_{m\ j|i})+\frac{3}{4}(K^l_{\ mi}K^k_{\ lj}-K^l_{\ mj}K^k_{\ li})
\end{split}
\end{equation*}
\item 
\begin{equation*}
\begin{split}
R^{\ k}_{m\ ji}-\overset{\ast}{R}\,^{\ k}_{m\ ji}&=-(K^{\ k}_{m\ i\overset{(0)}{|}j}-K^{\ k}_{m\ j\overset{(0)}{|}i})-\overset{(0)}{T}\,^l_{\ ji}K^k_{\ ml}\\
&=-(K^{\ k}_{m\ i|j}-K^{\ k}_{m\ j|i})-(K^{l}_{\ mi}K^{\ k}_{\ lj}-K^{l}_{\ mj}K^k_{\ li})
\end{split}
\end{equation*}
\item 
\begin{equation*}
\begin{split}
R^{\ k}_{m\ ji}+\overset{\ast}{R}\,^{\ k}_{m\ ji}=2\overset{(0)}{R}\,^{\ k}_{m\ ji}
+\frac{1}{2}(K^l_{\ mi}K^k_{\ lj}-K^l_{\ mj}K^k_{\ li}).
\end{split}
\end{equation*}
\end{enumerate}
\end{corollary}

\begin{proof}
Follows from directly from Theorem \ref{thm_Q1}, or from Corollary \ref{rem_C2}.

\end{proof}

\begin{proposition}\label{prop_Q3}
On the quasi-statistical manifold $(M,g,\nabla)$ we have

\begin{enumerate}[(i)]
\item 
\begin{equation*}
\begin{split}
R_{mi}&=\overset{(0)}{R}_{mi}-\frac{1}{2}(K^{\ j}_{m\ i\overset{(0)}{|}j}-K^{\ j}_{m\ j\overset{(0)}{|}i})+\frac{1}{4}(K^l_{\ mi}K^j_{\ lj}-K^l_{\ mj}K^j_{\ li})
-\frac{1}{2}\overset{(0)}{T}\,^l_{\ ji}K^j_{\ ml}\\
&=\overset{(0)}{R}_{mi}-\frac{1}{2}(K^{\ j}_{m\ i|j}-K^{\ j}_{m\ j|i})-\frac{1}{4}(K^l_{\ mi}K^j_{\ lj}-K^l_{\ mj}K^j_{\ li})
\end{split}
\end{equation*}
\item 
\begin{equation*}
\begin{split}
\overset{\ast}{R}_{mi}&=\overset{(0)}{R}_{mi}+\frac{1}{2}(K^{\ j}_{m\ i\overset{(0)}{|}j}-K^{\ j}_{m\ j\overset{(0)}{|}i})+\frac{1}{4}(K^l_{\ mi}K^j_{\ lj}-K^l_{\ mj}K^j_{\ li})
+\frac{1}{2}\overset{(0)}{T}\,^l_{\ ji}K^j_{\ ml}\\
&=\overset{(0)}{R}_{mi}+\frac{1}{2}(K^{\ j}_{m\ i|j}-K^{\ j}_{m\ j|i})+\frac{3}{4}(K^l_{\ mi}K^j_{\ lj}-K^l_{\ mj}K^j_{\ li})
\end{split}
\end{equation*}
\item 
\begin{equation*}
\begin{split}
R_{mi}-\overset{\ast}{R}_{mi}&=-(K^{\ j}_{m\ i\overset{(0)}{|}j}-K^{\ j}_{m\ j\overset{(0)}{|}i})-\overset{(0)}{T}\,^l_{\ ji}K^j_{\ ml}\\
&=-(K^{\ j}_{m\ i|j}-K^{\ j}_{m\ j|i})-(K^{l}_{\ mi}K^{j}_{\ lj}-K^{l}_{\ mj}K^j_{\ li})
\end{split}
\end{equation*}
\item 
\begin{equation*}
\begin{split}
R_{mi}+\overset{\ast}{R}_{mi}=2\overset{(0)}{R}_{mi}+\frac{1}{2}(K^l_{\ mi}K^j_{\ lj}-K^l_{\ mj}K^j_{\ li}).
\end{split}
\end{equation*}
\end{enumerate}

\end{proposition}

\begin{proof}
Follows directly from Proposition \ref{prop_C3}.

\end{proof}

\begin{corollary}\label{cor_Q4}
With the notations above, we have

\begin{enumerate}[(i)]
\item 
\begin{equation*}
\begin{split}
R_{mi}-R_{im}&=(\overset{(0)}{R}_{mi}-\overset{(0)}{R}_{im})\\
&-\frac{1}{2}(K^j_{\ mi}-K^j_{\ im})_{|j}+\frac{1}{2}(K^{\ j}_{m\ j|i}-K^{\ j}_{i\ j|m})\\
&-\frac{1}{4}(K^l_{\ mi}-K^k_{\ im})K^j_{\ lj}+\frac{1}{4}(K^l_{\ mj}K^j_{\ li}-K^l_{\ {jm}}K^j_{\ {il}})\\
\end{split}
\end{equation*}
\item 
\begin{equation*}
\begin{split}
\overset{\ast}{R}_{mi}-\overset{\ast}{R}_{im}&=(\overset{(0)}{R}_{mi}-\overset{(0)}{R}_{im})\\
&+\frac{1}{2}(K^j_{\ mi}-K^j_{\ im})_{|j}-\frac{1}{2}(K^{\ j}_{m\ j|i}-K^{\ j}_{i\ j|m})\\
&+\frac{3}{4}(K^l_{\ mi}-K^l_{\ im})K^j_{\ lj}-\frac{3}{4}(K^l_{\ mj}K^j_{\ li}-K^l_{\ {jm}}K^j_{\ {il}} )\\
\end{split}
\end{equation*}
\item 
\begin{equation*}
\begin{split}
R_{mi}+R_{im}&=(\overset{(0)}{R}_{mi}+\overset{(0)}{R}_{im})-\frac{1}{2}(K^j_{\ mi}+K^j_{\ im})_{|j}\\
&+\frac{1}{2}(K^{\ j}_{m\ j|i}+K^{\ j}_{i\ j|m})-\frac{1}{4}(K^l_{\ mi}+K^k_{\ im})K^j_{\ lj}\\
&+\frac{1}{4}(K^l_{\ mj}K^j_{\ li}+K^l_{\ {jm}}K^j_{\ {il}})\\
\end{split}
\end{equation*}
\item 
\begin{equation*}
\begin{split}
\overset{\ast}{R}_{mi}+\overset{\ast}{R}_{im}&=(\overset{(0)}{R}_{mi}+\overset{(0)}{R}_{im})+\frac{1}{2}(K^j_{\ mi}+K^j_{\ im})_{|j}\\
&-\frac{1}{2}(K^{\ j}_{m\ j|i}+K^{\ j}_{i\ j|m})+\frac{3}{4}(K^l_{\ mi}+K^k_{\ im})K^j_{\ lj}\\
&-\frac{3}{4}(K^l_{\ mj}K^j_{\ li}-K^l_{\ {jm}}K^j_{\ {il}}).
\end{split}
\end{equation*}
\end{enumerate}

\end{corollary}

Indeed, the formulas follow by adding and substracting the relations in Proposition \ref{prop_Q3}.

\begin{remark}
If $(M,g,\nabla)$ is a quasi-statistical manifold, then clearly, from \eqref{eq_Ricc1}, \eqref{eq_Ricc2}, we get
$$
R_{kj}-R_{jk}+g^{is}R_{isjk}=0
$$
and
\begin{equation*}
\begin{split}
\overset{\ast}{R}_{kj}-\overset{\ast}{R}_{jk}+g^{is}\overset{\ast}{R}_{isjk}=(\overset{\ast}{T}\!^m_{\ ij}\overset{\ast}{T}\!^i_{\ mk}+\overset{\ast}{T}\!^m_{\ jk}\overset{\ast}{T}\!^i_{\ mi}+\overset{\ast}{T}\!^m_{\ ki}\overset{\ast}{T}\!^i_{\ mj})
+\overset{\ast}{T}\!^{\ i}_{j\ k\overset{\ast}{|}i}+\overset{\ast}{T}\!^{\ i}_{k\ i\overset{\ast}{|}j}+\overset{\ast}{T}\!^{\ i}_{i\ j\overset{\ast}{|}k}.
\end{split}
\end{equation*}
\end{remark}

\begin{corollary}\label{cor_A25}
Let $(M,g,\nabla)$ be a quasi-statistical manifold. If the trace condition $K^k_{\ ik}=0$ holds, then
\begin{equation*}
\begin{split}
R_{mi}-R_{im}&=\overset{(0)}{R}_{mi}-\overset{(0)}{R}_{im}-\frac{1}{2}(K^j_{\ mi}-K^j_{\ im})_{|j}
+\frac{1}{4}(K^l_{\ mj}K^j_{\ li}-K^l_{\ ij}K^j_{\ lm}).
\end{split}
\end{equation*}
\end{corollary}

From Theorem \ref{thm_A13.19} we get:

\begin{theorem}
If $(M,g,\nabla)$ is a quasi statistical manifold, then
$$
R_{ij}-R_{ji}=\sum_{k=1}^n\left(\frac{\partial\Gamma^k_{\ jk}}{\partial x^i}-\frac{\partial\Gamma^k_{\ ik}}{\partial x^j}\right)
$$
and
$$
\overset{\ast}{R}_{ij}-\overset{\ast}{R}_{ji}=\nabla_k\overset{\ast}{T}\!^k_{\ ji}
+\sum_{k=1}^n\left(
\frac{\partial\overset{\ast}{\Gamma}\!^k_{\ jk}}{\partial x^i}-\frac{\partial\overset{\ast}{\Gamma}\!^k_{\ ik}}{\partial x^j}
\right).
$$
\end{theorem}

\begin{corollary}
If $(M,g,\nabla)$ is a quasi-statistical manifold and $\nabla$ is equiaffine connection, then
\begin{equation*}
\begin{split}
&R_{ij}-R_{ji}=0\qquad {\textrm and}\\
&\overset{\ast}{R}_{ij}-\overset{\ast}{R}_{ji}=\overset{\ast}{\nabla}_k\overset{\ast}{T}\!^k_{\ ji}
-\sum_{k=1}^n\left(
\frac{\partial(\overset{\ast}{T}\!^k_{\ jk})^\ast}{\partial x^i}-\frac{\partial(\overset{\ast}{T}\!^k_{\ ik})^\ast}{\partial x^j}
\right),
\end{split}
\end{equation*}
where $\overset{\ast}{T}$ is the torsion of the dual connection $\overset{\ast}{\nabla}$.
\end{corollary}
Indeed, since $\nabla$ is torsion free and equiaffine, from Corollary \ref{cor_A13.21} it results that the Ricci tensors are symmetric. 

For the dual connection $\overset{\ast}{\nabla}$, firstly, observe that Theorem \ref{thm_A13.19} reads
\begin{equation*}
\overset{\ast}{R}_{ij}-\overset{\ast}{R}_{ji}=\overset{\ast}{\nabla}_k\overset{\ast}{T}\!^k_{\ ji}
+\sum_{k=1}^n\left(
\frac{\partial\overset{\ast}{\Gamma}\!^k_{\ jk}}{\partial x^i}-\frac{\partial\overset{\ast}{\Gamma}\!^k_{\ ik}}{\partial x^j}
\right)
\end{equation*}

Next, since $\nabla$ is equiaffine, from Proposition \ref{prop: equiaffine dual conn} it results that $\overset{\ast}{\nabla}$ is also equiaffine with volume form $\overset{\ast}{\omega}$. Hence the Corollary \ref{cor: Ricci diff by torsion} implies the conclusion.

\subsection{The Einstein tensor}

The same study of Einstein tensors and their divergence that we studied in the case of statistical manifolds can be extended to the case of quasi-statistical manifolds.

Let us start with a quasi-statistical manifold $(M,g,\nabla,\overset{\ast}{\nabla})$ and consider the Einstein tensor for $\nabla$ given in the same way as \eqref{eq_app1}.

Since $\nabla$ is torsion free, the same divergence formula as in the case of statistical manifolds holds good. That is, first divergence formula for $\nabla^iG_{ij}$ in Theorem \ref{thm_app1} holds good in the case of quasi-statistical manifolds as well. 

The situation becomes more complicated for the dual connection $\overset{\ast}{\nabla}$ which has non-vanishing torsion $\overset{\ast}{T}$.

We consider the Einstein tensor for $\overset{\ast}{\nabla}$ given as in \eqref{eq_app1_*}.

We will study now the divergence $\overset{\ast}{\nabla}\,^{i}\overset{\ast}{G}_{ij}$.

{\it \underline {Step 1. Bianchi Identities}}

In the case of a quasi-statistical manifold $(M,g,\nabla,\overset{\ast}{\nabla})$
the Bianchi identity for the dual connection $\overset{\ast}{\nabla}$ reads
$$
\overset{\ast}{\nabla}_i \overset{\ast}{R}\,^{\  h}_{l\ jk}+\overset{\ast}{\nabla}_j\overset{\ast}{R}\,^{\ h}_{l\ ki}+\overset{\ast}{\nabla}_k\overset{\ast}{R}\,^{\ h}_{l\ ij}
+\overset{\ast}{R}\,^{\ h}_{l\ rk}\overset{\ast}{T}\,^{ r}_{\ ij}+\overset{\ast}{R}\,^{\ h}_{l\ ri}\overset{\ast}{T}\,^{ r}_{\ jk}+\overset{\ast}{R}\,^{\ h}_{l\ rj}\overset{\ast}{T}\,^{ r}_{\ ki}=0.
$$

{\it \underline {Step 2. Covariant derivative $\overset{\ast}{\nabla}_k\overset{\ast}{R}$}}

Summing $i$ to $h$ 
implies
$$
\overset{\ast}{\nabla}_h\overset{\ast}{R}\,^{\ h}_{l\ jk}+\overset{\ast}{\nabla}_j\overset{\ast}{R}\,^{\ h}_{l\ kh}+\overset{\ast}{\nabla}_k\overset{\ast}{R}\,^{ h}_{l\ hj}
+\overset{\ast}{R}\,^{\ h}_{l\ rk}\overset{\ast}{T}\,^{ r}_{\ hj}+\overset{\ast}{R}\,^{\ h}_{l\ rh}\overset{\ast}{T}\,^{ r}_{\ jk}+\overset{\ast}{R}\,^{\ h}_{l\ rj}\overset{\ast}{T}\,^{ r}_{\ kh}=0,
$$
i.e.
$$
\overset{\ast}{\nabla}_h\overset{\ast}{R}\,^{ h}_{l\ jk}-\overset{\ast}{\nabla}_j\overset{\ast}{R}_{lk}+\overset{\ast}{\nabla}_k\overset{\ast}{R}_{lj}
+\overset{\ast}{R}\,^{\ h}_{l\ rk}\overset{\ast}{T}\,^{ r}_{\ hj}{-}\overset{\ast}{R}_{lr}\overset{\ast}{T}\,^{ r}_{\ jk}+\overset{\ast}{R}\,^{\ h}_{l\ rj}\overset{\ast}{T}\,^{ r}_{\ kh}=0.
$$

We multiply by $g^{lj}$ and by a similar computation with the torsion-free case, we get

\begin{equation}\label{4.2_2}
g^{lj}\overset{\ast}{\nabla}_h\overset{\ast}{R}\,^{\ h}_{l\ jk}-g^{lj}\overset{\ast}{\nabla}_j\overset{\ast}{R}_{lk}+g^{lj}\overset{\ast}{\nabla}_k\overset{\ast}{R}_{lj}
+g^{lj}\overset{\ast}{R}\,^{\ h}_{l\ rk}\overset{\ast}{T}\,^{ r}_{\ hj}{-}g^{lj}\overset{\ast}{R}_{lr}\overset{\ast}{T}\,^{ r}_{\ jk}+g^{lj}\overset{\ast}{R}\,^{\ h}_{l\ rj}\overset{\ast}{T}\,^{ r}_{\ kh}=0.
\end{equation}

Using now a similar computation as in Remark \ref{rem: some nabla terms} formulated for $\overset{\ast}{\nabla}$ and its curvature, we observe that 
\begin{enumerate}
\item $g^{lj}\overset{\ast}{\nabla}_k\overset{\ast}{R}_{lj}=\overset{\ast}{\nabla}_k(g^{lj}\overset{\ast}{R}_{lj})-(\overset{\ast}{\nabla}_kg^{lj})\overset{\ast}{R}_{lj}$,
\item $g^{lj}\overset{\ast}{\nabla}_h\overset{\ast}{R}\,_{l\ jk}^{\ h}=-(\overset{\ast}{\nabla}_h g^{hs})R_{sk}-g^{hs}\overset{\ast}{\nabla}_hR_{sk}-(\overset{\ast}{\nabla}_hg^{lj})\overset{\ast}{R}\,^{\ h}_{l\ jk}$.
\end{enumerate}

By substitution in \eqref{4.2_2} we get 
\begin{equation*}
\begin{split}
& -(\overset{\ast}{\nabla}_h g^{hs})R_{sk}-g^{hs}\overset{\ast}{\nabla}_hR_{sk}-(\overset{\ast}{\nabla}_hg^{lj})\overset{\ast}{R}\,^{\ h}_{l\ jk}
-g^{lj}\overset{\ast}{\nabla}_j\overset{\ast}{R}_{lk}
+\overset{\ast}{\nabla}_k(g^{lj}\overset{\ast}{R}_{lj})-(\overset{\ast}{\nabla}_kg^{lj})\overset{\ast}{R}_{lj}\\
& +g^{lj}\overset{\ast}{R}^{\ h}_{l\ rk}\overset{\ast}{T}\,^{ r}_{\ hj}{-}g^{lj}\overset{\ast}{R}_{lr}\overset{\ast}{T}\,^{ r}_{\ jk}+g^{lj}\overset{\ast}{R}\,^{\ h}_{l\ rj}\overset{\ast}{T}\,^{r}_{\ kh}=0
\end{split}
\end{equation*}
and from here
\begin{equation*}\label{4.2_3}
\begin{split}
&\overset{\ast}{\nabla}_k(g^{lj}\overset{\ast}{R}_{lj}) =(\overset{\ast}{\nabla}_h g^{hs})R_{sk}+g^{hs}\overset{\ast}{\nabla}_hR_{sk}+(\overset{\ast}{\nabla}_hg^{lj})\overset{\ast}{R}\,^{\ h}_{l\ jk}
+g^{lj}\overset{\ast}{\nabla}_j\overset{\ast}{R}_{lk}
+(\overset{\ast}{\nabla}_kg^{lj})\overset{\ast}{R}_{lj}\\
& -g^{lj}\overset{\ast}{R}\,^{\ h}_{l\ rk}\overset{\ast}{T}\,^{ r}_{\ hj}{+}g^{lj}\overset{\ast}{R}_{lr}\overset{\ast}{T}\,^{r}_{\ jk}-g^{lj}\overset{\ast}{R}\,^{\ h}_{l\ rj}\overset{\ast}{T}\,^{ r}_{\ kh}.
\end{split}
\end{equation*}

{\it \underline {Step 3. Covariant derivative of Einstein tensor $\overset{\ast}{\nabla}_k\overset{\ast}{G}_{ij}$}}

We take the covariant derivative of \eqref{eq_app1_*} with respect to $\overset{\ast}{\nabla}$ and obtain 
$$
\overset{\ast}{\nabla}_k\overset{\ast}{G}_{ij}=\overset{\ast}{\nabla}_k\overset{\ast}{R}_{(ij)}-\frac{1}{2}(\overset{\ast}{\nabla}_kg_{ij})\overset{\ast}{R}-
\frac{1}{2}g_{ij}\overset{\ast}{\nabla}_k\overset{\ast}{R}.
$$

After substituting $\overset{\ast}{\nabla}_k\overset{\ast}{R}$ from \eqref{4.2_3} it results
\begin{equation*}
  \begin{split}
    \overset{\ast}{\nabla}_k\overset{\ast}{G}_{ij}&=\overset{\ast}{\nabla}_k\overset{\ast}{R}_{(ij)}-\frac{1}{2}(\overset{\ast}{\nabla}_kg_{ij})\overset{\ast}{R}\\
                    &  -\frac{1}{2}g_{ij}\left\{
(\overset{\ast}{\nabla}_h g^{hs})R_{sk}+\overset{\ast}{\nabla}\,^{h}R_{hk}+(\overset{\ast}{\nabla}_hg^{ls})\overset{\ast}{R}\,^{\ h}_{l\ sk}
+\overset{\ast}{\nabla}\,^{l}\overset{\ast}{R}_{lk}
+(\overset{\ast}{\nabla}_kg^{ls})\overset{\ast}{R}_{ls}\right.\\
& \left.-g^{ls}\overset{\ast}{R}\,^{\ h}_{l\ rk}\overset{\ast}{T}\,^{r}_{\ hs}{+}g^{ls}\overset{\ast}{R}_{lr}\overset{\ast}{T}\,^{r}_{\ sk}-g^{ls}\overset{\ast}{R}\,^{\ h}_{l\ rs}\overset{\ast}{T}\,^{r}_{\ kh}
    \right\}.
  \end{split}
\end{equation*}

{\it \underline {Step 4. The divergence of Einstein tensor $\overset{\ast}{\nabla}\,^{i}\overset{\ast}{G}_{ij}$}}

We multiply by $g^{pk}$ and  get 

\begin{equation*}
  \begin{split}
   g^{pk} \overset{\ast}{\nabla}_k\overset{\ast}{G}_{ij}&=g^{pk}\overset{\ast}{\nabla}_k\overset{\ast}{R}_{(ij)}-g^{pk}\frac{1}{2}(\overset{\ast}{\nabla}_kg_{ij})\overset{\ast}{R}\\
                    &  -\frac{1}{2}g^{pk}g_{ij}\left\{
(\overset{\ast}{\nabla}_h g^{hs})R_{sk}+\overset{\ast}{\nabla}\,^{h}R_{hk}+(\overset{\ast}{\nabla}_hg^{ls})\overset{\ast}{R}\,^{\ h}_{l\ sk}
+\overset{\ast}{\nabla}\,^{l}\overset{\ast}{R}_{lk}
+(\overset{\ast}{\nabla}_kg^{ls})\overset{\ast}{R}_{ls}\right.\\
& \left.-g^{ls}\overset{\ast}{R}\,^{\ h}_{l\ rk}\overset{\ast}{T}\,^{ r}_{\ hs}{+}g^{ls}\overset{\ast}{R}_{lr}\overset{\ast}{T}\,^{r}_{\ sk}-g^{ls}\overset{\ast}{R}\,^{\ h}_{l\ rs}\overset{\ast}{T}\,^{ r}_{kh}
    \right\}.
  \end{split}
\end{equation*}

Now taking the trace by summing up $p$ to $i$, by a similar computation with the statistical manifolds case, it follows
\begin{equation*}
  \begin{split}
  \overset{\ast}{\nabla}\,^{i}\overset{\ast}{G}_{ij}&=\overset{\ast}{\nabla}\,^{i}\overset{\ast}{R}_{(ij)}-\frac{1}{2}(\overset{\ast}{\nabla}\,^{i}g_{ij})\overset{\ast}{R}\\
                    &  -\frac{1}{2}\left\{
(\overset{\ast}{\nabla}_h g^{hs})R_{sj}+\overset{\ast}{\nabla}\,^{h}R_{hj}+(\overset{\ast}{\nabla}_hg^{ls})\overset{\ast}{R}\,^{\ h}_{l\ sj}
+\overset{\ast}{\nabla}\,^{l}\overset{\ast}{R}_{lj}
+(\overset{\ast}{\nabla}_jg^{ls})\overset{\ast}{R}_{ls}\right.\\
& \left.-g^{ls}\overset{\ast}{R}\,^{\ h}_{l\ rj}\overset{\ast}{T}\,^{ r}_{\ hs}+g^{ls}\overset{\ast}{R}_{lr}\overset{\ast}{T}\,^{ r}_{\ sj}-g^{ls}\overset{\ast}{R}\,^{\ h}_{l\ rs}\overset{\ast}{T}\,^{r}_{\ jh}
  \right\}\\
    &=
-\frac{1}{2}\left\{
\overset{\ast}{\nabla}\,^{i}(R_{ij}-\overset{\ast}{R}_{ji})+(\overset{\ast}{\nabla}\,^{i}g_{ij})\overset{\ast}{R}+
(\overset{\ast}{\nabla}_hg^{ls})\overset{\ast}{R}\,^{\ h}_{l\  sj}+(\overset{\ast}{\nabla}_jg^{ls})\overset{\ast}{R}_{ls}+(\overset{\ast}{\nabla}_hg^{hs})R_{sj}
                              \right\}\\
  &+\frac{1}{2}\left(g^{ls}\overset{\ast}{R}\,^{\ h}_{l\ rj}\overset{\ast}{T}\,^{ r}_{\ hs}-g^{ls}\overset{\ast}{R}_{lr}\overset{\ast}{T}\,^{ r}_{\ sj}{+}g^{ls}\overset{\ast}{R}\,^{\ h}_{l\ rs}\overset{\ast}{T}\,^{ r}_{\ jh}\right),
  \end{split}
\end{equation*}
where $R_{ij}-\overset{\ast}{R}_{ij}$ is given in Corollary 
 \ref{cor_S5} and $R-\overset{\ast}{R}$ can be immediately computed.

We can again consider the mixed divergences of types $\nabla^i \overset{\ast}{G}_{ij}$ and $\overset{\ast}{\nabla}\,^{i}G_{ij}$ as we did in the case of statistical
manifolds. By substracting again formulas \eqref{eq_app1} and \eqref{eq_app1_*}, we get the same formulas as \eqref{nabla star G} and \eqref{nabla G star} and a similar argument as in the case of statistical manifolds leads again to similar formulas.

Summarizing, we formulate

\begin{theorem}\label{thm_app3}
  Let $(M,g,\nabla,\overset{\ast}{\nabla})$ be a quasi-statistical manifold, and $\overset{\ast}{G}_{ij}:=\overset{\ast}{R}_{{(ij)}}-\frac{1}{2}g_{ij}\overset{\ast}{R}$ the Einstein tensor of $\overset{\ast}{\nabla}$. Then the following divergences of the Einstein tensors hold

\begin{equation*}
\begin{split}
\nabla^iG_{ij}&=-\frac{1}{2}\left\{
\nabla^i(\overset{\ast}{R}_{ij}-R_{ji})+(\nabla^ig_{ij})R+(\nabla_hg^{ls})R^{\ h}_{l\ sj}
+(\nabla_jg^{ls})R_{ls}+(\nabla_hg^{hs})\overset{\ast}{R}_{sj}
\right\}\\
\overset{\ast}{\nabla}\,^{i}\overset{\ast}{G}_{ij}&=
-\frac{1}{2}\left\{
\overset{\ast}{\nabla}\,^{ i}(R_{ij}-\overset{\ast}{R}_{ji})+(\overset{\ast}{\nabla}\,^{ i}g_{ij})\overset{\ast}{R}+
(\overset{\ast}{\nabla}_hg^{ls})\overset{\ast}{R}\,^{\ h}_{l\  sj}+(\overset{\ast}{\nabla}_jg^{ls})\overset{\ast}{R}_{ls}+(\overset{\ast}{\nabla}_hg^{hs})R_{sj}
                              \right\}\\
  &+\frac{1}{2}\left(g^{ls}\overset{\ast}{R}\,^{\ h}_{l\ rj}\overset{\ast}{T}\,^{ r}_{\ hs}-g^{ls}\overset{\ast}{R}_{lr}\overset{\ast}{T}\,^{r}_{\ sj}{+}g^{ls}\overset{\ast}{R}\,^{\ h}_{l\ rs}\overset{\ast}{T}\,^{ r}_{\ jh}\right)\\
\overset{\ast}{\nabla}\,^{ i} G_{ij}&=\overset{\ast}{\nabla}\,^{ i} R_{(ij)}-\frac{1}{2}(\overset{\ast}{\nabla}\,^{ i} g_{ij})R
-\frac{1}{2}\overset{\ast}{\nabla}_j R\\
\nabla^i \overset{\ast}{G}_{ij}&=\nabla^i \overset{\ast}{R}_{(ij)}-\frac{1}{2}(\nabla^{i} g_{ij})\overset{\ast}{R}-\frac{1}{2}\nabla_j\overset{\ast}{R}.
\end{split}
\end{equation*}
\end{theorem}

\section{\texorpdfstring{$\alpha$}{alpha}-connections}\label{sect_alpha}

In the present Section we present in a detailed and consistent way the basic differential geometric properties of the $\alpha$-connections, with a special emphasis on the Einstein tensor, and its properties.  

\subsection{Definitions and basic properties}

Using a pair of dual connections $\left(\nabla,\overset{\ast}{\nabla}\right)$ on a (pseudo)-Riemannian manifold, one
defines the parametric family of the connections {$\overset{(\alpha)}{\nabla}$, $%
\alpha \in R$, as 
\begin{equation}\label{def of alpha-conn}
\overset{(\alpha)}{\nabla}_XY:=\frac{1+\alpha}{2}\nabla _XY+\frac{1-\alpha}{2}\overset{\ast}{\nabla}_XY.
\end{equation}
}

It is obvious that {$\overset{(1)}{\nabla}_XY=\nabla_XY$, and $\overset{(-1)}{\nabla}_XY=\overset{\ast}{\nabla}
_XY$.}

\begin{proposition}
If $\nabla $ and $\overset{\ast}{\nabla}$ are dual connections
on a (pseudo-) Riemannian manifold $(M,g)$, {then} the following relation holds good: 
\begin{equation}
\left( \overset{(\alpha )}{\nabla }_{X}g\right) (Y,Z)=\alpha C(X,Y,Z),
\end{equation}%
where $C$ is the cubic form of $\nabla $, and {$\overset{(\alpha )}{\nabla }$ is given by \eqref{def of alpha-conn}}.
\end{proposition}

\begin{proof}
By a straightforward computation we have 
\begin{equation*}
\begin{split}
\left( \overset{(\alpha )}{\nabla }_{X}g\right) (Y,Z) &=
\frac{1+\alpha }{2}\left( \nabla _{X}g\right) (Y,Z)+\frac{1-\alpha }{2}\left( \overset{\ast}{\nabla} _{X}g\right) (Y,Z) \\
&=\frac{1+\alpha }{2}\left\{ Xg(Y,Z)-g\left( \nabla _{X}Y,Z\right) -g\left(Y,\nabla _{X}Z\right) \right\} \\
&+\frac{1-\alpha }{2}\left\{ Xg(Y,Z)-g\left( \overset{\ast}{\nabla} _{X}Y,Z\right)-g\left( Y,\overset{\ast}{\nabla} _{X}Z\right) \right\}   \\
&=\frac{1+\alpha }{2}\left\{ Xg(Y,Z)-g\left( \nabla _{X}Y,Z\right) -g\left(Y,\nabla _{X}Z\right) \right\}  \\
&+\frac{1-\alpha }{2}\Bigg\{ Xg\left( Y,Z\right)   \\
&-Xg(Y,Z)+g\left( Y,\nabla _{X}Z\right) -Xg(Y,Z)+g\left( \nabla_{X}Y,Z\right) \Bigg\}  \\
&=\left( \frac{1+\alpha }{2}-\frac{1-\alpha }{2}\right) \left( \nabla_{X}g\right) (Y,Z)=\alpha C(X,Y,Z), 
\end{split}
\end{equation*}%
where we have used the duality conditions 
\begin{equation*}
Xg\left( Y,Z\right) =g\left( \overset{\ast}{\nabla} _{X}Y,Z\right) +g\left( Y,\nabla
_{X}Z\right) ,
\end{equation*}
\begin{equation*}
Xg\left( Y,Z\right) =g\left( \nabla _{X}Y,Z\right) +g\left( Y,\overset{\ast}{\nabla}
_{X}Z\right) .
\end{equation*}
\end{proof}


\begin{proposition}\label{prop: prop9}
If $\nabla $ and $\overset{\ast}{\nabla}$ are dual connections on a (pseudo-)
Riemannian manifold $\left( M,g\right) $, then
\begin{equation*}
\overset{(\alpha)}{\nabla}_XY=\overset{(0)}{\nabla }_{X}Y-\frac{\alpha }{2}%
K\left( X,Y\right) ,
\end{equation*}%
where $\overset{(0)}{\nabla }$ is the average connection, and $K$ is the
difference tensor.
\end{proposition}

\begin{proof}
By the Definition \ref{def of alpha-conn} we have
\begin{equation*}
\begin{split}
\overset{(\alpha)}{\nabla}_XY &=\frac{1+\alpha }{2}\nabla _{X}Y+\frac{1-\alpha }{2}\overset{\ast}{\nabla} _{X}Y=\frac{1}{2}\left( \nabla _{X}Y+\overset{\ast}{\nabla} _{X}Y\right) +\frac{\alpha }{2}\left( \nabla _{X}Y-\overset{\ast}{\nabla} _{X}Y\right)\\
&=\overset{(0)}{\nabla }_{X}Y-\frac{\alpha }{2}K\left( X,Y\right) .
\end{split}
\end{equation*}
\end{proof}

\begin{remark}
Observe that by putting $\alpha :=1$ and $\alpha :=-1$ in Proposition \ref{prop: prop9}, we obtain that
if $\nabla $ and $\overset{\ast}{\nabla}$ are dual connections on a
Riemannian manifold $\left( M,g\right) $, then
\begin{equation}
\nabla _{X}Y=\overset{(0)}{\nabla }_{X}Y-\frac{1}{2}K(X,Y),  \label{11.1a}
\end{equation}
and 
\begin{equation}
\overset{\ast}{\nabla} _{X}Y=\overset{(0)}{\nabla }_{X}Y+\frac{1}{2}K(X,Y),
\label{11.1b}
\end{equation}
formulas which agree with Proposition \ref{prop: recover nablas}.
\end{remark}

\begin{proposition}\label{prop: alpha torsion}
  If $\nabla $ and $\overset{\ast}{\nabla}$ are dual connections on a (pseudo-)
Riemannian manifold $\left( M,g\right) $, then
\begin{equation*}
\overset{(\alpha)}{T}(X,Y)=\frac{1+\alpha }{2}T(X,Y)+\frac{1-\alpha }{2}\overset{\ast}{T}(X,Y),
\end{equation*}%
where $\overset{(\alpha)}{T}$ and $T$ are the torsions of $\overset{(\alpha)}{\nabla}$ and $%
\nabla $, respectively.
\end{proposition}

\begin{proof}
By definition
\begin{equation*}
  \begin{split}
\overset{(\alpha)}{T}(X,Y) &=\overset{(\alpha)}{\nabla}_XY-\overset{(\alpha)}{\nabla}_YX-[X,Y]
\\
&=\frac{1+\alpha }{2}\nabla _{X}Y+\frac{1-\alpha }{2}\overset{\ast}{\nabla} _{X}Y-%
\frac{1+\alpha }{2}\nabla _{Y}X-\frac{1-\alpha }{2}\overset{\ast}{\nabla} _{Y}X-[X,Y]
\\
&=\frac{1+\alpha }{2}\left( T(X,Y)+[X,Y]\right) +\frac{1-\alpha }{2}\left(
\overset{\ast}{T}(X,Y)+[X,Y]\right) -[X,Y] \\
&=\frac{1+\alpha }{2}T(X,Y)+\frac{1-\alpha }{2}\overset{\ast}{T}(X,Y)+\left( \frac{1+\alpha }{2}+\frac{1-\alpha }{2}%
    -1\right) [X,Y]\\
                           &=\frac{1+\alpha }{2}T(X,Y)+\frac{1-\alpha }{2}\overset{\ast}{T}(X,Y).
                             \end{split}
\end{equation*}%

\end{proof}

The basic properties of statistical and quasi-statistical manifolds leads to

\begin{corollary}
  \begin{enumerate}[(i)]
    \item
  If $(M,g,\nabla)$ is a statistical manifold with dual connection $\overset{\ast}{\nabla}$, then
  \begin{equation*}
\overset{(\alpha)}{T}(X,Y)=0.
\end{equation*}%
\item  If $(M,g,\nabla)$ is a quasi-statistical manifold with dual connection $\overset{\ast}{\nabla}$, then
  \begin{equation*}
\overset{(\alpha)}{T}(X,Y)=\frac{1-\alpha }{2}\overset{\ast}{T}(X,Y).
\end{equation*}
\end{enumerate}
  \end{corollary}


\subsection{\texorpdfstring{The curvature tensor $\overset{(\alpha )}{R}(X,Y)Z$}{The curvature tensor R^{(\alpha)}(X,Y)Z}}\label{sect3*}


We will compute now the curvature tensor $\overset{(\alpha )}{R}(X,Y)Z$ of
the $\alpha $-connection $\overset{(\alpha )}{\nabla }$. 

By definition we have

\begin{equation*}
\begin{split}
\overset{(\alpha )}{R}(X,Y)Z &=\overset{(\alpha )}{\nabla }_{X}\overset{(\alpha )}{\nabla }_{Y}Z-\overset{(\alpha )}{\nabla }_{Y}\overset{(\alpha )}{\nabla }_{X}Z-\overset{(\alpha )}{\nabla }_{[X,Y]}Z \\
&=\left( \frac{1+\alpha }{2}\nabla _{X}+\frac{1-\alpha }{2}\overset{\ast}{\nabla}_{X}\right) \left( \frac{1+\alpha }{2}\nabla _{Y}+\frac{1-\alpha }{2}\overset{\ast}{\nabla} _{Y}\right) Z \\
&-\left( \frac{1+\alpha }{2}\nabla _{Y}+\frac{1-\alpha }{2}\overset{\ast}{\nabla}_{Y}\right) \left( \frac{1+\alpha }{2}\nabla _{X}+\frac{1-\alpha }{2}\overset{\ast}{\nabla} _{X}\right) Z \\
&-\left( \frac{1+\alpha }{2}\nabla _{[ X,Y]}+\frac{1-\alpha }{2}\overset{\ast}{\nabla} _{[ X,Y]}\right) Z.
\end{split}
\end{equation*}

Using the linearity of the connections $\nabla $, $\overset{\ast}{\nabla}$, it
results

\begin{equation*}
  \begin{split}
& \overset{(\alpha )}{R}(X,Y)Z =\left( \frac{1+\alpha }{2}\right) ^{2}\nabla
_{X}\nabla _{Y}Z+\frac{1-\alpha ^{2}}{4}\nabla _{X}\overset{\ast}{\nabla} _{Y}Z+%
\frac{1-\alpha ^{2}}{4}\overset{\ast}{\nabla} _{X}\nabla _{Y}Z+\left( \frac{1-\alpha 
}{2}\right) ^{2}\overset{\ast}{\nabla} _{X}\overset{\ast}{\nabla} _{Y}Z \\
&-\left( \frac{1+\alpha }{2}\right) ^{2}\nabla _{Y}\nabla _{X}Z-\frac{%
1-\alpha ^{2}}{4}\nabla _{Y}\overset{\ast}{\nabla} _{X}Z-\frac{1-\alpha ^{2}}{4}%
\overset{\ast}{\nabla} _{Y}\nabla _{X}Z-\left( \frac{1-\alpha }{2}\right) ^{2}\overset{\ast}{\nabla}_{Y}\overset{\ast}{\nabla} _{X}Z \\
&-\frac{1+\alpha }{2}\nabla _{\lbrack X,Y]}Z-\frac{1-\alpha }{2}\overset{\ast}{\nabla}
   _{\lbrack X,Y]}Z.
   \end{split}
  \end{equation*}

By grouping together the terms, and using the definitions of the curvature
tensors for the connections $\nabla $ and $\overset{\ast}{\nabla}$it follows

\begin{equation*}
\begin{split}
\overset{(\alpha )}{R}(X,Y)Z &=\left( \frac{1+\alpha }{2}\right)^{2}R(X,Y)Z+\left( \frac{1+\alpha }{2}\right) ^{2}\nabla _{[ X,Y]}Z \\
&+\left( \frac{1-\alpha }{2}\right) ^{2}\overset{\ast}{R}(X,Y)Z+\left( \frac{1-\alpha }{2}\right) ^{2}\overset{\ast}{\nabla} _{[ X,Y]}Z \\
&+\frac{1-\alpha ^{2}}{4}\nabla _{X}\overset{\ast}{\nabla} _{Y}Z+\frac{1-\alpha ^{2}}{4}\overset{\ast}{\nabla} _{X}\nabla _{Y}Z-\frac{1-\alpha ^{2}}{4}\nabla _{Y}\overset{\ast}{\nabla}_{X}Z-\frac{1-\alpha ^{2}}{4}\overset{\ast}{\nabla} _{Y}\nabla _{X}Z \\
&-\frac{1+\alpha }{2}\nabla _{[ X,Y]}Z-\frac{1-\alpha }{2}\overset{\ast}{\nabla}_{[ X,Y]}Z.
\end{split}
\end{equation*}

Observe that the coefficients of $\nabla _{[ X,Y]}Z$ and $\overset{\ast}{\nabla}
_{[ X,Y]}Z$ are $\left( \frac{1+\alpha }{2}\right) ^{2}-\frac{%
1+\alpha }{2}=-\frac{1-\alpha ^{2}}{4}$ and $\left( \frac{1-\alpha }{2}%
\right) ^{2}-\frac{1-\alpha }{2}=-\frac{1-\alpha ^{2}}{4}$, respectively.
Hence we get

\begin{equation}\label{13.1}
\begin{split}
\overset{(\alpha )}{R}(X,Y)Z &=\left( \frac{1+\alpha }{2}\right)^{2}R(X,Y)Z+\left( \frac{1-\alpha }{2}\right) ^{2}\overset{\ast}{R}(X,Y)Z\\
&+\frac{1-\alpha ^{2}}{4}{\left\{ \nabla _{X}\overset{\ast}{\nabla} _{Y}+\overset{\ast}{\nabla}_{X}\nabla _{Y}-\nabla _{Y}\overset{\ast}{\nabla} _{X}-\overset{\ast}{\nabla} _{Y}\nabla _{X}-\nabla _{[X,Y]}-\overset{\ast}{\nabla} _{[ X,Y]}\right\} Z}.
\end{split}
\end{equation}

We will compute now the terms in the parenthesis $${\left\{ \nabla _{X}\overset{\ast}{\nabla} _{Y}+\overset{\ast}{\nabla}_{X}\nabla _{Y}-\nabla _{Y}\overset{\ast}{\nabla} _{X}-\overset{\ast}{\nabla} _{Y}\nabla _{X}-\nabla _{\lbrack X,Y]}-\overset{\ast}{\nabla} _{[ X,Y]}\right\} Z},$$
that is,
\begin{equation*}
\begin{split}
\rho _{X,Y}Z &=\left({ \nabla _{X}\overset{\ast}{\nabla} _{Y}+\overset{\ast}{\nabla} _{X}\nabla _{Y}-\nabla _{Y}\overset{\ast}{\nabla} _{X}-\overset{\ast}{\nabla} _{Y}\nabla _{X}-\nabla _{\lbrack X,Y]}-\overset{\ast}{\nabla}_{[ X,Y]}}\right) Z \\
&=\nabla _{X}\left( \overset{(0)}{\nabla }_{Y}Z+{\frac{1}{2}K(Y,Z)}\right)+\overset{\ast}{\nabla} _{X}\left( \overset{(0)}{\nabla }_{Y}Z-{\frac{1}{2}K(Y,Z)}\right) \\
&-\nabla _{Y}\left( \overset{(0)}{\nabla }_{X}Z+\frac{1}{2}K(X,Z)\right)-\overset{\ast}{\nabla} _{Y}\left( \overset{(0)}{\nabla }_{X}Z-\frac{1}{2}K(X,Z)\right) \\
&-\overset{(0)}{\nabla }_{[X,Y]}Z+\frac{1}{2}K([X,Y],Z)-\overset{(0)}{\nabla }_{[X,Y]}Z-\frac{1}{2}K([X,Y],Z) \\
&=\overset{(0)}{\nabla }_{X}\left( \overset{(0)}{\nabla }_{Y}Z+\frac{1}{2}K(Y,Z)\right) -\frac{1}{2}K\left( X,\overset{(0)}{\nabla }_{Y}Z{+\frac{1}{2}K(Y,Z)}\right) \\
&+\overset{(0)}{\nabla }_{X}\left( \overset{(0)}{\nabla }_{Y}Z-\frac{1}{2}K(Y,Z)\right) +\frac{1}{2}K\left( X,\overset{(0)}{\nabla }_{Y}Z{-\frac{1}{2}K(Y,Z)}\right) \\
&-\overset{(0)}{\nabla }_{Y}\left( \overset{(0)}{\nabla }_{X}Z+\frac{1}{2}K(X,Z)\right) +\frac{1}{2}K\left( Y,\overset{(0)}{\nabla }_{X}Z+\frac{1}{2}K(X,Z)\right) \\
&-\overset{(0)}{\nabla }_{Y}\left( \overset{(0)}{\nabla }_{X}Z-\frac{1}{2}K(X,Z)\right) -\frac{1}{2}K\left( Y,\overset{(0)}{\nabla }_{X}Z-\frac{1}{2}K(X,Z)\right) \\
&{-2\overset{(0) }{\nabla }_{[X,Y]}Z},
\end{split}
\end{equation*}
where we have used consecutively the formulas (\ref{11.1a}) and (\ref{11.1b}%
), respectively.

By linearity we get
\begin{equation*}
\begin{split}
\rho _{X,Y}Z &=\overset{(0)}{\nabla }_{X}\overset{(0)}{\nabla }_{Y}Z+\frac{1}{2}\overset{(0)}{\nabla }_{X}K(Y,Z)-\frac{1}{2}K\left( X,\overset{(0)}{\nabla }_{Y}Z\right) -\frac{1}{4}K\left( X,K\left( Y,Z\right) \right) \\
&+\overset{(0)}{\nabla }_{X}\overset{(0)}{\nabla }_{Y}Z{-\frac{1}{2}\overset{(0)}{\nabla }_{X}K(Y,Z)}+\frac{1}{2}K\left( X,\overset{(0)}{\nabla }_{Y}Z\right) -\frac{1}{4}K\left( X,K\left( Y,Z\right) \right) \\
&-\overset{(0)}{\nabla }_{Y}\overset{(0)}{\nabla }_{X}Z-\frac{1}{2}\overset{(0)}{\nabla }_{Y}K(X,Z)+\frac{1}{2}K\left( Y,\overset{(0)}{\nabla }_{X}Z\right) +\frac{1}{4}K\left( Y,K\left( X,Z\right) \right) \\
&-\overset{(0)}{\nabla }_{Y}\overset{(0)}{\nabla }_{X}Z+\frac{1}{2}\overset{(0)}{\nabla }_{Y}K(X,Z)-\frac{1}{2}K\left( Y,\overset{(0)}{\nabla }_{X}Z\right) +\frac{1}{4}K\left( Y,K\left( X,Z\right) \right) \\
&{-2\overset{(0) }{\nabla }_{[X,Y]}Z}\\
&{=2\overset{(0)}{R}(X,Y)Z-\frac{1}{2}(X,K(Y,Z))+\frac{1}{2}K(Y,K(X,Z))}.
\end{split}
\end{equation*}

By substituting back this relation in (\ref{13.1}) it follows
\begin{equation}\label{15.1}
\begin{split}
\overset{(\alpha )}{R}(X,Y)Z &=\left( \frac{1+\alpha }{2}\right)^{2}R(X,Y)Z+\left( \frac{1-\alpha }{2}\right) ^{2}\overset{\ast}{R}(X,Y)Z\\
&+\frac{1-\alpha ^{2}}{2}\left\{{\overset{(0)}{R}(X,Y)Z}-\frac{1}{4}K\left( X,K(Y,Z)\right)+\frac{1}{4}K\left( Y,K\left( X,Z\right) \right) \right\} .
\end{split}
\end{equation}

In the case $\alpha =0$ this relation reads
\begin{equation*}
\overset{(0)}{R}(X,Y)Z=\frac{1}{4}R(X,Y)Z+{\frac{1}{4}\overset{\ast}{R}(X,Y)Z}+\frac{1}{2}\left\{{\overset{(0)}{R}(X,Y)Z}-\frac{1}{4}K\left( X,K(Y,Z)\right) +\frac{1}{4}K\left(
Y,K\left( X,Z\right) \right) \right\} .
\end{equation*}

Therefore we obtain
\begin{equation*}
\frac{1}{2}\overset{(0)}{R}(X,Y)Z=\frac{1}{4}R(X,Y)Z+\frac{1}{4}\overset{\ast}{R}(X,Y)Z+\frac{1}{8}\left\{ K\left( Y,K\left( X,Z\right) \right) -K\left(
X,K(Y,Z)\right) \right\} ,
\end{equation*}%
i.e.
\begin{equation*}
{\overset{(0)}{R}(X,Y)Z}=\frac{1}{2}R(X,Y)Z+\frac{1}{2}\overset{\ast}{R}(X,Y)Z+\frac{1}{4}%
\left\{ K\left( Y,K\left( X,Z\right) \right) -K\left( X,K(Y,Z)\right)
\right\} .
\end{equation*}

Substituting back to (\ref{15.1}) it results
\begin{equation*}
\begin{split}
\overset{(\alpha )}{R}(X,Y)Z &=\left( \frac{1+\alpha }{2}\right)^{2}R(X,Y)Z+\left( \frac{1-\alpha }{2}\right) ^{2}\overset{\ast}{R}(X,Y)Z \\
&+\frac{1-\alpha ^{2}}{2}\Big\{\frac{1}{2}R(X,Y)Z+\frac{1}{2}\overset{\ast}{R}(X,Y)Z+\frac{1}{4}K\left( Y,K\left( X,Z\right) \right) -\frac{1}{4}K\left(X,K(Y,Z)\right) \\
&+\frac{1}{4}K\left( Y,K\left( X,Z\right) \right) -\frac{1}{4}K\left(X,K(Y,Z)\right) \Big\} \\
&=\frac{1+\alpha }{2}R(X,Y)Z+\frac{1-\alpha }{2}\overset{\ast}{R}(X,Y)Z\\
&+\frac{1-\alpha ^{2}}{4}\left\{ K\left( Y,K\left( X,Z\right) \right){-K\left( X,K(Y,Z)\right) }\right\} .
\end{split}
\end{equation*}

Therefore we obtain the following proposition

\begin{proposition}\label{Prop_R_alpha}
If $\nabla $ and $\overset{\ast}{\nabla}$ are dual connections on a
(pseudo-) Riemannian manifold $\left( M,g\right) $ then
\begin{equation}\label{16.1}
\begin{split}
\overset{(\alpha )}{R}(X,Y)Z&=\frac{1+\alpha }{2}R(X,Y)Z+\frac{1-\alpha }{2}%
                              \overset{\ast}{R}(X,Y)Z\\
  &+\frac{1-\alpha ^{2}}{4}\left\{ K\left( Y,K\left( X,Z\right)
\right) {-K\left( X,K(Y,Z)\right)} \right\} .  
\end{split}
\end{equation}
\end{proposition}

In a local coordinate system $\left( U,q,R^{n}\right) $ on $M$ with
coordinates $x=\left( x^{1},x^{2},...,x^{n}\right) $

the relation (\ref{16.1}) leads in local coordinates to

\begin{equation*}
  \overset{(\alpha )}{R}\, _{l\ ij}^{\ k}=\frac{1+\alpha }{2}R_{l\ ij}^{\ k}
  +\frac{1-\alpha }{2}\overset{\ast}{R}\,_{l\ ij}^{\ k}+\frac{1-\alpha ^{2}}{4}\left(
K_{\ il}^{m}K_{\ jm}^{k}-K_{\ jl}^{m}K_{\ im}^{k}\right) .
\end{equation*}

The Ricci curvature tensors are

\begin{equation*}
\overset{(\alpha )}{R}_{lj}=\overset{(\alpha )}{R}\, _{l\ ij}^{\ i}=\frac{1+\alpha 
}{2}R_{lj}+\frac{1-\alpha }{2}\overset{\ast}{R}_{lj}+\frac{1-\alpha ^{2}}{4}\left(
K_{\ il}^{m}K_{\ jm}^{i}-K_{\ jl}^{m}K_{\ im}^{i}\right) .
\end{equation*}

By multiplying with $g^{lj}$, we get
\begin{equation*}
\overset{(\alpha )}{R}=\frac{1+\alpha 
}{2}R+\frac{1-\alpha }{2}\overset{\ast}{R}+\frac{1-\alpha ^{2}}{4}g^{lj}\left(
K_{\ il}^{m}K_{\ jm}^{i}-K_{\ jl}^{m}K_{\ im}^{i}\right) .
\end{equation*}

We can conclude
\begin{theorem}\label{alpha-Ricci}
If $\nabla $ and $\overset{\ast}{\nabla}$ are dual connections on a
(pseudo-) Riemannian manifold $\left( M,g\right) $ then the Ricci curvature tensors and the Ricci scalar of
$\overset{(\alpha)}{\nabla}$ are
\begin{equation}
  \overset{(\alpha )}{R}_{lj}
  =\frac{1+\alpha}{2}R_{lj}+\frac{1-\alpha }{2}\overset{\ast}{R}_{lj}+\frac{1-\alpha ^{2}}{4}\mathcal{K}_{lj}
\end{equation}
and
\begin{equation}
\overset{(\alpha )}{R}=\frac{1+\alpha 
}{2}R+\frac{1-\alpha }{2}\overset{\ast}{R}+\frac{1-\alpha ^{2}}{4}\mathcal{K}
\end{equation}
where
\begin{equation}
  \mathcal{K}_{lj}:=K_{\ il}^{m}K_{\ jm}^{i}-K_{\ im}^{i}K_{\ jl}^{m},\quad 
  \mathcal{K}:=g^{lj}\mathcal{K}_{lj}.
  \end{equation}
  \end{theorem}

\subsection{Equiaffine connections on $\alpha$-statistical manifolds}\label{sect4}

\begin{proposition}
Let $\nabla $ and $\overset{\ast}{\nabla}$ be dual affine connections
on $\left( M,g\right) $. If $\nabla $ is equiaffine with parallel volume $%
\omega =\lambda (x)\cdot dx^{1}\wedge dx^{2}\wedge ...\wedge dx^{n}$, then $%
\overset{(\alpha)}{\nabla}$ is also equiaffine with parallel volume $\overset{(\alpha)}{\omega}=\lambda ^{\alpha }|g|^{\frac{1-\alpha }{2}}dx^{1}\wedge
dx^{2}\wedge ...\wedge dx^{n}$.
\end{proposition}

\begin{proof}
We know from Proposition \ref{prop: equiaffine dual conn} that if $\nabla $ is equiaffine with
parallel volume $\omega $ then $\overset{\ast}{\omega}$ is also equiaffine with
parallel volume $\overset{\ast}{\omega}=\frac{|g|}{\lambda }dx^{1}\wedge
dx^{2}\wedge ...\wedge dx^{n}$. It is trivial to observe that
\begin{equation*}
\begin{split}
\frac{1+\alpha }{2}\Gamma _{\ ik}^{i}+\frac{1-\alpha }{2}\overset{\ast}{\Gamma}\, _{\ ik}^{ i}
&=\frac{1+\alpha }{2}\frac{\partial }{\partial x^{k}}\log \lambda +\frac{1-\alpha }{2}\frac{\partial }{\partial x^{k}}\log \frac{|g|}{\lambda } \\
&=\frac{\partial }{\partial x^{k}}\log \lambda ^{\frac{1+\alpha }{2}}\frac{|g|^{\frac{1-\alpha }{2}}}{\lambda ^{\frac{1-\alpha }{2}}}=\frac{\partial }{\partial x^{k}}\log \lambda ^{\alpha }|g|^{\frac{1-\alpha }{2}}.
\end{split}
\end{equation*}
\end{proof}

Let us observe that Theorem \ref{thm_A13.19} combined with Proposition \ref{prop: alpha torsion} implies
\begin{equation}
  \begin{split}\label{alpha nabla T comput}
  \overset{(\alpha)}{\nabla}_k\overset{(\alpha)}{T}\,^k_{\ ji}&=\frac{1+\alpha}{2}{\nabla}_k\overset{(\alpha)}{T}\,^k_{\ ji}+\frac{1-\alpha}{2}\overset{\ast}{\nabla}_k\overset{(\alpha)}{T}\,^k_{\ ji}
  =\frac{1+\alpha}{2}{\nabla}_k\left(\frac{1+\alpha}{2}{T}^k_{\ ji}+\frac{1-\alpha}{2}\overset{\ast}{T}\,^{k}_{\ ji}\right)\\
                                                          &+\frac{1-\alpha}{2}\overset{\ast}{\nabla}_k\left(\frac{1+\alpha}{2}{T}^k_{\ ji}+\frac{1-\alpha}{2}\overset{\ast}{T}\,^{k}_{\ ji}\right)\\
    &=\left(\frac{1+\alpha}{2}\right)^2{\nabla}_k{T}^k_{\ ji}+\frac{1-\alpha^2}{4}\left( \overset{\ast}{\nabla}_k{T}^k_{\ ji}+{\nabla}_k\overset{\ast}{T}\,^{k}_{\ ji}\right)+\left(\frac{1-\alpha}{2}\right)^2\overset{\ast}{\nabla}_k\overset{\ast}{T}\,^{k}_{\ ji}.
    \end{split}
  \end{equation}

  Next, observe that the definition \eqref{def of alpha-conn} gives
  \begin{equation}\label{partial alpha Gammas}
\frac{\partial\overset{(\alpha)}{\Gamma}\,^k_{\ jk}}{\partial x^i}-\frac{\partial\overset{(\alpha)}{\Gamma}\,^k_{\ ik}}{\partial x^j}=\frac{1+\alpha}{2}\left(\frac{\partial{\Gamma}^k_{\ jk}}{\partial x^i}-\frac{\partial{\Gamma}^k_{\ ik}}{\partial x^j}\right)+\frac{1-\alpha}{2}\left(\frac{\partial\overset{\ast}{\Gamma}\,^{k}_{\ jk}}{\partial x^i}-\frac{\partial\overset{\ast}{\Gamma}\,^{k}_{\ ik}}{\partial x^j}\right).
    \end{equation}

    Using now \eqref{alpha nabla T comput} and \eqref{partial alpha Gammas}, we get
    \begin{equation}
      \begin{split}
        \overset{(\alpha)}{\nabla}_k\overset{(\alpha)}{T}\,^k_{\ ji}&+\left(\frac{\partial\overset{(\alpha)}{\Gamma}\,^k_{\ jk}}{\partial x^i}-\frac{\partial\overset{(\alpha)}{\Gamma}\,^k_{\ ik}}{\partial x^j}\right)=\left(\frac{1+\alpha}{2}\right)^2{\nabla}_k{T}^k_{\ ji}+\frac{1-\alpha^2}{4}\left( \overset{\ast}{\nabla}_k{T}^k_{\ ji}+{\nabla}_k\overset{\ast}{T}\,^{k}_{\ ji}\right)\\
        &+\left(\frac{1-\alpha}{2}\right)^2\overset{\ast}{\nabla}_k\overset{\ast}{T}\,^{k}_{\ ji}+\frac{1+\alpha}{2}\left(\frac{\partial{\Gamma}^k_{\ jk}}{\partial x^i}-\frac{\partial{\Gamma}^k_{\ ik}}{\partial x^j}\right)+\frac{1-\alpha}{2}\left(\frac{\partial\overset{\ast}{\Gamma}\,^{k}_{\ jk}}{\partial x^i}-\frac{\partial\overset{\ast}{\Gamma}\,^{k}_{\ ik}}{\partial x^j}\right).
        \end{split}
      \end{equation}

      Therefore, Theorem \ref{thm_A13.19} reads
      \begin{theorem}\label{alpha thm_A13.19}
If $\nabla$ is an affine connection with torsion $T$, then
\begin{equation}
  \begin{split}
\overset{(\alpha)}{R}_{ij}-\overset{(\alpha)}{R}_{ji}&=\left(\frac{1+\alpha}{2}\right)^2{\nabla}_k{T}^k_{\ ji}+\frac{1-\alpha^2}{4}\left( \overset{\ast}{\nabla}_k{T}^k_{\ ji}+{\nabla}_k\overset{\ast}{T}\,^{k}_{ji}\right)\\
        &+\left(\frac{1-\alpha}{2}\right)^2\overset{\ast}{\nabla}_k\overset{\ast}{T}\,^{k}_{ji}+\frac{1+\alpha}{2}\left(\frac{\partial{\Gamma}^k_{\ jk}}{\partial x^i}-\frac{\partial{\Gamma}^k_{\ ik}}{\partial x^j}\right)+\frac{1-\alpha}{2}\left(\frac{\partial\overset{\ast}{\Gamma}\,^{k}_{\ jk}}{\partial x^i}-\frac{\partial\overset{\ast}{\Gamma}\,^{k}_{\ ik}}{\partial x^j}\right).
  \end{split}
  \end{equation}
\end{theorem}

\begin{corollary}\label{cor: alpha Ricci diff by torsion}
If $\overset{(\alpha)}{\nabla}$ is an equiaffine connection with volume form 
$\omega=\lambda dx^1\wedge \dots \wedge dx^n$, $\lambda>0$, then  
$$
\overset{(\alpha)}{R}_{ij}-\overset{(\alpha)}{R}_{ji}=\overset{(\alpha)}{\nabla}_k\overset{(\alpha)}{T}\,^k_{\ ji}
-\left(\frac{\partial \overset{(\alpha)}{T}\,^k_{\ jk}}{\partial x^i}-\frac{\partial \overset{(\alpha)}{T}\,^k_{\ ik}}{\partial x^j}\right).
$$
\end{corollary}

\begin{corollary}
If $(M,g,\nabla)$ is a statistical manifold with dual connection $\overset{\ast}{\nabla}$, and $Tr_1(\Gamma)=Tr_1(\overset{\ast}{\Gamma})=0$, then the Ricci curvature tensors of $\overset{(\alpha)}{\nabla}$ are symmetric, i.e.
$
\overset{(\alpha)}{R}_{ij}=\overset{(\alpha)}{R}_{ji}.
$
\end{corollary}

\begin{corollary}\label{cor_A13.21*}
  If $(M,g,\nabla)$ is a statistical manifold with dual connection $\overset{\ast}{\nabla}$, and $\overset{(\alpha)}{\nabla}$ 
is equiaffine with volume form 
$\omega=\lambda dx^1\wedge \dots \wedge dx^n$, $\lambda>0$, then the Ricci curvature tensors  of $\overset{(\alpha)}{\nabla}$ are symmetric, i.e.
$
\overset{(\alpha)}{R}_{ij}=\overset{(\alpha)}{R}_{ji}.
$
\end{corollary}

\subsection{The Einstein tensor}\label{sect5}

We define the Einstein tensor as
\begin{equation}
\overset{(\alpha)}{G}_{ij}=\overset{(\alpha)}{R}_{(ij)}-\frac{1}{2}g_{ij}\overset{(\alpha)}{R}. 
  \end{equation}

By using Theorem \ref{alpha-Ricci} we can now calculate the expression of the Einstein tensor
\begin{equation*}
\begin{split}
\overset{(\alpha)}{R}_{(ij)}-\frac{1}{2}g_{ij}\overset{(\alpha)}{R} 
&=\frac{1+\alpha }{2}R_{(ij)}+\frac{1-\alpha }{2}\overset{\ast}{R}_{(ij)}+\frac{1-\alpha ^{2}}{4}\mathcal{K}_{(ij)}  \\
&-\frac{1}{2}g_{ij}\left[ \frac{1+\alpha }{2}R+\frac{1-\alpha }{2}\overset{\ast}{R}+\frac{1-\alpha ^{2}}{4}{\mathcal{K}}\right]  \\
&=\frac{1+\alpha }{2}\left( R_{(ij)}-\frac{1}{2}g_{ij}R\right) +\frac{1-\alpha }{2}\left( \overset{\ast}{R}_{(ij)}-\frac{1}{2}g_{ij}\overset{\ast}{R}\right)  \\
&+\frac{1-\alpha ^{2}}{4}\left(\mathcal{ K}_{(ij)}-\frac{1}{2}g_{ij}{\mathcal{K}}\right) ,
\end{split}
\end{equation*}
where  
\begin{equation*}
\mathcal{K}_{(ij)}:=\frac{1}{2}\left(\mathcal{K}_{ij}+\mathcal{K}_{ji}\right)=\left(K^l_{\ mi}K^m_{\ jl}+K^m_{\ lj}K^l_{\ im}   \right)-K^l_{\ lm}\left( K^m_{\ ij}+K^m_{\ ji}  \right).
\end{equation*}

Hence we obtain the following 

\begin{theorem}
For the dual connections $\nabla $ and $\overset{\ast}{\nabla}$the
Einstein equations in vacuum of the $\alpha$-connection $\overset{(\alpha)}\nabla$ are given by
\begin{equation*}
\overset{(\alpha)}{G}_{ij}=\frac{1+\alpha }{2}G_{ij}+\frac{1-\alpha }{2}\overset{\ast}{G}_{ij}+\frac{1-\alpha ^{2}}{4}H_{ij},
\end{equation*}
where $G_{ij}$ and $\overset{\ast}{G}_{ij}$ are the Einstein tensors of $\nabla$ and $\overset{\ast}{\nabla}$, respectively,
and we have denoted
\begin{equation*}
H_{ij}:=\mathcal{K}_{(ij)}-\frac{1}{2}g_{ij}{\mathcal{K}}.
\end{equation*}%
\end{theorem}

We will consider now the divergence $\overset{(\alpha)}{\nabla}{}^{\!i}(G_{ij})$ of the $\alpha$-Einstein tensor $\overset{(\alpha)}{G}_{ij}$. We have
\begin{equation}
  \begin{split}
  \overset{(\alpha)}{\nabla}{}^{\! i}\overset{(\alpha)}{G}_{\!ij}&=\frac{1+\alpha}{2}\nabla^i\overset{(\alpha)}{G}_{ij}+
  \frac{1-\alpha}{2}{\overset{\ast}{\nabla}\,^i}\overset{(\alpha)}{G}_{ij}\\
  &=\frac{1+\alpha}{2}\nabla^i\left[ \frac{1+\alpha }{2}G_{ij}+\frac{1-\alpha }{2}\overset{\ast}{G}_{ij}+\frac{1-\alpha ^{2}}{4}H_{ij}  \right]\\
    &+
  \frac{1-\alpha}{2}{\overset{\ast}{\nabla}\,^i}\left[ \frac{1+\alpha }{2}G_{ij}+\frac{1-\alpha }{2}\overset{\ast}{G}_{ij}+\frac{1-\alpha ^{2}}{4}H_{ij}  \right]\\
                                                                 &=\left(\frac{1+\alpha}{2}\right)^2\nabla^iG_{ij}+\frac{1-\alpha^2}{4}\left(\overset{\ast}{\nabla}\,^{i}G_{ij}+\nabla^i\overset{\ast}{G}_{ij} \right)
                                                                   +\left(\frac{1-\alpha}{2}\right)^2\overset{\ast}{\nabla}\,^{i}\overset{\ast}{G}_{ij}\\
    &+\frac{1-\alpha^2}{4}\left(\nabla^{i}H_{ij}+\overset{\ast}{\nabla}\,^{i}H_{ij} \right).
    \end{split}
  \end{equation}
  
\section{Discussions and final remarks}\label{sect6}

In the present paper, we have presented in detail the mathematical properties of the statistical and quasi-statistical manifolds. In this regard our main goal is to provide some geometric results that could be helpful from the point of view of the physical applications. In particular we have studied in detail the Einstein tensor as defined on statistical and quasi-statistical manifolds, and its basic geometric properties. Hopefully, the presented results could find some applications in the field of gravitational theories, leading to the formulation and exploration of new perspective on the gravitational interaction.  

Quasi-statistical manifolds are important geometrical and mathematical structures, which are the generalizations of the geometric concept of Statistical Manifolds. Statistical Manifolds  were first introduced for the geometric description of stochastic processes \cite{Amari}. However, after being reformulated in a rigorous mathematical and  differential geometric structure, Statistical Manifolds have found a large number of applications in various branches of mathematics, engineering, and physics, thus becoming an active and dynamic field of research. 

Quasi-statistical manifolds are also very interesting mathematical objects, which extend and generalize in a non-trivial way the properties of the Statistical Manifolds, by significantly enlarging the geometric framework of the theory. However,  they have been less studied from the point of view of their  applications in natural sciences. The preliminary investigations performed in \cite{Iosifidis2} and \cite{Csillag} have shown that the presence of torsion in the Statistical Manifolds may have very important implications from the point of view of the description of the gravitational interaction and of the cosmological phenomena.

In \cite{Iosifidis2} a gravitational theory based on the action 
  \begin{equation}
  S=\frac{1}{4\kappa}\int{\left(R^{(1)}+R^{(2)}+K\right)\sqrt{-g}d^4x}, 
  \end{equation}
 was proposed,  where $K:=K^{\lambda \mu \nu}K_{\mu \nu \lambda}-K^{\lambda \mu}_{\;\;\;\;\mu}K^{\lambda \nu}_{\;\;\;\nu}$ is the difference scalar. For mathematical and phsyical consistency this theory assumes the existence of two hypermomenta, defined according to 
  \begin{equation}
  \Delta _\lambda ^{\;\;\;\mu \nu (1)}=-\frac{2}{\sqrt{-g}}\frac{ \delta S_m}{\delta \Gamma ^{\lambda\;\;\;\;\;(1)}_{\;\;\;\mu \nu}}=\Xi _\lambda ^{\;\;\mu \nu}, 
  \end{equation}
  and 
  \begin{equation}
  \Delta _\lambda ^{\;\;\;\mu \nu (2)}=-\frac{2}{\sqrt{-g}}\frac{ \delta S_m}{\delta \Gamma ^{\lambda\;\;\;\;\;(2)}_{\;\;\;\mu \nu}}=-\Xi _\lambda ^{\;\;\mu \nu}, 
  \end{equation}
  respectively. 
  
  By assuming that the hypermomentum tensor $\Xi_{\alpha \mu \nu}$ is traceless and totally symmetric, the connection coefficients can be represented as 
  \begin{equation}
   \Gamma ^{\lambda\;\;\;\;(1)} _{\;\;\; \mu \nu}=\tilde{\Gamma}_{\;\;\mu \nu}^\lambda +\kappa \Xi ^{\lambda}_{\;\;\mu \nu}, 
  \end{equation}
  and 
  \begin{equation}
  \Gamma ^{\lambda\;\;\;\;(2)} _{\;\;\; \mu \nu}=\tilde{\Gamma}_{\;\;\mu \nu}^\lambda -\kappa \Xi ^{\lambda}_{\;\;\mu \nu}, 
  \end{equation}
  respectively. The gravitational field equations for this gravitational theory take the form \cite{Iosifidis2}
 \begin{equation}
    \tilde{R}_{\mu \nu}-\frac{1}{2}g_{\mu \nu}\tilde{R}=\kappa T_{\mu \nu}-\kappa ^2\left(\Xi ^{\alpha \beta}_{\;\;\mu}\Xi _{\alpha \beta \mu}-\frac{1}{2}\Xi^{\alpha \beta \gamma}\Xi_{\alpha \beta \gamma}g_{\mu \nu}\right).
\end{equation}

For the case of a quasi-static manifold the Einstein gravitational field equations can be formulated as
\begin{equation}
\overset{(\alpha)}{G}_{\!ij}=\frac{1+\alpha}{2}G_{ij}+\frac{1-\alpha}{2}G^\ast_{ij}+\frac{1-\alpha^2}{4}H_{ij}=\frac{8\pi G}{c^4}T_{ij},
\end{equation}
where  the matter energy momentum tensor $T_{ij}$ has been defined in a way similar to standard general relativity. These field equations can be reformulated as effective Einstein equations, given by
\begin{equation}
G_{ij}=\frac{8\pi G}{c^4}T_{ij}^{(eff)},
\end{equation}
where 
\begin{equation}
T_{ij}^{(eff)}=\frac{2}{1+\alpha}T_{ij}-\frac{c^4}{8\pi G}\left(\left(1-\alpha ^2\right)G^\ast _{ij}-\frac{(1+\alpha)^2(1-\alpha)}{2}H_{ij}\right).
\end{equation}

Hence, the modified effective energy-momentum tensor $T_{ij}^{(eff)}$ contains some extra geometric contributions coming from the quasi-statistical manifold structure of the theory, which can be interpreted as describing a geometric dark energy and, presumably, a dark matter component. It is interesting to note that while $G_{ij}^\ast $ is the Einstein tensor defined with the help of the dual connection $\nabla ^\ast$, the $H_{ij}$ tensor is defined with the help of the difference tensor $K$. Therefore, the Einstein field equations defined on a quasi-statistical manifold  could give a geometric explanation of the origin of the two fundamental components of the Universe, dark energy and dark matter, respectively. 

In the investigations of the mathematical properties of the Statistical Manifolds one routinely assumes the condition of the vanishing of the torsion tensors $T$ and $T^*$ of the connection $\nabla$ and of its dual $\nabla ^\ast$, respectively. In the present work, we have presented in detail the mathematical foundations of the quasi- statistical manifolds, or of the information geometries with torsion. These types of mathematical structure, the quasi-statistical manifolds \cite{Kurose}, significantly enlarge the field of information geometry, and they also open some new and important  perspectives for mathematical, physical or engineering applications.  In our study  we have calculated the curvature tensors, and the curvature scalars of the quasi-statistical manifolds, thus obtaining all the mathematical and theoretical tools necessary to build novel gravitational field theories.

In our study we have proved a number of Theorems that refer to geometries consisting of a connection and its torsion dual.  If $\nabla$ and its  dual $\nabla ^\ast$ are given, then one can introduce a parametric family of connections $\nabla^{(\alpha)}$, which are given as a linear combination of the connections $\nabla$ and $\nabla ^\ast$. These connections generate an interesting but complex mathematical structure. Moreover, similarly to the case of the Statistical Manifolds, a totally symmetric tensor, called in information geometry the cubic tensor, can also be introduced. On quasi-statistical manifolds the cubic tensor is related to the torsion tensor,  an important property that significantly distinguishes statistical and quasi-statistical manifolds.  

The main  reason for studying different geometric models in gravity is to understand more about the specific phases of the history of the Universe. The direct astrophysical and cosmological applications of the present work could offer a new perspective on the geometric description of the gravitational interaction. In the present study we have provided some of the basic mathematical concepts and results necessary for the in depth physical exploration of the role of torsion and complex geometric structures in the understanding of the dynamics and evolution of the Universe.

\section*{Acknowledgments}

This research was partially supported financially by the Prince of Songkla University, Surat Thani Campus Collaborative Research Fund 005/2568.

\section*{ORCID}

\noindent Rattanasak Hama - \url{https://orcid.org/0000-0002-7303-1495}

\noindent Tiberiu Harko - \url{https://orcid.org/0000-0002-1990-9172}

\noindent Sorin V. Sabau - \url{https://orcid.org/0000-0003-4620-2620}

\section{Appendix}\label{App}

In the present Appendix we present the explicit  calculations of the results of the curvature  

\subsection{Curvature computations}\label{App1}

\subsubsection{The proof of Theorem \ref{thm_C1}}

From relation \eqref{eq_Diff1} and \eqref{eq_Av1} we get
$$
\nabla_XY=\overset{(0)}{\nabla}_XY-\frac{1}{2}K(X,Y)
$$
which implies
\begin{equation*}
\begin{split}
R(X,Y)Z&=\nabla_X\nabla_YZ-\nabla_Y\nabla_XZ-\nabla_{[X,Y]}Z\\
&=\nabla_X(\overset{(0)}{\nabla}_YZ-\frac{1}{2}K(Y,Z))-\nabla_Y(\overset{(0)}{\nabla}_XZ-\frac{1}{2}K(X,Z))-\overset{(0)}{\nabla}_{[X,Y]}Z+\frac{1}{2}K([X,Y],Z)\\
&=\overset{(0)}{\nabla}_X(\overset{(0)}{\nabla}_YZ-\frac{1}{2}K(Y,Z))-\frac{1}{2}K(X,\overset{(0)}{\nabla}_YZ-\frac{1}{2}K(Y,Z))\\
&-\overset{(0)}{\nabla}_Y(\overset{(0)}{\nabla}_XZ-\frac{1}{2}K(X,Z))+\frac{1}{2}K(Y,\overset{(0)}{\nabla}_XZ-\frac{1}{2}K(X,Z))\\
&-\overset{(0)}{\nabla}_{[X,Y]}Z+\frac{1}{2}K([X,Y],Z)\\
&=\overset{(0)}{R}(X,Y)Z-\frac{1}{2}\overset{(0)}{\nabla}_X(K(Y,Z))-\frac{1}{2}K(X,\overset{(0)}{\nabla}_YZ)+\frac{1}{4}K(X,K(Y,Z))\\
&+\frac{1}{2}\overset{(0)}{\nabla}_Y(K(X,Z))+\frac{1}{2}K(Y,\overset{(0)}{\nabla}_XZ)-\frac{1}{4}K(Y,K(X,Z))+\frac{1}{2}K([X,Y],Z).
\end{split}
\end{equation*}

Taking into account that
$$
(\overset{(0)}{\nabla}_XK)(Y,Z)=\overset{(0)}{\nabla}_X(K(Y,Z))-K(\overset{(0)}{\nabla}_XY,Z)-K(Y,\overset{(0)}{\nabla}_XZ)
$$
and
$$
(\overset{(0)}{\nabla}_YK)(X,Z)=\overset{(0)}{\nabla}_Y(K(X,Z))-K(\overset{(0)}{\nabla}_YX,Z)-K(X,\overset{(0)}{\nabla}_YZ)
$$
it results
\begin{equation}\label{eq_App_A1}
\begin{split}
R(X,Y)Z&=\overset{(0)}{R}(X,Y)Z-\frac{1}{2}(\overset{(0)}{\nabla}_XK)(Y,Z)-\frac{1}{2}K(\overset{(0)}{\nabla}_XY,Z)\\
&+\frac{1}{2}(\overset{(0)}{\nabla}_YK)(X,Z)+\frac{1}{2}K(\overset{(0)}{\nabla}_YX,Z)\\
&+\frac{1}{4}[K(X,K(Y,Z))-K(Y,K(X,Z))]+\frac{1}{2}K([X,Y],Z)\\
&=\overset{(0)}{R}(X,Y)Z-\frac{1}{2}\{(\overset{(0)}{\nabla}_XK)(Y,Z)-(\overset{(0)}{\nabla}_YK)(X,Z)\}\\
&-\frac{1}{2}K(\overset{(0)}{\nabla}_XY-\overset{(0)}{\nabla}_YX-[X,Y],Z)+\frac{1}{4}[K_X,K_Y]Z\\
&=\overset{(0)}{R}(X,Y)Z-\frac{1}{2}\{(\overset{(0)}{\nabla}_XK)(Y,Z)-(\overset{(0)}{\nabla}_YK)(X,Z)-\frac{1}{2}[K_X,K_Y]Z\}\\
&-\frac{1}{2}K(\overset{(0)}{T}(X,Y),Z),
\end{split}
\end{equation}
where we used $[K_X,K_Y]Z=K_X(K(Y,Z))-K_Y(K(X,Z))$, hence \underline{the first part of (i) is proved}.

\bigskip

Likewise, by using now
$$
\overset{(0)}{\nabla}_XY=\nabla_XY+\frac{1}{2}K(X,Y),
$$
we obtain
\begin{equation*}
\begin{split}
{(\overset{(0)}{\nabla}_XK)(Y,Z)}&=\overset{(0)}{\nabla}_X(K(Y,Z))-K(\overset{(0)}{\nabla}_XY,Z)-K(Y,\overset{(0)}{\nabla}_XZ)\\
&=\nabla_X(K(Y,Z))+\frac{1}{2}K(X,K(Y,Z))\\
&-K(\nabla_XY+\frac{1}{2}K(X,Y),Z)-K(Y,\nabla_XZ+\frac{1}{2}K(X,Z))\\
&=(\nabla_XK)(Y,Z)+\frac{1}{2}\{
K(X,K(Y,Z))-K(K(X,Y),Z)-K(Y,K(X,Z))
\}\\
&=(\nabla_XK)(Y,Z)+\frac{1}{2}[K_X,K_Y]Z-\frac{1}{2}K(K(X,Y),Z),
\end{split}
\end{equation*}
i.e. we have obtained
\begin{equation}\label{eq_App_A2}
(\overset{(0)}{\nabla}_XK)(Y,Z)=(\nabla_XK)(Y,Z)+\frac{1}{2}[K_X,K_Y]Z-\frac{1}{2}K(K(X,Y),Z),
\end{equation}
using this formula, we observe that
\begin{equation*}
\begin{split}
&(\overset{(0)}{\nabla}_XK)(Y,Z)-(\overset{(0)}{\nabla}_YK)(X,Z)-\frac{1}{2}[K_X,K_Y]Z\\
&=(\nabla_XK)(Y,Z)+\cancel{\frac{1}{2}[K_X,K_Y]Z}-\frac{1}{2}K(K(X,Y),Z)\\
&-(\nabla_YK)(X,Z)-\frac{1}{2}[K_Y,K_X]Z+\frac{1}{2}K(K(Y,X),Z)-\cancel{\frac{1}{2}[K_X,K_Y]Z}\\
&=(\nabla_XK)(Y,Z)-(\nabla_YK)(X,Z)-\frac{1}{2}K(K(X,Y)-K(Y,X),Z)+\frac{1}{2}[K_X,K_Y]Z\\
&=(\nabla_XK)(Y,Z)-(\nabla_YK)(X,Z)+\frac{1}{2}[K_X,K_Y]Z-\frac{1}{2}K(\overset{\ast}{T}(X,Y)-T(X,Y),Z),
\end{split}
\end{equation*}
where we have used Proposition \ref{prop:K diff vs T diff}.

From substituting this last formula in \eqref{eq_App_A1} it results
\begin{equation*}
\begin{split}
R(X,Y)Z&=\overset{(0)}{R}(X,Y)Z\\
&-\frac{1}{2}\left\{
(\nabla_XK)(Y,Z)-(\nabla_YK)(X,Z)+\frac{1}{2}[K_X,K_Y]Z
-\frac{1}{2}K(\overset{\ast}{T}(X,Y)-T(X,Y),Z)
\right\}\\
&-\frac{1}{2}K(\overset{(0)}{T}(X,Y),Z)\\
&=\overset{(0)}{R}(X,Y)Z
-\frac{1}{2}\left\{
(\nabla_XK)(Y,Z)-(\nabla_YK)(X,Z)+\frac{1}{2}[K_X,K_Y]Z
\right\}\\
&+\frac{1}{4}K(\overset{\ast}{T}(X,Y)-T(X,Y)-2\overset{(0)}{T}(X,Y),Z)\\
&=\overset{(0)}{R}(X,Y)Z
-\frac{1}{2}\left\{
(\nabla_XK)(Y,Z)-(\nabla_YK)(X,Z)+\frac{1}{2}[K_X,K_Y]Z
\right\}\\
&-\frac{1}{2}K(T(X,Y),Z),
\end{split}
\end{equation*}
where we use that $\overset{(0)}{T}=\frac{1}{2}(T+\overset{\ast}{T})$, hence \underline{(i) is now completely proved}.

\bigskip

Next, using
$$
\overset{\ast}{\nabla}_XY=\overset{(0)}{\nabla}_XY+\frac{1}{2}K(X,Y),
$$
we have
\begin{equation}
\begin{split}\label{eq_App_A3}
\overset{\ast}{R}(X,Y)Z&=\overset{\ast}{\nabla}_X\overset{\ast}{\nabla}_YZ-\overset{\ast}{\nabla}_Y\overset{\ast}{\nabla}_XZ-\overset{\ast}{\nabla}_{[X,Y]}Z\\
&=\overset{\ast}{\nabla}_X(\overset{(0)}{\nabla}_YZ+\frac{1}{2}K(Y,Z))-\overset{\ast}{\nabla}_Y(\overset{(0)}{\nabla}_XZ+\frac{1}{2}K(X,Z))\\
&-\overset{(0)}{\nabla}_{[X,Y]}Z-\frac{1}{2}K([X,Y],Z)\\
&=\overset{(0)}{\nabla}_X(\overset{(0)}{\nabla}_YZ+\frac{1}{2}K(Y,Z))+\frac{1}{2}K(X,\overset{(0)}{\nabla}_YZ+\frac{1}{2}K(Y,Z))\\
&-\overset{(0)}{\nabla}_Y(\overset{(0)}{\nabla}_XZ+\frac{1}{2}K(X,Z))-\frac{1}{2}K(Y,\overset{(0)}{\nabla}_XZ+\frac{1}{2}K(X,Z))\\
&-\overset{(0)}{\nabla}_{[X,Y]}Z-\frac{1}{2}K([X,Y],Z)\\
&=\overset{(0)}{R}(X,Y)Z
+\frac{1}{2}\overset{(0)}{\nabla}_X(K(Y,Z))+\frac{1}{2}K(X,\overset{(0)}{\nabla}_YZ)+\frac{1}{4}K(X,K(Y,Z))\\
&-\frac{1}{2}\overset{(0)}{\nabla}_Y(K(X,Z))-\frac{1}{2}K(Y,\overset{(0)}{\nabla}_XZ)-\frac{1}{4}K(Y,K(X,Z))\\
&-\frac{1}{2}K([X,Y],Z)\\
&=\overset{(0)}{R}(X,Y)Z+\frac{1}{2}(\overset{(0)}{\nabla}_XK)(Y,Z)+\frac{1}{2}K(\overset{(0)}{\nabla}_XY,Z)\\
&-\frac{1}{2}(\overset{(0)}{\nabla}_YK)(X,Z)-\frac{1}{2}K(\overset{(0)}{\nabla}_YX,Z)\\
&+\frac{1}{4}[K(X,K(Y,Z))-K(Y,K(X,Z))]-\frac{1}{2}K([X,Y],Z)\\
&=\overset{(0)}{R}(X,Y)Z+\frac{1}{2}
\{
(\overset{(0)}{\nabla}_XK)(Y,Z)-(\overset{(0)}{\nabla}_YK)(X,Z)
\}\\
&+\frac{1}{4}[K(X,K(Y,Z))-K(Y,K(X,Z))]
{+\frac{1}{2}K(\overset{(0)}{\nabla}_XY-\overset{(0)}{\nabla}_YX-[X,Y],Z)}\\
&=\overset{(0)}{R}(X,Y)Z+\frac{1}{2}
\{
(\overset{(0)}{\nabla}_XK)(Y,Z)-(\overset{(0)}{\nabla}_YK)(X,Z)+\frac{1}{2}[K_X,K_Y]Z
\}\\
&{+}\frac{1}{2}K(\overset{(0)}{T}(X,Y),Z),
\end{split}
\end{equation}
hence \underline{the first part of (ii) is proved}.

Next, by using again \eqref{eq_App_A2} we have
\begin{equation*}
\begin{split}
&(\overset{(0)}{\nabla}_XK)(Y,Z)-(\overset{(0)}{\nabla}_YK)(X,Z)+\frac{1}{2}[K_X,K_Y]Z\\
&=(\nabla_XK)(Y,Z)+\frac{1}{2}[K_X,K_Y]Z-\frac{1}{2}K(K(X,Y),Z)\\
&-(\nabla_YK)(X,Z)-\frac{1}{2}[K_Y,K_X]Z+\frac{1}{2}K(K(Y,X),Z)+\frac{1}{2}[K_X,K_Y],Z\\
&=(\nabla_XK)(Y,Z)-(\nabla_YK)(X,Z)-\frac{1}{2}K(K(X,Y)-K(Y,X),Z)+\frac{3}{2}[K_X,K_Y]Z\\
&=(\nabla_XK)(Y,Z)-(\nabla_YK)(X,Z)-\frac{1}{2}K(\overset{\ast}{T}(X,Y)-T(X,Y),Z)+\frac{3}{2}[K_X,K_Y]Z.
\end{split}
\end{equation*}

From \eqref{eq_App_A3}
\begin{equation*}
\begin{split}
\overset{\ast}{R}(X,Y)Z&=\overset{(0)}{R}(X,Y)Z+\frac{1}{2}\Bigg\{
(\nabla_XK)(Y,Z)-(\nabla_YK)(X,Z)
-\frac{1}{2}K(\overset{\ast}{T}(X,Y)-T(X,Y),Z)\\
&+\frac{3}{2}[K_X,K_Y]Z
\Bigg\}
+\frac{1}{2}K(\overset{(0)}{T}(X,Y),Z)\\
&=\overset{(0)}{R}(X,Y)Z+\frac{1}{2}\left\{
(\nabla_XK)(Y,Z)-(\nabla_YK)(X,Z)+\frac{3}{2}[K_X,K_Y]Z
\right\}\\
&-\frac{1}{4}K(\overset{\ast}{T}(X,Y)-T(X,Y)-2\overset{(0)}{T}(X,Y),Z)\\
&=\overset{(0)}{R}(X,Y)Z+\frac{1}{2}\left\{
(\nabla_XK)(Y,Z)-(\nabla_YK)(X,Z)+\frac{3}{2}[K_X,K_Y]Z
\right\}\\
&+\frac{1}{2}K(T(X,Y),Z),
\end{split}
\end{equation*}
hence \underline{(ii) is completely proved}.

\bigskip

Moreover, by subtracting the first equalities in (i), (ii) we get
\begin{equation*}
\begin{split}
&R(X,Y)Z-\overset{\ast}{R}(X,Y)Z\\
&=
-\frac{1}{2}\{
(\overset{(0)}{\nabla}_XK)(Y,Z)-(\overset{(0)}{\nabla}_YK)(X,Z)-\cancel{\frac{1}{2}[K_X,K_Y]Z}-\frac{1}{2}K(\overset{(0)}{T}(X,Y),Z)
\}\\
&-\frac{1}{2}\{(\overset{(0)}{\nabla}_XK)(Y,Z)-(\overset{(0)}{\nabla}_YK)(X,Z)+\cancel{\frac{1}{2}[K_X,K_Y]}Z\}-\frac{1}{2}K(\overset{(0)}{T}(X,Y),Z)\\
&=-\{(\overset{(0)}{\nabla}_XK)(Y,Z)-(\overset{(0)}{\nabla}_YK)(X,Z)\}-K(\overset{(0)}{T}(X,Y),Z),
\end{split}
\end{equation*}
hence \underline{the first part of (iii) is proved}.

\bigskip

By using equalities in (i), (ii) we get
\begin{equation*}
\begin{split}
&R(X,Y)Z-\overset{\ast}{R}(X,Y)Z\\
&=-\frac{1}{2}
\{
(\nabla_XK)(Y,Z)-(\nabla_YK)(X,Z)+\frac{1}{2}[K_X,K_Y]Z
\}
-\frac{1}{2}K(T(X,Y),Z)\\
&-\frac{1}{2}
\{
(\nabla_XK)(Y,Z)-(\nabla_YK)(X,Z)+\frac{3}{2}[K_X,K_Y]Z
\}
-\frac{1}{2}K(T(X,Y),Z)\\
&=-\{
(\nabla_XK)(Y,Z)-(\nabla_YK)(X,Z)+[K_X,K_Y]Z
\}-K(T(X,Y),Z),
\end{split}
\end{equation*}
i.e. \underline{the second part of (iii) is proved}.

\bigskip

Next, we add these formulas and obtain
\begin{equation*}
\begin{split}
R(X,Y)Z+\overset{\ast}{R}(X,Y)Z&=2\overset{(0)}{R}(X,Y)Z-\frac{1}{2}
\{
\cancel{(\overset{(0)}{\nabla}_XK)(Y,Z)}-\cancel{(\overset{(0)}{\nabla}_YK)(X,Z)}-\frac{1}{2}[K_X,K_Y]Z
\}\\
&-\cancel{\frac{1}{2}K(\overset{(0)}{T}(X,Y),Z)}+\frac{1}{2}\{
\cancel{(\overset{(0)}{\nabla}_XK)(Y,Z)}
-\cancel{(\overset{(0)}{\nabla}_YK)(X,Z)}
+\frac{1}{2}[K_X,K_Y]Z
\}\\
&+\cancel{\frac{1}{2}K(\overset{(0)}{T}(X,Y),Z)}\\
&=2\overset{(0)}{R}(X,Y)Z+\frac{1}{2}[K_X,K_Y]Z,
\end{split}
\end{equation*}
and
\begin{equation*}
\begin{split}
R(X,Y)Z+\overset{\ast}{R}(X,Y)Z&=2\overset{(0)}{R}(X,Y)Z\\
&-\frac{1}{2}\{
\cancel{(\nabla_XK)(Y,Z)}-\cancel{(\nabla_YK)(X,Z)}+\frac{1}{2}[K_X,K_Y]
\}-\cancel{\frac{1}{2}K(T(X,Y),Z)}\\
&+\frac{1}{2}\{
\cancel{(\nabla_XK)(Y,Z)}-\cancel{(\nabla_YK)(X,Z)}+\frac{3}{2}[K_X,K_Y]
\}+\cancel{\frac{1}{2}K(T(X,Y),Z)}\\
&=2\overset{(0)}{R}(X,Y)Z+\frac{1}{2}[K_X,K_Y]Z,
\end{split}
\end{equation*}
i.e. we get the same formula as previously, hence (iv) is also completely proved.

\subsection{The proof of Corollary \ref{rem_C2}}

We take $X,Y,Z$ as $\dfrac{\partial}{\partial x^j},\dfrac{\partial}{\partial x^i},\dfrac{\partial}{\partial x^m}$ and rewrite (i) in Theorem \ref{thm_C1} as
\begin{equation*}
\begin{split}
R\left(\frac{\partial}{\partial x^j},\frac{\partial}{\partial x^i}\right)\frac{\partial}{\partial x^m}&=\overset{(0)}{R}\left(\frac{\partial}{\partial x^j},\frac{\partial}{\partial x^i}\right)\frac{\partial}{\partial x^m}\\
&-\frac{1}{2}\Bigg\{
\left(\overset{(0)}{\nabla}_{\dfrac{\partial}{\partial x^j}}{K}\left(\frac{\partial}{\partial x^i},\frac{\partial}{\partial x^m}\right)\right)
-\left(\overset{(0)}{\nabla}_{\dfrac{\partial}{\partial x^i}}{K}\left(\frac{\partial}{\partial x^j},\frac{\partial}{\partial x^m}\right)\right)\\
&-\frac{1}{2}\left[
K_{\dfrac{\partial}{\partial x^j}}\left(K_{\dfrac{\partial}{\partial x^i}}\frac{\partial}{\partial x^m}\right)
-K_{\dfrac{\partial}{\partial x^i}}\left(K_{\dfrac{\partial}{\partial x^j}}\frac{\partial}{\partial x^m}\right)
\right]
\Bigg\}\\
&-\frac{1}{2}{K}\left(\overset{(0)}{T}\left(\frac{\partial}{\partial x^j},\frac{\partial}{\partial x^i}\right),\frac{\partial}{\partial x^m}\right),
\end{split}
\end{equation*}
that is
\begin{equation*}
\begin{split}
R^{\ k}_{m\ ji}&=\overset{(0)}{R}\,^{\ k}_{m\ ji}-\frac{1}{2}(K^{\ k}_{m\ i\overset{(0)}{|}j}-K^{\ k}_{m\ j\overset{(0)}{|}i})\\
&+\frac{1}{4}(K^l_{\ mi}K^k_{\ lj}-K^l_{\ mj}K^k_{\ li})
-\frac{1}{2}\overset{(0)}{T}\,^l_{\ ji}K^k_{\ ml}.
\end{split}
\end{equation*}

Likewise, we have
\begin{equation*}
\begin{split}
R\left(\frac{\partial}{\partial x^j},\frac{\partial}{\partial x^i}\right)\frac{\partial}{\partial x^m}&=\overset{(0)}{R}\left(\frac{\partial}{\partial x^j},\frac{\partial}{\partial x^i}\right)\frac{\partial}{\partial x^m}\\
&-\frac{1}{2}\left\{
\left(\nabla_{\dfrac{\partial}{\partial x^j}}K\right)\left(\frac{\partial}{\partial x^i},\frac{\partial}{\partial x^m}\right)
-\left(\nabla_{\dfrac{\partial}{\partial x^i}}K\right)\left(\frac{\partial}{\partial x^j},\frac{\partial}{\partial x^m}\right)
\right\}\\
&-\frac{1}{4}\left\{
K_{\dfrac{\partial}{\partial x^j}}\left(K_{\dfrac{\partial}{\partial x^i}}\frac{\partial}{\partial x^m}\right)
-K_{\dfrac{\partial}{\partial x^i}}\left(K_{\dfrac{\partial}{\partial x^j}}\frac{\partial}{\partial x^m}\right)
\right\}\\
&-\frac{1}{2}K\left(T\left(\frac{\partial}{\partial x^j},\frac{\partial}{\partial x^i}\right),\frac{\partial}{\partial x^m}\right),
\end{split}
\end{equation*}
i.e. in components
\begin{equation*}
\begin{split}
R^{\ k}_{m\ ji}&=\overset{(0)}{R}\,^{\ k}_{m\ ji}-\frac{1}{2}(K^{\ k}_{m\ i|j}-K^{\ k}_{m\ j|i})\\
&-\frac{1}{4}(K^l_{\ mi}K^k_{\ lj}-K^l_{\ mj}K^k_{\ li})
-\frac{1}{2}T^l_{\ ji}K^k_{\ ml},
\end{split}
\end{equation*}
hence (i) in {Corollary} \ref{rem_C2} is proved.

Likewise (ii), (iii) and (iv) follows from Theorem \ref{thm_C1} by substituting $X,Y,Z$ as above.

\end{document}